\documentclass[a4paper,10pt]{article}

\usepackage[left=1.2in,right=1.2in,top=1in,bottom=1.2in]{geometry}
\usepackage{authblk} % add author information
\usepackage[titletoc]{appendix} % add appendix
\usepackage{abstract}

\usepackage[T1]{fontenc} % font encoding 
\usepackage{lmodern} % latin modern fonts
\usepackage{mathrsfs} % extra fonts
\usepackage{xcolor} % add color
\usepackage{bm} % bold texts 
\usepackage{amsmath,amsfonts,amssymb,amsthm} % all ams packages
\usepackage{enumerate}
\usepackage{graphicx,float} % insert figures
\usepackage{caption,subcaption} % insert subfigures
\usepackage[colorlinks=true,linkcolor=blue,citecolor=blue,urlcolor=blue]{hyperref} % add hyperlinks
\usepackage{url} % add url
\usepackage[sort,authoryear,round]{natbib} % citation
\usepackage[linesnumbered,ruled,vlined]{algorithm2e} % add algorithms (MUST add before cleveref)
 
\SetCommentSty{mycommfont} % set algorithm comment style
\usepackage[capitalize,nameinlink,compress]{cleveref} 
\crefname{figure}{Figure}{Figures} 
\crefname{equation}{}{} 
\AtBeginEnvironment{appendices}{\crefalias{section}{appendix}} % cref settings
\usepackage{tikz-cd} % commutative diagrams
\usepackage{amscd} % commutative diagrams
\usepackage{tabularx} % add tables
\usepackage{blkarray,bigstrut} % add matrix decorations
\usepackage{derivative} % dot derivative odv

\newtheorem{theorem}{Theorem}[section]
\newtheorem{proposition}[theorem]{Proposition}
\newtheorem{lemma}[theorem]{Lemma}
\newtheorem{corollary}[theorem]{Corollary}
\newtheorem{example}[theorem]{Example}
\newtheorem{definition}[theorem]{Definition}
\newtheorem{remark}[theorem]{Remark}
\numberwithin{equation}{section}

\newcommand{\Div}{\mathrm{Div}}
\newcommand{\pdiv}{\mathrm{Prin}}
\newcommand{\Trop}{\mathrm{Trop}}

\newcommand{\homz}{\mathrm{Hom}}

\newcommand{\jac}{\mathrm{Jac}}
\newcommand{\alb}{\mathrm{Alb}}
\newcommand{\svp}{\mathrm{SVP}}
\newcommand{\svpz}{\mathrm{SVP}_\zeta}
\newcommand{\cvp}{\mathrm{CVP}}
\newcommand{\cvpz}{\mathrm{CVP}_\zeta}
\newcommand{\ajmap}{{\mathcal{J}_q}}
\newcommand{\matC}{\mathbf{C}}
\newcommand{\matY}{\mathbf{Y}}
\newcommand{\matL}{\mathbf{L}}
\newcommand{\matV}{\mathbf{V}}
\newcommand{\matA}{\mathbf{A}}
\newcommand{\matB}{\mathbf{B}}
\newcommand{\matQ}{\mathbf{Q}}
\newcommand{\matP}{\mathbf{P}}
\newcommand{\matW}{\mathbf{W}}
\newcommand{\matU}{\mathbf{U}}
\newcommand{\matD}{\mathbf{D}}
\newcommand{\matM}{\mathbf{M}}
\newcommand{\veca}{\mathbf{a}}
\newcommand{\vecx}{\mathbf{x}}
\newcommand{\vecy}{\mathbf{y}}
\newcommand{\vect}{\mathbf{t}}

\newcommand{\btau}{\bm{\tau}}
\newcommand{\vecn}{\mathbf{n}}
\newcommand{\vecb}{\mathbf{b}}
\newcommand*\diff{\mathop{}\!\mathrm{d}}

\def\diag{{\rm diag}}

\def\spacingset#1{\renewcommand{\baselinestretch}%
{#1}\small\normalsize} \spacingset{1}
\makeatletter
\def\@title{
  \begin{center}
    \fontsize{12}{14.4}\selectfont\bfseries\MakeUppercase \@titletext
  \end{center}
}
\newcommand{\titletext}[1]{\gdef\@titletext{#1}}
\makeatother

\titletext{Computing the Tropical Abel--Jacobi Transform and Tropical Distances for Metric Graphs}
\author[1]{Yueqi Cao\thanks{\url{yueqic@kth.se}}} % use [1,2] to separate contacts
\author[2]{Anthea Monod\thanks{\url{a.monod@imperial.ac.uk}}}
\affil[1]{Department of Mathematics, KTH Royal Institute of Technology}
\affil[2]{Department of Mathematics, Imperial College London}
\date{}

\begin{document}

\maketitle

%%%%%%%%%%%%%%%%%%%%%%%%%%%%%%%%%%%%%%%%%%%%%%%%%%%

\begin{abstract}
Metric graphs are important models for capturing the structure of complex data across various domains. While much effort has been devoted to extracting geometric and topological features from graph data, computational aspects of metric graphs as abstract tropical curves remains unexplored. In this paper, we present the first computational and machine learning-driven study of metric graphs from the perspective of tropical algebraic geometry. Specifically, we study the tropical Abel--Jacobi transform, a vectorization of points on a metric graph via the tropical Abel--Jacobi map into its associated flat torus, the tropical Jacobian. We develop algorithms to compute this transform and investigate how the resulting embeddings depend on different combinatorial models of the same metric graph.

Once embedded, we compute pairwise distances between points in the tropical Jacobian under two natural metrics: the tropical polarization distance and the Foster--Zhang distance. Computing these distances are generally NP-hard as they turn out to be linked to classical lattice problems in computational complexity, however, we identify a class of metric graphs where fast and explicit computations are feasible. For the general case, we propose practical algorithms for both exact and approximate distance matrix computations using lattice basis reduction and mixed-integer programming solvers. Our work lays the groundwork for future applications of tropical geometry and the tropical Abel--Jacobi transform in machine learning and data analysis.

\medskip
\noindent \textbf{Keywords.} Metric graphs; tropical curves; tropical Abel--Jacobi map; tropical Jacobian; tropical polarization; lattice problems; lattice reduction; computational complexity 
\end{abstract}

\newpage
\tableofcontents

%%%%%%%%%%%%%%%%%%%%%%%%%%%%%%%%%%%%%%%%%%%%%
\section{Introduction}\label{sec:intro}

Metric graphs are fundamental mathematical objects that enjoy both the combinatorial structure of graphs and the geometric structure of metric spaces. A metric graph is the geometric realization of its combinatorial model, which gives it richer and more intricate properties, but also makes it more challenging to work with computationally. Metric graphs are ubiquitous in both pure and applied mathematics, playing a key role in fields such as tropical algebraic geometry \citep{chan2021moduli}, algebraic number theory \citep{zhang1993admissible}, geometric group theory \citep{culler1986moduli}, and mathematical biology \citep{nicaise1985some}. Beyond mathematics, they are widely used to model complex real-world data and have applications in numerous scientific fields including quantum mechanics \citep{berkolaiko2017elementary}, machine learning \citep{ceschini2024graphs}, road network analysis \citep{thomson1995graph}, and medical imaging \citep{de2015graph}.

Extracting and representing the geometric and topological information from a metric graph is a key challenge in modern data analysis. A common approach is to represent graphs as vectors, which simplifies and speeds up subsequent computations and applications. This process of \emph{vectorizing} (the nodes or edges of) a graph is referred to as \emph{graph embedding} \citep{xu2021understanding}, or more broadly, \emph{graph representation learning} \citep{hamilton2020graph}. Given a graph $G$, the goal is to compute a map from $G$ to a target space $\mathcal{X}$, which may be a Euclidean space or any suitable manifold and is usually chosen for computational tractability and the applicability of statistical methods. However, in the case of metric graphs, current (combinatorial) graph embedding methods face three important limitations: (i) these maps are not deterministic, as they are optimized based on certain objective functions, making their mathematical characterization challenging; (ii) the choice of target space of embedding varies from case to case, lacking a unified framework; and (iii) existing methods are limited to combinatorial graphs, meaning they can only be applied to any finite set of nodes or edges of a metric graph, leading to inefficiencies in generating vectors from the full structure of the metric graph itself. 

Tropical geometry is a modern and rapidly developing field of mathematics which is inherently connected to many branches of pure and applied mathematics. It is known as the ``combinatorial shadow of algebraic geometry'' \citep{maclagan2021introduction}, which simplifies algebraic structures into piecewise linear forms. In tropical geometry, metric graphs are known to be abstract tropical curves---the tropical analogs of classical algebraic curves \citep{mikhalkin2007tropical,chan2021moduli}. Consequently, many geometric constructions on classical algebraic curves have analogs on metric graphs, making them crucial in translating complex geometric problems into a form that can be tackled using tropical techniques.

Despite the vast literature in graph representation learning, the identity of metric graphs as tropical curves has never been explored for computational, data analytic, and machine learning tasks---such an approach has been difficult  due to the lack of translation from abstract mathematics to computable quantities and applicable algorithms. In this paper, we show that the well-known \emph{tropical Abel--Jacobi map} in tropical geometry provides a natural approach to ``vectorize'' metric graphs. In particular, given a metric graph $\Gamma$, we show that there is a canonically associated flat torus to $\Gamma$---its \emph{tropical Jacobian}---which serves as the natural target embedding space $\mathcal{X}$.  The tropical Abel--Jacobi map serves as the vectorization map sending $\Gamma$ to the target space $\mathcal{X}$; we present algorithms to compute vector representations from any combinatorial model of $\Gamma$.  The piecewise linear nature of the tropical Abel--Jacobi map allows us to directly sample vectors from its image in $\mathcal{X}$, making it possible to efficiently sample point clouds of different sizes. 

Furthermore, we show that there are distance functions defined on the target space $\mathcal{X}$ which are compatible with its tropical structure. We give a complete characterization of the class of metric graphs where the \emph{tropical polarization distance}---an $\ell_2$ distance on the tropical Jacobian---can be computed in closed form. On a generic flat torus, the computation of distances is closely related to the \emph{closest vector problem} (CVP) in computational complexity theory and cryptology, which is known to be NP-hard. We first demonstrate that in low dimensions, exact computations are feasible using existing CVP solvers and mixed integer programming (MIP) solvers. In high dimensions, we propose an algorithm to compute local distances, which is sufficient for many practical applications in real-world problems. \Cref{fig:framework-illustrate} presents an illustration of our framework. Our work presents the first foundations to apply the tropical transformation of metric graphs to problems in other fields.

% add a figure to illustrate
\begin{figure}[h]
    \centering
    \begin{subfigure}{0.32\linewidth}
        \includegraphics[width=\linewidth]{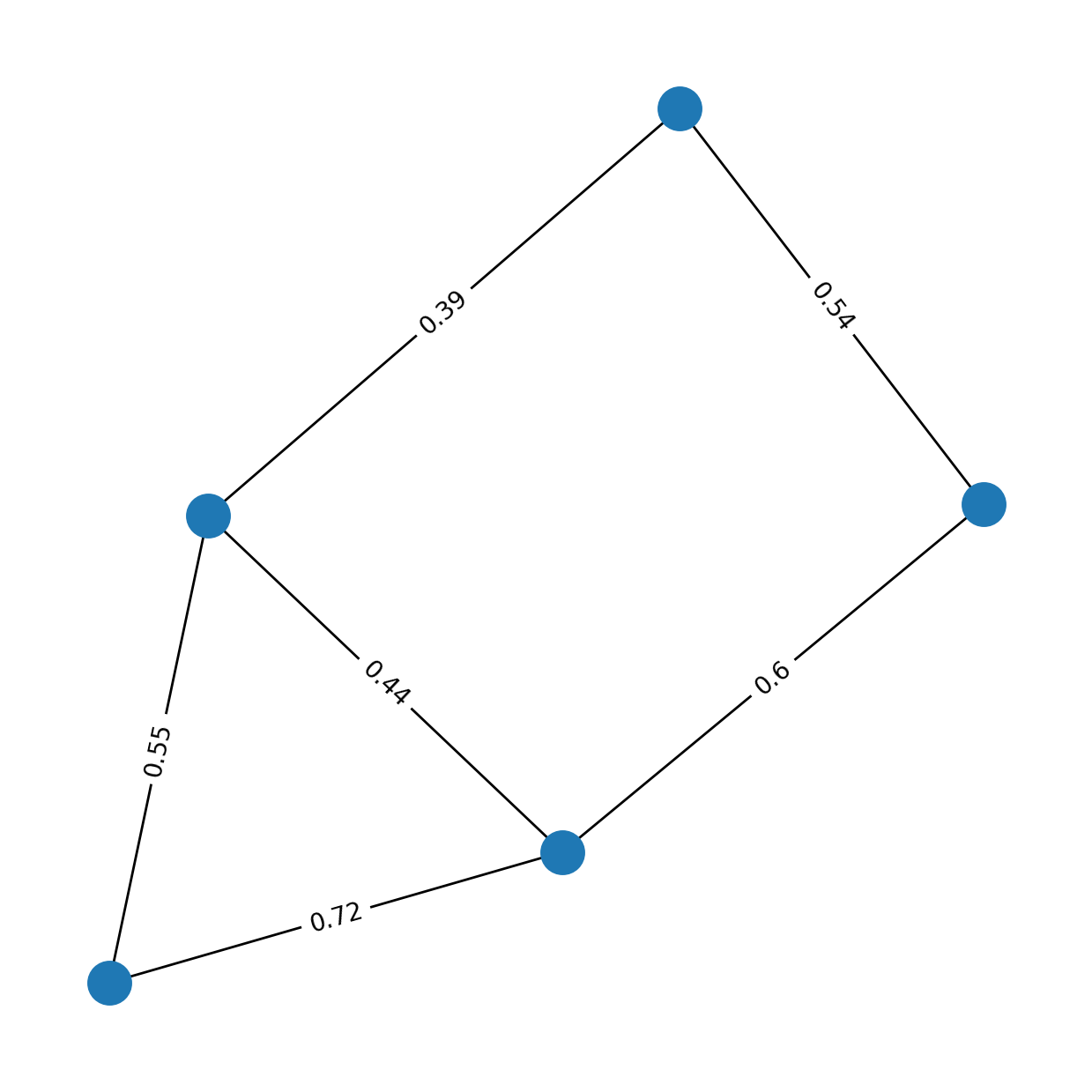}
        \caption{}
    \end{subfigure}
    \begin{subfigure}{0.32\linewidth}
        \includegraphics[width=\linewidth]{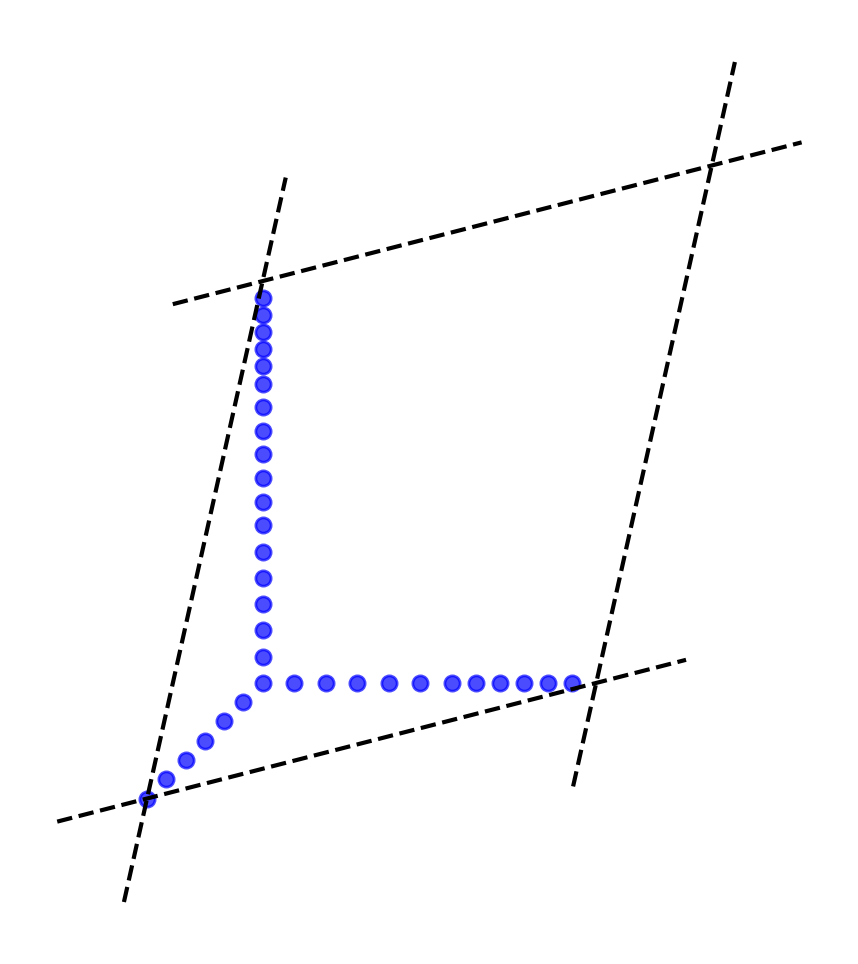}
        \caption{}
    \end{subfigure}
    \begin{subfigure}{0.32\linewidth}
        \includegraphics[width=0.8\linewidth]{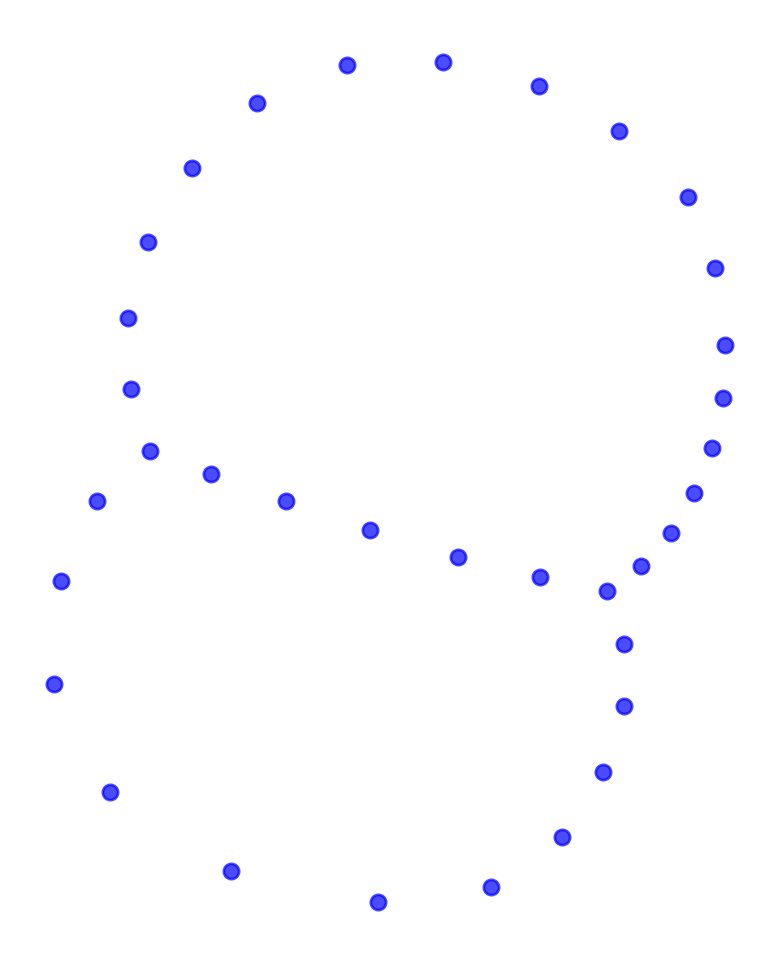}
        \caption{}
    \end{subfigure}
    \caption{An illustrative example of our proposed framework. (a) We begin with a combinatorial model $G$ of a random metric graph $\Gamma$. (b) We compute the tropical Abel--Jacobi transform of $G$ and sample additional points via interpolation. The dotted lines represent the fundamental domain of the tropical Jacobian and the blue points correspond to the tropical Abel--Jacobi transform of $\Gamma$. (c) We compute the pairwise distance matrix for the point cloud data sampled from the tropical Jacobian. The distance matrix is visualized via multidimensional scaling (MDS).}
    \label{fig:framework-illustrate}
\end{figure}

\paragraph{Contributions.}

In this paper, we focus on the computation of the tropical Abel--Jacobi transform of metric graphs and the associated distance functions on the tropical Jacobian. Our specific contributions are summarized as follows:  

\begin{itemize}
    \item We propose an algorithm to compute the vector representations of the tropical Abel--Jacobi transform of a metric graph, based on the construction of the cycle--edge incidence matrix and path--edge incidence matrix from its combinatorial model. We further propose another sampling algorithm to efficiently sample additional points through interpolation based on the piecewise linear nature of the tropical Abel--Jacobi transform. We prove the correctness of both algorithms and demonstrate the computational advantage over working with refined combinatorial models.
    \item We provide a comprehensive analysis of the computational complexity of our main algorithm, along with a full characterization of how the output vector representations change with respect to the choice of different combinatorial models. We prove key properties of the tropical Abel--Jacobi transform from a computational and algorithmic perspective, which in turn allows us to develop algorithms to simplify the combinatorial models and enhance computational efficiency.
    \item We compute the tropical polarization and the Foster--Zhang distance functions, which are associated distance functions of the Riemannian metric and the Finsler metric proposed in \cite{baker2011metric}. We show that the computation of distances on the tropical Jacobian is equivalent to solving the CVP in computational complexity theory and cryptology. We fully characterize the special case where an explicit solution exists and elicit the NP-hardness of computing and approximating pairwise distances for general metric graphs and their tropical Jacobians. 
    \item For low dimensional tropical Jacobians, we show that exact computations of tropical distances are feasible with existing CVP and MIP solvers. In high dimensions, we propose the \emph{truncated tropical polarization distance} and present an algorithm for its computation based on lattice basis reduction. We give theoretical guarantees that, below a certain threshold, the truncated polarization distance matrix preserves correct local distances. We implement all algorithms in this paper and test efficiency and accuracy via numerical experiments.
\end{itemize}

\paragraph{Outline.}  

The remainder of this paper is organized as follows. In \Cref{sec:curve-graph}, we introduce the background on metric graphs and tropical curves, and discuss how metric graphs arise in the study of tropical geometry. In \Cref{sec:trop-geom-graph}, we illustrate tropical geometric constructions on metric graphs. In particular, we introduce key concepts related to the tropical Abel--Jacobi map and prove necessary results for our study. 

\Cref{sec:alg} discusses the computation of the tropical Abel--Jacobi transform in detail.  In \Cref{subsec:main-algorithm}, we introduce the construction of cycle--edge incidence matrix and path--edge incidence matrix for any combinatorial model of a metric graph. We present the main algorithm to compute the tropical Abel--Jacobi transform and analyze its computational complexity. In \Cref{subsec:change-vectors}, we show how the vector representations change with respect to the choice of different combinatorial models in our algorithm. In \Cref{subsec:property}, we prove computational properties of the tropical Abel--Jacobi transform. 

\Cref{sec:metrics} discusses the computation of tropical distances on the tropical Jacobian. In \Cref{subsec:dist-func}, we introduce two distance functions on the tropical Jacobian which are compatible with the underlying tropical structure. In \Cref{subsec:np-hardness}, we discuss the connection between distance computation and classical lattice problems which are known to be NP-hard. In \Cref{subsec:truncated}, we introduce the truncated tropical polarization matrix and present an algorithm for its computation. We implement our algorithms and show results of simulation studies and numerical experiments in \Cref{subsec:simulation}. 

We close the paper with a discussion
on our work and propose directions for future research in \Cref{sec:discussion}. Technical proofs of some claims in the main content can be found in \Cref{app:proof}. We also include summaries of some related topics in complex geometry and tropical geometry in \Cref{app:theory}.    

%%%%%%%%%%%%%%%%%%%%%%%%%%%%%%%%%%%%%%%%%%%%%

\section{Metric Graphs as Tropical Curves}\label{sec:curve-graph}

In this section, we begin by introducing concepts related to metric graphs and their combinatorial models. We provide an overview of tropical curves, illustrating how metric graphs naturally arise in the field of tropical algebraic geometry. We then elaborate on the bijective correspondence between tropical structures and length metrics on a topological graph---a classical theorem in the field which lays the foundation of tropical geometric constructions on metric graphs. 

\begin{remark}
We begin with an important comment on vocabulary: In the machine learning literature, the term ``embedding'' is used quite freely to mean any map into another space.  From now on, we use the term ``embedding'' in the mathematical sense, as a map that is injective and structure-preserving; relevant to this work, a ``structure-preserving'' map means that it should be continuous in topological settings and homomorphic in algebraic settings.  

%\yc{It is very annoying that the machine learning community abuses the use of "embedding". A graph embedding in their context can be any map into another space. But in mathematics, an embedding should be injective and preserves math structure (continuous if used in topology, homomorphism if used in algebra. Unfortunately we are using "embedding" everywhere in this paper and sometimes we are talking about machine learning (e.g. in abstract) and sometimes we are talking about math (e.g. in section of tropical geometry of metric graphs)).  Shall we add a caveat somewhere?} 
\end{remark}

\subsection{Notions of Graphs}\label{sec:graph}

We clarify and unify terminologies from different fields for our use in the remainder of the paper.  

\begin{definition}\label{def:top-graph}
  A \emph{combinatorial graph} $G$ consists of a discrete set $V(G)$ of \emph{vertices} (or nodes) and a multiset $E(G)$ of ordered pairs of vertices called \emph{edges}. For each edge $e\in E(G)$, denote its initial vertex by $e_-$ and terminal vertex by $e_+$. The realization of $G$ is a topological space constructed by first taking a copy of the unit interval $[0,1]_e$ for each edge $e\in E(G)$, and then attaching $0$ to the initial vertex $e_-$ and $1$ to the terminal vertex $e_+$. Denote the resulting topological space by $|G|$.
  
  A \emph{topological graph} $\Gamma$ is a topological space that is homeomorphic to the  realization of some combinatorial graph. A combinatorial model for $\Gamma$ is a choice of combinatorial graph $G$ together with a homeomorphism $\psi:|G|\to\Gamma$. For each edge $e\in E(G)$, the image $\psi([0,1]_e)$ is called an \emph{edge} in $\Gamma$ and still denoted by $e$.
\end{definition}

Our \Cref{def:top-graph} presents a compromise between graph theory and algebraic topology: A combinatorial graph in graph theory is normally referred to as a \emph{directed multigraph} \citep{bollobas2013modern}, while a topological graph in graph theory is particularly defined as an embedding of a combinatorial graph to some surface \citep{gross2001topological}. A topological graph in algebraic topology is a 1-dimensional CW complex and a combinatorial model is equivalent to a \emph{cellular decomposition}, or \emph{cellulation}, of the 1-dimensional CW complex \citep{hatcher}. Our definition of a topological graph is a reinterpretation of the definition of a 1-dimensional CW complex. The main spirit here is to distinguish between topological graphs, which are ``continuous'' topological spaces, and combinatorial graphs, which are discrete (multi)sets.        

Unless stated otherwise, all topological and combinatorial graphs are assumed to be connected.

Let $\Gamma$ be a topological graph. Suppose $d$ is a metric on $\Gamma$. Let $\gamma:[0,1]\to \Gamma$ be a Lipschitz curve. The \emph{length} of curve $\gamma$ is defined as
\begin{equation*}
\ell(\gamma) = \sup\bigg\{\sum_{i=0}^{n-1}d(\gamma(t_i),\gamma(t_{i+1})): 0=t_0\le t_1\le\cdots\le t_n=1\bigg\} .   
\end{equation*}
The length function $\ell$ induces another metric on $\Gamma$ by
\begin{equation*}
d_\ell(p,q) = \inf \{\ell(\gamma): \gamma\text{ is a Lipschitz curve connecting }p\text{ and }q\} .    
\end{equation*}
The metric $d$ on $\Gamma$ is called a length metric if $d=d_\ell$.

\begin{definition}
    A \emph{metric graph} is a topological graph equipped with a length metric. 
\end{definition}

Alternatively, we can define metric graphs in a similar vein to \Cref{def:top-graph}: First, we define a \emph{weighted combinatorial graph} $G$ to be a combinatorial graph with a weight function $\ell:E(G)\to\mathbb{R}_+$. Then, a metric graph is obtained by gluing intervals $[0,\ell(e)]$ for edges $e\in E(G)$ and assigning the quotient metric. This construction approach is typically used for polyhedral spaces in metric geometry \citep{burago2001course}. 

Metric graphs are also known as \emph{metrized graphs} in other literature in mathematical biology, chemical engineering, and fluid dynamics, to name a few \citep{baker2006metrized}. Certain metric graphs with additional structures have specific names; for example, a metric graph with a self-adjoint differential operator is called a \emph{quantum graph} \citep{berkolaiko2017elementary}, while a metric graph with a piecewise $C^2$-embedding to some Euclidean space is called a \emph{$c^2$-network} \citep{von1985characteristic}. 

\subsection{Tropical Curves in Algebraic Geometry}

We now present and discuss aspects of curves in tropical algebraic geometry relevant to our study.

\paragraph{Tropical Plane Curves.}

An \emph{algebraic plane curve} is the vanishing set of a bivariate polynomial. Analogously, a \emph{tropical plane curve} is defined as the tropical vanishing set of a bivariate tropical polynomial: Let $(\mathbb{R}\cup\{\infty\},\oplus,\odot)$ be the tropical semiring where the basic arithmetic operations of addition and multiplication are redefined as 
\begin{equation*}
x\oplus y = \min\{x,y\},\quad x\odot y = x+y .   
\end{equation*}
A bivariate tropical polynomial is 
\begin{equation*}
f(x,y) = \bigoplus_{(i,j)\in I}c_{ij}\otimes x^i\otimes y^j = \min_{(i,j)\in I}\big\{c_{ij}+ix+jy\big\},   
\end{equation*}
where $I\subseteq \mathbb{Z}^2$ is a finite index set. Its tropical vanishing set $\mathcal{V}(f)$ is defined as the set of all points at which the minimum of $f$ is attained at least twice. Equivalently, a point $(x,y)\in\mathbb{R}^2$ is in $\mathcal{V}(f)$ if and only if the piecewise linear function $f$ is not linear or differentiable at $(x,y)$. A tropical plane curve is defined as the tropical vanishing set $\mathcal{V}(f)$ of some bivariate tropical polynomial $f$.

A tropical plane curve has the following characterizing structure: it is a graph embedded in $\mathbb{R}^2$, of which all edges have rational slopes. Moreover, it satisfies the 
\emph{balancing condition}: Let $f$ be a bivariate tropical polynomial with generic coefficients $c_{ij}$ such that whenever $c_{ij}+ix+jy=c_{rs}+rx+sy$ at a minimizer $(x,y)\in\mathcal{V}(f)$, the differences $i-r$ and $j-s$ are coprime. Let $p\in\mathcal{V}(f)$ be any point on the tropical plane curve. Suppose $p$ has $k$ incident edges, and $v_1,\ldots,v_k$ are outward primitive lattice vectors emanating from $p$, then $v_1+\ldots+v_k=0$ \citep[Proposition 1.3.1]{maclagan2021introduction}. 

More generally, we can define \emph{tropical space curves} embedded in Euclidean spaces of any dimension, known as 1-dimensional \emph{tropical varieties}. The balancing condition turns out to be one of the most fundamental characteristics of a tropical variety \citep[Section 3.3]{maclagan2021introduction}. We will revisit the balancing condition in the definition of tropical structures on metric graphs further on in \Cref{sec:trop-structure-graph}.

\paragraph{Tropical Curves from Non-Archimedean Geometry.}

Tropical curves arise as a type of \emph{degeneration} of algebraic curves over non-Archimedean fields. The degeneration process, known as \emph{tropicalization}, allows tropical curves to be defined intrinsically as metric graphs.  We will now briefly outline the appearance of metric graphs in the context of non-Archimedean geometry; a full discussion on tropicalization and non-Archimedean geometry is beyond the scope of this paper but full details can be found in \cite{payne2009analytification,baker2016nonarchimedean,chan2012tropical}.

Let $(K,\|\cdot\|)$ be a normed field. It satisfies the \emph{Archimedean axiom} if for any nonzero element $x\in K^*$, there is an integer $m\in\mathbb{Z}$ such that $\|mx\|>1$. The field is called non-Archimedean if it fails the axiom. Common examples of non-Archimedean fields are fields with nontrivial valuations \citep[Section 2.1]{maclagan2021introduction}. Suppose $\nu:K\to\mathbb{R}\cup\{\infty\}$ is a valuation on $K$. The valuation $\nu$ is called non-Archimedean if $K$ with the induced norm $\|x\|_\nu=\exp(-\nu(x))$ is non-Archimedean.

%For example, the Puiseux series $\mathbb{C}\{\!\{t\}\!\}$ defined by
% $$
% \{f(t) = a_1t^{q_1}+a_2t^{q_2}+\cdots\mid a_i\in\mathbb{C},q_i\in\mathbb{Q},q_i\text{'s have common denominator}\}
% $$
% with a valuation defined by $\nu:f(t)\mapsto q_1$.  Then $\mathbb{C}\{\!\{t\}\!\}$ is a non-Archimedean field.

Let $K$ be an algebraically closed field that is complete with respect to $\|\cdot\|_\nu$ for a nontrivial non-Archimedean valuation $\nu$. Suppose $X$ is an affine algebraic variety over $K$ and $\iota:X\hookrightarrow\mathbb{A}_K^n$ is an affine embedding of $X$. The \emph{extrinsic tropicalization} of $(X,\iota)$ is the closure of 
\begin{equation*}
\{(\nu(x_1),\ldots,\nu(x_n)): (x_1,\ldots,x_n)\in \iota(X)\subseteq \mathbb{A}^n_K\} ,
\end{equation*}
under the standard topology of $(\mathbb{R}\cup\{\infty\})^n$. Denote this set as $\mathrm{Trop}(X,\iota)$. In the special case where $X$ is the vanishing set of some bivariate Laurent polynomial $f\in K[x^\pm,y^\pm]$ in the algebraic torus $\mathbb{T}^2_K\cong (K^*)^2$, the extrinsic tropicalization of $X$ coincides with the tropical plane curve defined by $\mathcal{V}(\Trop(f))$ as a consequence of Kapranov's theorem \citep[Theorem 3.1.3]{maclagan2021introduction}.

Defining an intrinsic tropicalization of $X$ independent of the choice of embeddings entails considering its \emph{Berkovich analytification} $X^{\mathrm{an}}$, which is a space defined as the collection of all multiplicative seminorms $|\cdot|$ on the coordinate ring $K[X]$ which extend the norm $\|\cdot\|_\nu$ on $K$, and is endowed with the coarsest topology such that for every $f\in K[X]$, the evaluation map $|\cdot|\mapsto |f|$ is continuous \citep{baker2008introduction}. Now suppose an affine embedding $\iota:X\hookrightarrow\mathbb{A}_K^n$ is given by generators $f_1,\ldots,f_n\in K[X]$. There exists an associated continuous map  
\begin{equation*}
\begin{aligned}
\pi_\iota:X^{\mathrm{an}}&\to \mathrm{Trop}(X,\iota)\\
|\cdot| &\mapsto (-\log|f_1|,\cdots,-\log|f_n|) .
\end{aligned}
\end{equation*}
Let $\iota':X\hookrightarrow\mathbb{A}_K^{n'}$ be another embedding and $\phi:\mathbb{A}_K^n\to\mathbb{A}_K^{n'}$ be an equivariant morphism such that $\iota'=\phi\circ\iota$. Then there exists a map $\mathrm{Trop}(\phi):\mathrm{Trop}(X,\iota)\to\mathrm{Trop}(X,\iota')$ such that $\pi_{\iota'}=\mathrm{Trop}(\phi)\circ\pi_{\iota}$. Hence there is an induced map
\begin{equation*}
\varprojlim \pi_\iota : X^{\mathrm{an}}\to \varprojlim\mathrm{Trop}(X,\iota) ,
\end{equation*}
where the inverse limit is taken in the category of topological spaces over the system of all possible affine embeddings of $X$. \cite{payne2009analytification} proved that the map $\varprojlim \pi_\iota$ is a homeomorphism. In a certain sense, this means that the extrinsic tropicalization of $X$ is a snapshot of its Berkovich analytification $X^{\mathrm{an}}$.

When $X$ is an algebraic curve over $K$, the space $X^{\mathrm{an}}$ contains a metric graph $\Sigma(X)$ as its deformation retract, known as its \emph{minimal Berkovich skeleton}. Since the metric graph $\Sigma(X)$ captures the topological and combinatorial properties of $X^{\mathrm{an}}$, in tropical geometry, it makes sense to simply regard $\Sigma(X)$ as the intrinsic tropicalization of $X$ \citep{chan2012tropical,haase2012linear,caporaso2010torelli}. The metric graph $\Sigma(X)$ also measures the faithfulness of an extrinsic tropicalization: a tropicalization $\mathrm{Trop}(X,\iota)$ is called \emph{faithful} if $\Sigma(X)$ is mapped isometrically onto its image. Finding faithful tropicalizations for algebraic curves over non-Archimedean fields is a challenging problem in theoretical and computational algebraic geometry \citep{gubler2016skeletons,jell2020constructing,markwig2023faithful}.

%Further, if $X$ is a variety in the algebraic torus $\mathbb{T}^d_K\cong (K^*)^d$, then the extrinsic tropicalization of $X$ can be described as intersections of tropical hypersurfaces over tropical semirings \citep[Section 3.2]{maclagan2021introduction}.

\subsection{Tropical Structures and Length Metrics}\label{sec:trop-structure-graph}

The previous introduction inspires us to view a tropical curve as a topological graph with an additional ``tropical structure''. In this section, we seek to understand precisely what the tropical structure is. Since a metric graph is a topological graph with a length metric, the connection between tropical curves and metric graphs is essentially the connection between two structures on the same underlying topological graph. 

\begin{definition}
    Let $\Gamma$ be a compact topological graph. For each point $p\in\Gamma$, there is a neighborhood $U$ homeomorphic to a finite union of $k_p$ rays at $p$. The number $k_p$ is called the \emph{valence} of $p$. A compact topological graph is called \emph{regular} if all points have valence greater or equal to $2$.
\end{definition}

The regularity condition is also referred to as \emph{smoothness} for tropical curves \citep{gross2023tautological,mikhalkin2009tropical}.  In essence, it means a graph does not have leaf edges. In this paper, all graphs are assumed to be regular unless otherwise stated. The assumption will not affect our computation since we can always trim leaf edges of a graph without affecting its image under the tropical Abel--Jacobi map (see \Cref{sec:alg}).

% \begin{definition}
%     Let $\Gamma$ be a topological graph. The dimension of the 1-homology group $g=\dim H_1(\Gamma,\mathbb{R})$ is called the genus of $\Gamma$.
% \end{definition}

The following definition of tropical structure is adapted from \cite{ji2012complete}, where it is referred to as a \emph{smooth integral affine structure}. The original definition, which is formulated in a more general setting, can be found in \cite{mikhalkin2008tropical}.

\begin{definition}
   Let $\Gamma$ be a topological graph. A \emph{tropical atlas} is a collection of charts $\{(U_i,\phi_{i})\}_{i\in I}$ of $\Gamma$ such that
   \begin{enumerate}[(i)]
       \item For each point $p\in \Gamma$ of valence $k$, there is a chart $(U_i,\phi_{i})$ where $U_i$ is a neighborhood of $p$ and $\phi_{i}:U_i\to\mathbb{R}^{k-1}$ is an embedding whose image is the union of $k$ line segments with rational slopes joined at $\phi_i(p)$. Furthermore, let $v_{j}\in\mathbb{Z}^{k-1}$ be the primitive lattice vector in the direction of the $j$th line. Then the primitive lattice vectors satisfy the balancing condition that
       \begin{equation*}
       v_1+\ldots+v_k=0 ,
       \end{equation*}
       and the nondegeneracy condition that any $k-1$ vectors from $v_{1},\ldots,v_{k}$ form a basis of $\mathbb{Z}^{k-1}$; and
       \item Given two overlapping sets and embeddings $\phi_{i}:U_{i}\to\mathbb{R}^{k_1}$ and $\phi_{j}:U_{j}\to \mathbb{R}^{k_2}$, there is a common space $\mathbb{R}^N$ and inclusions $\iota_i:\mathbb{R}^{k_1} \hookrightarrow\mathbb{R}^N$, $\iota_j:\mathbb{R}^{k_2} \hookrightarrow\mathbb{R}^N$, and an integral affine map $\Phi_{ij}:\mathbb{R}^N\to\mathbb{R}^N$ such that
       \begin{equation*}
       \iota_i\circ\phi_{i}|_{U_i\cap U_j} = \Phi_{ij}\circ \iota_j\circ\phi_{j}|_{U_i\cap U_j} .
       \end{equation*}
       An integral affine map is of the form $\Phi_{ij}(w)=Aw+b$, where $A\in \mathrm{GL}(N;\mathbb{Z})$ is an invertible matrix with integer entries, and $b\in\mathbb{R}^N$ is a vector of real entries.
   \end{enumerate}
   Two tropical atlases are said to be \emph{equivalent} if their union is again a tropical atlas. A \emph{tropical structure} on $\Gamma$ is an equivalence class of tropical atlases.
\end{definition}

For each point $p$ of valence $2$, the first condition (i) implies that there is a neighborhood $U$ of $p$ and an embedding $\phi_U:U\to\mathbb{R}$ such that its image is an interval. Since the only primitive vectors in $\mathbb{R}$ are $\pm 1$, the balancing and nondegeneracy conditions are automatically satisfied. 

\begin{remark}
The manifold-like construction can also be generalized to higher dimensions and the resulting space is called a \emph{tropical manifold} \citep[Chapter 7]{mikhalkin2009tropical}. A tropical structure on $\Gamma$ makes it a 1-dimensional tropical manifold.
\end{remark}

The following theorem is well-known in the literature of tropical geometry. We reformulate and re-prove it in our setting.

\begin{figure}
    \centering
    \includegraphics[width=0.9\linewidth]{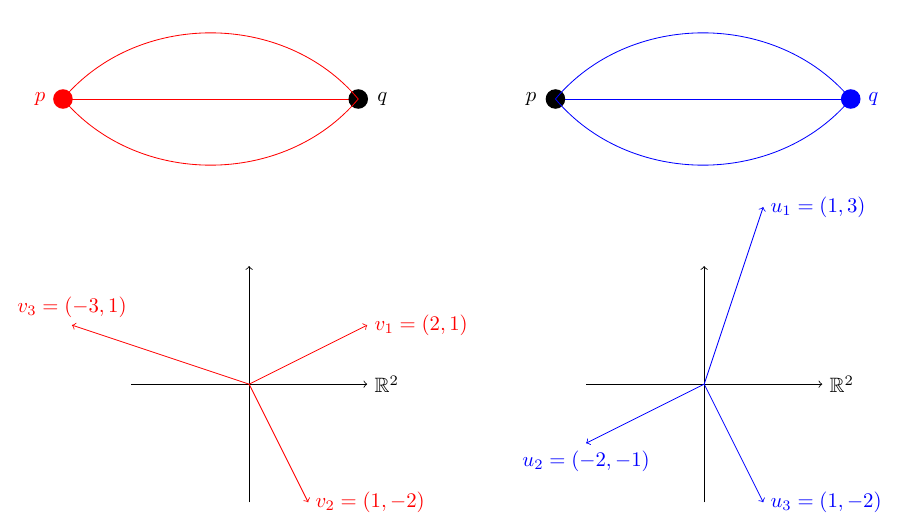}
    \caption{A tropical atlas corresponding to the length metric on $\Gamma$. The metric graph is constructed by joining three unit intervals at two end points $p$ and $q$. Red lines indicate the open neighborhood $U_p=\Gamma\backslash \{q\}$ and the primitive lattice vectors under the embedding $\phi_{U_p}$. Blue lines indicate the open neighborhood $U_q=\Gamma\backslash\{p\}$ and primitive lattice vectors under the embedding $\phi_{U_q}$.}
    \label{fig:Zaffine}
    \end{figure}

\begin{theorem}{\cite[Proposition 3.6]{mikhalkin2008tropical}}\label{thm:bijection}
    Let $\Gamma$ be a compact regular topological graph. Then there is a bijection from the set of tropical structures on $\Gamma$ to the set of length metrics on $\Gamma$.
\end{theorem}

\begin{proof}
Let $\mathcal{A}$ be a tropical structure. For any point $p\in\Gamma$, there is a local chart $\phi_U:U\to\mathbb{R}^{k-1}$. If $k=2$, define the metric on $U$ to be the standard metric on the image of $\phi_U$ as a subset of $\mathbb{R}$. If $k\ge 3$, let $v_1,\ldots,v_k$ be the primitive lattice vectors. We require that each primitive lattice vector have unit length and define the metric on $U$ to be the length metric on the image of $\phi_U$. Since primitive lattice vectors are preserved under integral affine transformations, the local metrics can be glued together and form a well-defined length metric $d_{\mathcal{A}}$ on $\Gamma$. 

Conversely, let $d$ be a length metric on $\Gamma$. Fix a combinatorial model for $\Gamma$. For each open edge $e$, we define $\phi_e:e\to\mathbb{R}$ to be the function mapping $e$ isometrically to an open interval. For each vertex $p$ of degree $k\ge 3$, let $e_1,\ldots,e_k$ be the edges incident to $p$. For each $e_i, 1\le i\le k-1$, we fix a primitive lattice vector $v_i=(0,\ldots,0,1,0,\ldots,0)\in\mathbb{R}^{k-1}$ whose $i$th coordinate is 1 and other coordinates are 0. Let $v_k = (-1,\ldots,-1)\in\mathbb{R}^{k-1}$ be the vector of all $-1$s for $e_k$. Let $U_\epsilon$ be the $\epsilon$-ball centered at $p$ such that $U_\epsilon\cap e_i\subseteq e_i$. If $q\in U_\epsilon \cap e_i$, we send $q$ to be $d(p,q)\cdot v_i$ in $\mathbb{R}^{k-1}$. This defines an embedding $\phi:U_\epsilon\to\mathbb{R}^{k-1}$. The charts are compatible by construction and form a tropical structure $\mathcal{A}_d$ on $\Gamma$. 

From the above constructions we can check that $\mathcal{A}_{d_{\mathcal{A}}}=\mathcal{A}$ and $d_{\mathcal{A}_d}=d$. Thus the map sending $\mathcal{A}$ to $d_{\mathcal{A}}$ and the map sending $d$ to $\mathcal{A}_d$ are inverses to each other, which gives the required bijections.
\end{proof}

% \begin{remark}
% Due to the one-to-one correspondence between tropical structures and length metrics, in the literature of tropical geometry (abstract) tropical curves are often defined as metric graphs \citep{chan2012tropical}.
% \end{remark}

\begin{example}
    Let $\Gamma$ be a metric graph constructed by joining three unit intervals at two end points $p$ and $q$; see \Cref{fig:Zaffine}. The metric graph can be covered by two open sets $U_p=\Gamma\backslash \{q\}$ and $U_q=\Gamma\backslash\{p\}$. The embeddings of  $U_p$ and $U_q$ are pictured in \Cref{fig:Zaffine}. On the overlap $U_p\cap U_q$, the transition map is given by the integral linear map $\Phi_{U_qU_p}(w) = \begin{bmatrix}
        0 & 1\\
        1 & 1
    \end{bmatrix}w$. Note that this tropical atlas yields the same tropical structure as the one given in the proof of \Cref{thm:bijection}. 
\end{example}

%%%%%%%%%%%%%%%%%%%%%%%%%%%%%%%%%%%%%%%%%%%%%
\section{Tropical Geometry on Metric Graphs}\label{sec:trop-geom-graph}

In this section, we present geometric constructions on a metric graph from the perspectives of both metric geometry and tropical geometry. A key concept is the notion of \emph{tropical harmonic 1-forms} on a metric graph. We begin by providing an intrinsic definition of tropical harmonic 1-forms, followed by a characterization in terms of combinatorial models. We then define the \emph{tropical Jacobian} of a metric graph, which is a tropical torus endowed with a flat Riemannian metric known as the \emph{tropical polarization}. Finally, we introduce the tropical Abel–-Jacobi map, whose image is what we call the \emph{tropical Abel--Jacobi transform} of the metric graph.

\subsection{Tropical Harmonic 1-Forms}\label{sec:trop-harmonic-form}

%Using metric geometry, we can define tangent vectors and differentiable functions on arbitrary length spaces \citep{burago2001course}. Here we specialize the definitions for metric graphs. 

Let $\Gamma$ be a metric graph. For any two points $p,q\in \Gamma$, a \emph{shortest curve} connecting $p$ and $q$ is a continuous curve $\alpha:[0,L]\to\Gamma$ such that for any $0\leq t_1\leq t_2\leq L$,
\begin{equation*}
    d(\alpha(t_1),\alpha(t_2)) = \frac{t_2-t_1}{L}d(p,q) .
\end{equation*}
A \emph{geodesic} is a locally shortest curve. Given $p\in\Gamma$, two geodesics $\alpha,\beta$ emanating from $p$ are equivalent if there exists $\epsilon>0$ such that for all $s,t<\epsilon$,
\begin{equation*}
|d(p,\alpha(t))-d(p,\beta(s))| = d(\alpha(t),\beta(s)) .
\end{equation*}
The equivalence class $[\alpha]$ is called a \emph{direction} at $p$. Let $r_\alpha(t) = d(p,\alpha(t))$. The \emph{tangent vector} of $\alpha$ at $p$ is a pair $\alpha'(0) 
 =([\alpha],r'_\alpha(0))$ where $[\alpha]$ represent its direction and $r'_\alpha(0)$ represents its norm. The set of all tangent vectors at $p$ is denoted by $T_p(\Gamma)$ and is called the \emph{tangent cone} at $p$. 

\begin{definition}
    Let $f:\Gamma\to\mathbb{R}$ be a function on $\Gamma$. Given a tangent vector $v_p\in T_p(\Gamma)$, let $\alpha$ be a geodesic such that $\alpha(0)=p$ and $\alpha'(0)=v_p$, the \emph{directional derivative} of $f$ along $v_p$ is defined as
    \begin{equation*}
    D_{v_p}f = \lim_{t\to 0^+}\frac{ f(\alpha(t))-f(p)}{t} ,
    \end{equation*}
    provided the limit exists. The function is \emph{piecewise differentiable} if $D_{v_p}f$ exists for all $p\in \Gamma$ and $v_p\in T_p(\Gamma)$. 
\end{definition}

Using the tropical structure of $\Gamma$, tangent vectors are represented by primitive lattice vectors under tropical charts. Specifically, for any $p\in\Gamma$, let $(U,\phi)$ be a tropical chart such that $\phi(U)$ is the union of $k$ line segments with rational slopes balanced at $\phi(p)$. Suppose $\alpha:[0,\epsilon]\to\Gamma$ is a geodesic contained in $U$. Then $\tilde{\alpha}(t) = \phi(\alpha(t))$ is a geodesic along one of the $k$ line segments. Since the primitive lattice vectors are endowed with unit lengths, there exists $v_i\in\mathbb{Z}^{k-1}$ and $c>0$ such that $\tilde{\alpha}(t) = ctv_i+\phi(p)$. The pushforward of the tangent vector $\alpha'(0)$ under $\phi$ is $cv_i$. Thus we see that the tangent cone $T_p(\Gamma)$ is isometric to the wedge sum of $k$ half-lines $\bigvee_{i=1}^k\mathbb{R}_+v_{i}$ where each $v_i$ has unit length.

Let $S_p(\Gamma)=\{v_1,\ldots,v_k\}$ be the set of unit tangent vectors at $p$. Consider the direct sum of vector spaces $\bigoplus\mathbb{R}v_i$.
A cotangent vector at $p$ is an element in the dual vector space $(\bigoplus\mathbb{R}v_i)^*=\bigoplus\mathbb{R}v_i^*$. 
% \begin{equation*}
% v_i^*(v_j)=\delta_{ij}=\begin{cases}
%     1,\,j=i\\
%     0,\,j\neq i .
% \end{cases}
% \end{equation*}
A cotangent vector $\omega=\sum_{i=1}^ka_iv_i^*$ at $p$ is \emph{tropical} if its coefficients are balanced, i.e., $\sum_{i=1}^na_i=0$.  The space of tropical cotangent vectors at $p$ is denoted by $T_{p}^*(\Gamma)$.

For any $p\in\Gamma$ with valence 2, it has a neighborhood isometric to an interval in $\mathbb{R}$. We will identify its tangent cone and tropical cotangent space with the standard tangent space and cotangent space of a point in the interval. 

\begin{definition}
    Let $\Gamma$ be a metric graph. A \emph{tropical 1-form} $\omega$ on $\Gamma$ is a tropical cotangent vector field such that $\omega_p\in T_p^*(\Gamma)$ for any $p\in\Gamma$.
\end{definition}

Since the transition functions between tropical charts are integral affine maps, instead of considering the class of piecewise differentiable functions, we will restrict to a smaller class of functions that are preserved under integral affine transformations. Let $U\subseteq \Gamma$ be any open subset. A function $f$ on $U$ is \emph{piecewise linear} if $f\circ\phi_{U\cap V}^{-1}:\mathbb{R}^{k-1}\to\mathbb{R}$ is a continuous, piecewise linear function under any tropical chart $(U\cap V,\phi_{U\cap V})$. Let $\mathscr{A}_\Gamma(U)$ be the set of all piecewise linear functions on $U$. The \emph{Laplacian} of $f\in\mathscr{A}_\Gamma(U)$ is defined as
\begin{equation}\label{eq:lapla-metric-graph}
\Delta f  = \sum_{p\in U}\bigg(\sum_{v_p\in S_p(\Gamma)}D_{v_p}f\bigg)\delta_p \, ,
\end{equation}
where $\delta_p$ is the Dirac measure at $p$ \citep{baker2006metrized}. A function $f\in\mathscr{A}_\Gamma(U)$ is \emph{harmonic} if $\Delta f=0$. The set of all harmonic functions on $U$ is denoted by $\mathscr{H}_\Gamma(U)$. 

For any piecewise linear function $f:U\to\mathbb{R}$, its \emph{differential} $\mathrm{d}f$ is a cotangent vector field given by $\mathrm{d}f_p(v_p)=D_{v_p}f$ for $v_p\in T_p(\Gamma)$. If $f\in\mathscr{H}_\Gamma(U)$, at any point $p\in U$, 
\begin{equation*}
\sum_{v_p\in S_p(\Gamma)}\mathrm{d}f_p(v_p) = \sum_{v_p\in S_p(\Gamma)} D_{v_p}f = 0 .
\end{equation*}
Thus $\mathrm{d}f_p\in T_p^*(\Gamma)$, i.e., the differential of a harmonic function on $U$ defines a tropical 1-form on $U$. Moreover, it serves as a local model for tropical harmonic 1-forms on $\Gamma$.

\begin{definition}
    Let $\omega$ be a cotangent vector field on $\Gamma$. It is called a \emph{tropical harmonic 1-form} if for any tropical chart $(U,\phi)$, $\omega|_U = \mathrm{d}f$ for some $f\in\mathscr{H}_\Gamma(U)$. The set of all tropical harmonic 1-forms on $\Gamma$ is denoted by $\Omega(\Gamma)$.
\end{definition}

% \begin{proposition}
%     A cotangent vector field $\omega$ is a tropical harmonic 1-form on $\Gamma$ if and only if for any tropical chart $(U,\phi)$ with primitive lattice vectors $v_1,\ldots,v_n$, there exist constants $a_1,\ldots,a_n\in\mathbb{R}$ such that $\sum_{i=1}^na_i=0$, and
%     $$(\phi^{-1})^*(\omega|_U) = a_1\mathrm{d}t_1+\cdots+a_n\mathrm{d}t_n$$
%     where $\mathrm{d}t_i$ is the standard 1-form along the line $\mathbb{R}v_i$.
% \end{proposition}

% \begin{proof}
%     Suppose $\omega$ is a tropical harmonic 1-form on $\Gamma$. Let $(U,\phi)$ be a tropical chart such that $\phi(U)$ is the union of $n$ line segments with rational slopes balanced at $\phi(p)$ for some $p\in\Gamma$, and $f:U\to\mathbb{R}$ be a harmonic function such that $\omega|_U = \mathrm{d}f$. By the harmonicity of $f$, there exist constants $a_1,\ldots,a_n\in\mathbb{R}$ such that $\sum_{i=1}^na_i=0$, and on the $i$th segment  $$(\phi^{-1})*(f)(tv_i) = f\circ\phi^{-1}(tv_i) = a_it+f(p)$$ Therefore,
%     $$
%     (\phi^{-1})^*(\omega|_U) = \mathrm{d}(f\circ\phi^{-1}) =  a_1\mathrm{d}t_1+\cdots+a_n\mathrm{d}t_n
%     $$
% \end{proof}

Let $G$ be a combinatorial model for $\Gamma$. For each edge $e\in E(G)$, let $\psi_e:[0,\ell(e)]\to \Gamma$ be its length parameterization. A harmonic function on $e$ is equivalent to an affine function on $e$. Thus any tropical harmonic 1-form $\omega$ restricted to $e$ is of the form
\begin{equation}\label{eq:oneform-edge}
   \psi_e^*(\omega|_e) = a_e\mathrm{d}t_e \, ,
\end{equation}
where $\mathrm{d}t_e$ is the standard 1-form on the interval $[0,\ell(e)]$ and $a_e\in\mathbb{R}$ is a constant.

We can then characterize tropical harmonic 1-forms in terms of combinatorial models of $\Gamma$, which is crucial for our computations following.

\begin{theorem}\label{thm:oneform-char}
    Let $\Gamma$ be a metric graph.
    \begin{enumerate}
        \item Fix a combinatorial model $G$ for $\Gamma$. Any tropical harmonic 1-form $\omega\in\Omega(\Gamma)$ defines a function $h_{G,\omega}:E(G)\to\mathbb{R}$ sending each edge $e$ to $a_e$ given by \cref{eq:oneform-edge}. For any vertex $p\in V(G)$, the function $h_{G,\omega}$ is such that
        \begin{equation}\label{eq:conservative}
            \sum_{\substack{e\in E(G)\\
        e_-=p}}h_{G,\omega}(e) \, = \sum_{\substack{e\in E(G)\\
        e_+=p}}h_{G,\omega}(e) .
        \end{equation}
        
        \item Any combinatorial model $G$ for $\Gamma$ yields a vector space isomorphism
        $$
        \begin{aligned}
            h_{G}:\Omega(\Gamma)&\to H_1(G;\mathbb{R})\\
            \omega&\mapsto \sum_{e\in E(G)}h_{G,\omega}(e)e .
        \end{aligned}
        $$
    \end{enumerate}
\end{theorem}

\begin{proof}
We now prove the two components of the statement.
    \begin{enumerate}
        \item Let $p\in V(G)$, and $S_p(\Gamma)=\{v_1,\ldots,v_k\}$ be the set of unit tangent vectors at $p$. Since $\omega$ is a tropical harmonic 1-form, we have
        \begin{equation}\label{eq:balance-at-p}
         \sum_{i=1}^k\omega_p(v_i) = 0 .  
        \end{equation}
        
        If $e_i\in E(G)$ is an edge whose initial vertex is $p$, then the geodesic corresponding to $v_i$ is given by $\psi_{e_i}(t)$. By definition,
        \begin{equation}\label{eq:initial}
          \omega_p(v_i)=\omega_p(\psi_{e_i}'(0))=h_{G,\omega}(e_i) . 
        \end{equation}
        
        If $e_i\in E(G)$ is an edge whose terminal vertex is $p$, then the geodesic corresponding to $v_i$ is given by $\psi_{e_i}^{-}(t)=\psi_{e_i}(\ell(e_i)-t)$. Thus,
        \begin{equation}\label{eq:terminal}
            \omega_p(v_i) = \omega_p((\psi_{e_i}^-)'(0)) = -h_{G,\omega}(e_i) .
        \end{equation}
        Substituting \cref{eq:initial} and \cref{eq:terminal} into \cref{eq:balance-at-p}, we obtain \cref{eq:conservative}.
        
        \item First, we verify that 
        \begin{equation}
        \begin{aligned}
        \partial\bigg(\sum_{e\in E(G)}h_{G,\omega}(e)e\bigg) &= \sum_{e\in E(G)}h_{G,\omega}(e)\partial(e)=\sum_{e\in E(G)}h_{G,\omega}(e)(e_+-e_-)\\
        &= \sum_{p\in V(G)}\bigg(\sum_{\substack{e\in E(G)\\
        e_+=p}}h_{G,\omega}(e)-\sum_{\substack{e\in E(G)\\
        e_-=p}}h_{G,\omega}(e)\bigg)p=0.
        \end{aligned}
        \end{equation}
    Thus, the map $h_G$ is well-defined. 
    
    Let $\omega,\omega'\in\Omega(\Gamma)$, from \cref{eq:oneform-edge}, we have
    $
    h_{G,\omega+\omega'} = h_{G,\omega}+h_{G,\omega'}
    $.
    It follows that $h_G$ is a vector space homomorphism. If $h_G(\omega) = 0$, then $h_{G,\omega}(e)=0$ for all $e\in E(G)$, which implies $\omega=0$. Thus $h_G$ is injective.

    It remains to prove that $h_G$ is surjective. Suppose $\sigma=\sum_{e\in E(G)}a_ee\in H_1(G;\mathbb{R})$. Construct a cotangent vector field $\omega$ such that on every edge $\omega$ is given by \cref{eq:oneform-edge}. Since $\partial\sigma=0$,  at any $p\in V(G)$ we have
    \begin{equation*}
    \sum_{\substack{e\in E(G)\\
        e_-=p}}a_e \, = \sum_{\substack{e\in E(G)\\
        e_+=p}}a_e \, .
    \end{equation*}
    It follows that the cotangent vector field $\omega$ is balanced at $p$ and thus is a tropical harmonic 1-form.
    \end{enumerate}
\end{proof}

For a combinatorial graph $G$, define inner products on cochain spaces by
\begin{equation*}
\begin{aligned}
    \langle \xi_1,\xi_2 \rangle &= \sum_{v\in V(G)}\xi_1(v)\xi_2(v),\, \forall\ \xi_1,\xi_2\in C^0(G;\mathbb{R}) ,\\
    \langle \eta_1,\eta_2 \rangle &= \sum_{e\in E(G)}\eta_1(e)\eta_2(e)\ell(e),\, \forall\ \eta_1,\eta_2\in C^1(G;\mathbb{R}) .
\end{aligned}
\end{equation*}
Define the coboundary operator $\delta:C^0(G;\mathbb{R})\to C^1(G;\mathbb{R})$ by 
\begin{equation}\label{eq:cobound-graph}
\delta\xi(e) = \frac{\xi(e_+)-\xi(e_-)}{\ell(e)} .
\end{equation}
The adjoint operator of $\delta$ is 
\begin{equation}\label{eq:adjoint-cobound-graph}
\delta^*\eta(p) = \sum_{\substack{e\in E(G)\\
e_+=p}}\eta(e) - \sum_{\substack{e\in E(G)\\
e_-=p}}\eta(e) .
\end{equation}
The \emph{combinatorial Hodge Laplacian} on $C^1(G;\mathbb{R})$ is given by $\delta\delta^*$. A function $\eta:E(G)\to\mathbb{R}$ is harmonic if $\delta\delta^*\eta = 0$ \citep{baker2011metric}. Since $\langle \delta\delta^*\eta,\eta\rangle=\langle\delta^*\eta,\delta^*\eta\rangle=0$, we have $\delta^*\eta=0$, which implies
\begin{equation}\label{eq:harmonic-function}
\sum_{\substack{e\in E(G)\\
e_+=p}}\eta(e) = \sum_{\substack{e\in E(G)\\
e_-=p}}\eta(e) .
\end{equation}
From \Cref{thm:oneform-char}, we see that a tropical harmonic 1-form $\omega$ is represented by a harmonic function $h_{G,\omega}$ once we fix a combinatorial model $G$.

\begin{remark}
 The notion of harmonic functions on a metric graph $\Gamma$ differs from the notion of harmonic functions on its combinatorial model $G$. In fact, the only harmonic functions on $\Gamma$ are constant functions \citep[Theorem 3]{baker2006metrized}, while the space of harmonic functions on $G$ is isomorphic to its 1-dimensional cohomology group $H^1(G;\mathbb{R})$. 
\end{remark}

\subsection{The Tropical Jacobian and the Tropical Abel--Jacobi Map}\label{sec:trop-aj}

We now present the tropical Jacobian and the tropical Abel--Jacobi map which are objects of central interest of our work.

Let $\gamma:[a,b]\to\Gamma$ be a geodesic and $\omega$ be a cotangent vector field on $\Gamma$. Define the path integration of $\omega$ along $\gamma$ by
$$
\int_\gamma \omega = \int_a^b\omega(\gamma'(t))\,\mathrm{d}t .
$$
The singular 1-chain group $C_1(\Gamma;\mathbb{R})$ consists of formal linear combinations of continuous functions from $[0,1]$ to $\Gamma$. Without loss of generality, we may use its subspace of formal linear combinations of geodesics and, perhaps by a slight abuse of notation, still denote it by $C_1(\Gamma;\mathbb{R})$. The path integration defines a pairing between the 1-chain group and the space of tropical harmonic 1-forms 
$$
\begin{aligned}
\int: C_1(\Gamma;\mathbb{R})\times \Omega(\Gamma)&\to \mathbb{R}\\
(\gamma,\omega)&\mapsto \int_{\gamma}\omega .
\end{aligned}
$$
Let $\Omega^*(\Gamma)$ be the dual vector space of $\Omega(\Gamma)$. The pairing induces a vector space homomorphism
$$
\begin{aligned}
    \mathcal{P}:C_1(\Gamma;\mathbb{R})&\to \Omega^*(\Gamma)\\
    \gamma&\mapsto \int_{\gamma} \quad .
\end{aligned}
$$
By \Cref{thm:oneform-char}, the linear map $\mathcal{P}$ restricted to $H_1(\Gamma;\mathbb{R})$ is an isomorphism, and $\mathcal{P}$ embeds $H_1(\Gamma;\mathbb{Z})$ as a full rank lattice of $\Omega^*(\Gamma)$. For simplicity, we will identify $H_1(\Gamma;\mathbb{Z})$ as a lattice of $\Omega^*(\Gamma)$ without explicitly referencing the embedding $\mathcal{P}$.
%Moreover, the linear map $\mathcal{P}$  embeds the lattice $H_1(\Gamma;\mathbb{Z})$ into $\Omega^*(\Gamma)$. Taking quotient of the lattice we obtain a real torus $\Omega^*(\Gamma)/H_1(\Gamma;\mathbb{Z})$. 
% \begin{definition}
%     Let $g=\dim(\Omega(\Gamma))=\dim(H_1(\Gamma;\mathbb{R}))$. The dimension of $\Omega(\Gamma)$ is called the \emph{genus} of the metric graph $\Gamma$.
% \end{definition}

%Other than the lattice $H_1(\Gamma;\mathbb{Z})$, there is another lattice of $\Omega^*(\Gamma)$ given as follows. 

\begin{definition}
    A tropical harmonic 1-form $\omega\in\Omega(\Gamma)$ is \emph{integral} if it takes integral values on unit tangent vectors at every point $p\in\Gamma$. The set of integral tropical harmonic 1-forms is denoted by $\Omega_{\mathbb{Z}}(\Gamma)$.
\end{definition}

The proof of \Cref{thm:oneform-char} also shows that any combinatorial model $G$ for $\Gamma$ yields a group isomorphism $
h_G:\Omega_{\mathbb{Z}}(\Gamma)\to H_1(G;\mathbb{Z})
$. Let $\Omega_{\mathbb{Z}}^*(\Gamma)=\mathrm{Hom}(\Omega_{\mathbb{Z}}(\Gamma),\mathbb{Z})$ be the dual lattice of $\Omega_{\mathbb{Z}}(\Gamma)$. We obtain that $\Omega_{\mathbb{Z}}^*(\Gamma)$ is another full rank lattice of $\Omega^*(\Gamma)$. 

%A \emph{tropical torus} is a real torus together with a second lattice called the tropical structure of the torus \citep{ji2012complete}.

\begin{definition}\label{def:trop-jac}
    Let $\Gamma$ be a metric graph. Its \emph{tropical Jacobian} is defined as a pair 
    $$
    \jac(\Gamma) = \left(\frac{\Omega^*(\Gamma)}{H_1(\Gamma;\mathbb{Z})}\,, \Omega_{\mathbb{Z}}^*(\Gamma)\right),
    $$ 
    where $\Omega_{\mathbb{Z}}^*(\Gamma)$ is called the tropical structure of the real torus $\Omega^*(\Gamma)/H_1(\Gamma;\mathbb{Z})$.
\end{definition}

\begin{remark}
Let $\mathbb{L},\mathbb{M}$ be two lattices contained in a common vector space $V$. A \emph{tropical torus} is a pair $(V/\mathbb{L},\mathbb{M})$. The lattice $\mathbb{M}$ is referred to as the tropical structure of $V/\mathbb{L}$ because it induces an affine manifold structure on the torus $V/\mathbb{L}$. A detailed discussion on tropical tori is given in \Cref{app:trop-abel-var}.
\end{remark}

Fix a combinatorial model $G$ for $\Gamma$. Define a bilinear form $Q_G$ on $C_1(G;\mathbb{R})$ by assigning
\begin{equation}\label{eq:inner-product}
  Q_G(e,e') = \begin{cases}
    \ell(e),&\text{ if } e=e'\\
    0,&\text{ if } e\neq e' 
\end{cases}  
\end{equation}
and extending bilinearly to the whole space $C_1(G;\mathbb{R})$. The bilinear form is compatible with respect to refinements of combinatorial models and extends to be a well-defined bilinear form $Q_\Gamma$ on $C_1(\Gamma;\mathbb{R})$. Since $H_1(\Gamma;\mathbb{R})$ is a subspace of $C_1(\Gamma;\mathbb{R})$, and $H_1(\Gamma;\mathbb{R})$ is isomorphic to $\Omega^*(\Gamma)$, we can identify $Q_\Gamma$ as a bilinear form on $\Omega^*(\Gamma)$. 

\begin{definition}\label{def:trop-polar}
    Let $\sigma,\sigma'\in H_1(\Gamma;\mathbb{R})$. The bilinear form on $\Omega^*(\Gamma)$ given by
    \begin{equation}\label{eq:trop-polar}
    \left\langle\int_{\sigma\,,}\,\int_{\sigma'}\right\rangle = Q_\Gamma(\sigma,\sigma')  
    \end{equation}
    is called the \emph{tropical polarization} on the tropical Jacobian $\jac(\Gamma)$.
\end{definition}

By construction, the bilinear form $Q_\Gamma$ is symmetric positive definite and defines an inner product on $\Omega^*(\Gamma)$. Moreover, it defines a flat Riemannian metric on the torus $\jac(\Gamma)$.

%The tropical polarization thus yields an isomorphism $Q_\Gamma^\# = \Omega^*(\Gamma)\to\Omega^{**}(\Gamma)\cong\Omega(\Gamma)$, together with the restriction of  $\mathcal{P}$ to $:H_1(\Gamma;\mathbb{Z})$ we obtain an embedding
%$$
%\iota_Q : H_1(\Gamma;\mathbb{Z})\to\Omega(\Gamma)
%$$

\begin{remark}
For a tropical torus $(V/\mathbb{L},\mathbb{M})$, a \emph{tropical polarization} is a symmetric positive definite bilinear form on $V$ whose restriction to $\mathbb{L}\times\mathbb{M}$ is integral. A tropical polarization is \emph{principal} if there exist bases of $\mathbb{L}$ and $\mathbb{M}$ such that the bilinear form is represented by the identity matrix. For a metric graph $\Gamma$, the bilinear form $Q_\Gamma$ is shown to be a principal polarization, thereby turning the tropical Jacobian $\jac(\Gamma)$ into a \emph{principally polarized tropical abelian variety}. We discuss tropical abelian varieties in detail in \Cref{app:trop-abel-var}.
\end{remark}

We are now ready to define the \emph{tropical Abel--Jacobi map} for a metric graph. In the classical setting, the Abel--Jacobi map is defined for a Riemann surface (or complex projective algebraic curve). Let $\mathscr{C}$ be a compact Riemann surface. Fix a base point $q\in \mathscr{C}$. The (complex) Abel--Jacobi map sends a point $p\in \mathscr{C}$ to the path integral of holomophic differential 1-forms over a smooth path $\gamma_{q,p}$ joining $q$ and $p$. In the tropical setting, tropical harmonic 1-forms on a metric graph serve as the tropical analog of holomorphic differential 1-forms. The tropical Abel--Jacobi map can thus be defined in a manner analogous to its complex counterpart.

\begin{definition}
    Let $\Gamma$ be a metric graph and $\jac(\Gamma)$ be its tropical Jacobian. Fix a base point $q\in \Gamma$. The tropical Abel--Jacobi map is given by
    \begin{equation}\label{eq:abel-jacobi}
        \begin{aligned}
            \ajmap:\Gamma&\to\jac(\Gamma)\\
            p&\mapsto \int_{\gamma_{q,p}} \quad\big(\mathrm{mod}\, H_1(\Gamma;\mathbb{Z})\big), 
        \end{aligned}
    \end{equation}
    where $\gamma_{q,p}$ is any geodesic joining $q$ and $p$ on $\Gamma$. Its image $\ajmap(\Gamma)$ is called the \emph{tropical Abel--Jacobi transform} of $\Gamma$ relative to $q$.
\end{definition}

The definition of $\ajmap$ is independent of choice of geodesics connecting $q$ and $p$. In fact, suppose $\tilde{\gamma}_{q,p}$ is another geodesic. Then $\gamma_{q,p}-\tilde{\gamma}_{q,p}\in H_1(\Gamma;\mathbb{Z})$, which implies that integration over $\gamma_{q,p}$ and integration over $\tilde{\gamma}_{q,p}$ define the same element in $\jac(\Gamma)$. Thus the tropical Abel--Jacobi map is well-defined.

% \yc{start from here} Since $\Gamma$ and $\jac(\Gamma)$ both have tropical structures, it is natural to ask if the tropical Abel--Jacobi map $\ajmap$ is, indeed, \emph{tropical}: A continuous map $F:\Gamma\to\jac(\Gamma)$ is called tropical if $F$ is piecewise linear under tropical charts on both sides. Furthermore, it satisfies the balancing condition in the sense that, for any $p\in\Gamma$, the pushforward map ${F_*}_p:T_p(\Gamma)\to T_{F(p)}(\jac(\Gamma))$ sends integral tangent vectors to integral tangent vectors, and $\sum_{v\in S_p(\Gamma)}{F_*}_p(v) = 0$. \yc{definition of tropical map from Baker's paper. Note (2) is for integral vector but (3) is for primitive integral vector}

% \begin{theorem}{\citep[Theorem 4.1]{baker2011metric}}
%     The map $\ajmap:\Gamma\to\jac(\Gamma)$ defined by \cref{eq:abel-jacobi} is tropical.
% \end{theorem}

A comprehensive characterization of the tropical Abel--Jacobi map can be found in \cite{baker2011metric}. We will revisit some key properties of the tropical Abel--Jacobi map from a computational perspective in \Cref{subsec:property} after introducing our main algorithm.

\begin{remark}
In a more general setting, the tropical Abel--Jacobi map can be defined on the divisor group of a metric graph. The tropical Abel--Jacobi theorem states that this map induces a canonical isomorphism from the tropical divisor class group to the tropical Jacobian. We give a summary of Abel--Jacobi theories for different types of data in \Cref{app:aj-theory}. 
\end{remark}

%%%%%%%%%%%%%%%%%%%%%%%%%%%%%%%%%%%%%%%%%%%%%%%%%%
\section{Computing the Tropical Abel--Jacobi Transform}\label{sec:alg}

In this section, we present our main algorithm to compute and sample vectors from the tropical Abel--Jacobi transform of a metric graph. We show that computing the tropical Abel--Jacobi transform on combinatorial models can be efficiently addressed using classical graph algorithms and matrix computations. The piecewise linear nature of the tropical Abel--Jacobi transform allows us to sample additional vectors via interpolation. An analysis of computational complexity is given. We then examine how the output vectors vary with respect to different combinatorial models, and prove  properties of the tropical Abel--Jacobi transform from a computational perspective. We also demonstrate how these properties can be used to preprocess and simplify combinatorial models, improving computational efficiency.

\subsection{Main Algorithm}\label{subsec:main-algorithm}

\paragraph{Goal of Computation.}

Let $\Gamma$ be a metric graph. Our goal is to compute a list of vectors, or a point cloud, from the tropical Abel--Jacobi transform $\ajmap(\Gamma)\subseteq \jac(\Gamma)$. By fixing a basis of the vector space $\Omega^*(\Gamma)$, the output is a list of vectors in $\mathbb{R}^g$, where $g=\dim(H_1(\Gamma;\mathbb{R}))$ is called the \emph{genus} of $\Gamma$. It is important to note that the vector representations of $\ajmap(\Gamma)$ are not unique since they are allowed to translate along lattice vectors in $H_1(\Gamma;\mathbb{Z})$, resulting in infinitely many equivalent representations. Therefore, to make the vector representations meaningful, we also must compute a set of vectors representing the lattice basis. 

In practice, a metric graph $\Gamma$ is always represented by a combinatorial model $G$, whose data structure is either a list of edges with weights or an adjacency matrix with weights. We split our computation into two parts. First we compute the tropical Abel--Jacobi transform for the vertex set $V(G)$. Then we show that additional vectors from $\ajmap(\Gamma)$ can be obtained via linear interpolations based on $\ajmap(V(G))$.

\paragraph{Fundamental Bases.}

In the first step, we aim to choose bases for the lattices $H_1(G;\mathbb{Z})$, $\Omega^*_{\mathbb{Z}}(\Gamma)$ and the vector space $\Omega^*(\Gamma)$. Let $G$ be a combinatorial model of $\Gamma$ with $n_G$ vertices and $m_G$ edges. It is well-known in graph theory that a basis of 1-cycles for $H_1(G;\mathbb{Z})$ can be determined by a spanning tree of $G$ \citep{deo1982algorithms}: Choose a spanning tree $\mathrm{ST}$ for $G$. Adding any edge in the complement $G\backslash \mathrm{ST}$ to the spanning tree $\mathrm{ST}$ generates a 1-cycle in $H_1(G;\mathbb{Z})$. The corresponding set of 1-cycles $\{\sigma_i\}_{1\le i\le g}$ is the fundamental 1-cycle basis with respect to the spanning tree $\mathrm{ST}$. Through the isomorphism $h_G:\Omega(\Gamma)\to H_1(G;\mathbb{R})$ in \Cref{thm:oneform-char}, the tropical harmonic 1-forms defined by $\omega_i = h_G^{-1}(\sigma_i)$ form a basis for $\Omega_{\mathbb{Z}}(\Gamma)$. Let $\{\omega_i^*\}_{1\le i\le g}$ be its dual basis for $\Omega_{\mathbb{Z}}^*(\Gamma)$. We call $\{\omega_i^*\}_{1\le i\le g}$ the \emph{fundamental dual 1-form basis}, or fundamental basis for short, of the tropical Jacobian $\jac(\Gamma)$ induced by the spanning tree $\mathrm{ST}$. Thus, our goal is to compute vectors of the tropical Abel--Jacobi transform $\ajmap(\Gamma)$ under the fundamental basis. 

\paragraph{Cycle--Edge and Path--Edge Incidence Matrices.}

In the second step, we compute the tropical Abel--Jacobi transform for the vertex set $V(G)=\{p_1,\ldots,p_{n_G}\}$. Essentially, we aim to compute a $g\times n_G$ matrix $\matV$ such that the $j$th column $\matV[:,j]$ consists of coefficients of $\ajmap(p_j)$ under the fundamental basis $\{\omega_i^*\}_{1\le i\le g}$. We assume the spanning tree $\mathrm{ST}$ is rooted at $p_1\in V(G)$ and set $q=p_1$ as the base point of the tropical Abel--Jacobi map. For any $p_j\in V(G)$, we pick a path $\gamma_j$ from $p_1$ to $p_j$. We view $\gamma_j$ as a 1-chain in $C_1(G;\mathbb{Z})$, and let $\gamma_j(e)$ be the coefficient of $e$ in $\gamma_j$. By definition, we have
\begin{equation}\label{eq:aj-vector}
\ajmap(p_j) = \int_{\gamma_{j}}  = \sum_{e\in E(G)}\bigg(\gamma_j(e)\int_e\bigg) .
\end{equation}
Thus, it suffices to compute $\int_e$ for each $e\in E(G)$ under the fundamental basis $\{\omega_i^*\}_{1\le i\le g}$. Suppose 
\begin{equation}\label{eq:edge-vector}
  \int_e = c^e_1\omega_1^*+\cdots+c^e_g\omega_g^* . 
\end{equation}
Then the $i$th coefficient $c^e_i$ is given by
$$
c^e_i = \int_e\omega_i|_e=\int_e \sigma_i(e)\,\mathrm{d}t_e = \sigma_i(e)\ell(e) .
$$
Let $E(G) = \{e_1,\ldots,e_{m_G}\}$ be the edge set of $G$. Construct a $g\times m_G$ matrix $\matC$ as 
$$
\matC[i,j] = \sigma_i(e_j) . 
$$
The matrix $\matC$ is called the \emph{cycle--edge incidence matrix} of $G$. Similarly, consider the $n_G\times m_G$ matrix $\matY$ defined by
$$
\matY[i,j] = \gamma_i(e_j) .
$$
We call $\matY$ the \emph{path--edge incidence matrix}. The diagonal matrix $\matL=\mathrm{diag}\{\ell(e_1),\ldots,\ell(e_{m_G})\}$ of edge lengths is called the \emph{edge length matrix}. By \Cref{eq:aj-vector,eq:edge-vector}, the tropical Abel--Jacobi transform of $V(G)$ is given by
\begin{equation}\label{eq:aj-mat-1}
\matV = \matC\matL\matY^\top ,
\end{equation}
where $\matY^\top$ is the transpose of $\matY$.

The \textbf{key observation} here is that, once a spanning tree $\mathrm{ST}$ is computed in the first step, the paths from $p_1$ to all the $p_j$'s are uniquely determined. Instead of picking an arbitrary path from $p_1$ to $p_j$, we use the unique paths inside the spanning tree $\mathrm{ST}$. As a result, we can reduce the size of both the cycle--edge incidence matrix and the path--edge incidence matrix by storing only the columns corresponding to the edges in $\mathrm{ST}$. This approach allows for a more economic matrix storage and faster computation. 

Specifically, we reorder the edges so that $e_1,\ldots,e_{n_G-1}$ are edges in the spanning tree $\mathrm{ST}$, and $e_{n_G-1+i}$ is the edge determining the $i$th fundamental 1-cycle $\sigma_i$ for $1\le i\le g$. Hence the cycle--edge incidence matrix is of the form
$$
\matC = \begin{bmatrix}
    \matC_{\mathrm{ST}} & \mathbf{I}_{g}
\end{bmatrix} ,
$$
where $\mathbf{I}_g$ is the $g\times g$ identity matrix and $\matC_{\mathrm{ST}}$ is the submatrix representing the incidence relationship between fundamental 1-cycles and edges in the spanning tree $\mathrm{ST}$. We call the $g\times (n_G-1)$ submatrix $\matC_{\mathrm{ST}}$ the \emph{reduced} cycle--edge incidence matrix. Similarly, the path--edge incidence matrix is of the form
$$
\matY = \begin{bmatrix}
    \matY_{\mathrm{ST}} & \bm{0}
\end{bmatrix} .
$$
We also call the $n_G\times(n_G-1)$ submatrix $\matY_{\mathrm{ST}}$ the \emph{reduced} path--edge incidence matrix. Consequently, we can decompose the edge length matrix into blocks 
$$
\matL = \begin{bmatrix}
    \matL_{\mathrm{ST}} & \bm{0}\\
    \bm{0} & \matL_g
\end{bmatrix} ,
$$
where the first $(n_G-1)\times (n_G-1)$ block corresponds to lengths of the edges in $\mathrm{ST}$ and the last $g\times g$ block corresponds to lengths of the edges determining fundamental 1-cycles. Therefore, \Cref{eq:aj-mat-1} is equivalent to 
$$
\matV = \matC_{\mathrm{ST}}\matL_{\mathrm{ST}}\matY_{\mathrm{ST}}^\top .
$$
The vector representation of $\ajmap(p_j)$ under the fundamental basis $\{\omega_i^*\}_{1\le i\le g}$ is given by the vector
$$
\matV[:,j]=\matC_{\mathrm{ST}}\matL_{\mathrm{ST}}\matY_{\mathrm{ST}}^\top[:,j].
$$ 

We also need to compute the vector representations of $\sigma_1,\ldots,\sigma_g$ for the lattice $H_1(\Gamma;\mathbb{Z})$. Let $\matQ$ be the $g\times g$ matrix representing lattice vectors in the tropical Jacobian. Since the coefficients of the $i$th cycle $\sigma_i$ are given by the $i$th row of the cycle--edge incidence matrix $\matC[i,:]$, by \Cref{eq:aj-vector} we have
\begin{equation}\label{eq:inner-product-matrix}
    \begin{aligned}
       \matQ &= \matC\matL\matC^\top= \begin{bmatrix}
           \matC_{\mathrm{ST}} & \mathbf{I}_g
       \end{bmatrix}\begin{bmatrix}
           \matL_{\mathrm{ST}}&\bm{0}\\
           \bm{0}& \matL_g
       \end{bmatrix}\begin{bmatrix}
           \matC^\top_{\mathrm{ST}}\\
           \mathbf{I}_g
       \end{bmatrix}\\
       &=\matC_{\mathrm{ST}}\matL_{\mathrm{ST}}\matC_{\mathrm{ST}}^\top+\matL_g .
    \end{aligned}
\end{equation}

As a byproduct, we obtain the matrix representing the tropical polarization on $\jac(\Gamma)$ under the basis $\{\omega_i^*\}_{1\le i\le g}$. In fact, suppose $\eta_i\in H_1(\Gamma;\mathbb{R})$ is such that 
$$
\int_{\eta_i}=\omega_i^* , 
$$
and $\eta_i = \sum_{j=1}^g b_{ij}\sigma_j$ for some $b_{ij}\in\mathbb{R}$. Then we can compute
$$
\int_{\eta_i}\omega_k = \sum_{j=1}^gb_{ij}\int_{\sigma_j}\omega_k=\sum_{j=1}^gb_{ij}\matQ[j,k]=\omega_i^*(\omega_k)=\delta_{ik} \, . 
$$
Thus $b_{ij} = \matQ^{-1}[i,j]$. By \Cref{eq:trop-polar} we have
$$
\langle\omega_i^*,\omega_j^*\rangle=Q_\Gamma(\eta_i,\eta_j) = \matQ^{-1}[i,j] . 
$$
We see that the matrix representing the tropical polarization under $\{\omega_i^*\}_{1\le i\le g}$ is $\matQ^{-1}$. 

We summarize the computation from this step as pseudocode in \Cref{alg:cycle--edge}.

\begin{remark}
Let $\mathscr{C}$ be a compact Riemann surface. Suppose $\{\sigma_{i}\}_{1\le i\le 2g}$ is a basis of 1-cycles for $H_1(\mathscr{C};\mathbb{Z})$ and $\{\omega_i\}_{1\le i\le g}$ is a basis of holomorphic 1-forms for $\Omega(\mathscr{C})$. The $g\times 2g$ complex matrix defined by
$$
\mathbf{Z}[i,j] = \int_{\sigma_j}\omega_i 
$$
is called the \emph{period matrix} of $\mathscr{C}$ with respect to the bases $\{\sigma_{i}\}_{1\le i\le 2g}$ and $\{\omega_i\}_{1\le i\le g}$. It is known that there always exist bases for $H_1(\mathscr{C};\mathbb{Z})$ and $\Omega(\mathscr{C})$ such that the period matrix can be reduced to the form
$$
\mathbf{Z} = \begin{bmatrix}
    \mathbf{Z}_g & \mathbf{I}_g
\end{bmatrix} , 
$$
where $\mathbf{Z}_g$ is symmetric and its imaginary part $\mathrm{Im}(\mathbf{Z}_g)$ is positive definite. The matrix representing the canonical polarization on the (complex) Jacobian variety $\jac(\mathscr{C})$ is $\big(\mathrm{Im}(\mathbf{Z}_g)\big)^{-1}$. See \Cref{app:theory} for further background and details.

From \Cref{eq:inner-product-matrix}, we can check that
\begin{equation}\label{eq:trop-polar-int-def}
\matQ[i,j] = \int_{\sigma_i}\omega_j\, .    
\end{equation}
To align with the analogy to its complex geometry counterpart, we refer to the matrix $\matQ$ as the \emph{tropical period matrix} and to $\matQ^{-1}$ as the \emph{tropical polarization matrix}. 
\end{remark}

\begin{algorithm}[htbp]
\DontPrintSemicolon
\SetAlgoLined
\SetNoFillComment

\caption{Computation of cycle/path--edge incidence matrices and the tropical Abel--Jacobi transform of a combinatorial graph}
\label{alg:cycle--edge}

%\begin{algorithmic}[1]
    \KwIn{$V$: List of $n$ vertices,\, $E$: List of $m$ edges,\, $L$: List of $m$ lengths}
    \KwOut{$\matV$: $g\times n$ matrix,\, $\matQ$: $g\times g$ matrix}

    Sort $E$ such that the first $(n-1)$ edges form a spanning tree $\mathrm{ST}$\; 
    Initialize a $g\times (n-1)$ zero matrix $\matC_{\mathrm{ST}}$ \tcp*[r]{Compute the reduced cycle--edge incidence matrix}
    \For{$i \gets 1$ \KwTo $g$ }{ 
        Let $v_-,v_+$ be  vertices of  $E[n-1+i]$\; 
        Find the path $\gamma$ from $v_+$ to $v_-$ in $\mathrm{ST}$ via binary lifting\;
        \For{$j \gets 1$ \KwTo $(n-1)$}{
        Let $u_-,u_+$ be vertices of $E[j]$\; \uIf{$[u_-,u_+]\in\gamma$}{ $\matC_{\mathrm{ST}}[i,j]\gets 1$\;}\ElseIf{$[u_+,u_-]\in\gamma$}{ $\matC_{\mathrm{ST}}[i,j]\gets -1$\;
        }
        }
    }

    Let $\matL$ be the diagonal matrix constructed from $L$\;
    Initialize an $n\times (n-1)$ zero matrix $\matY_{\mathrm{ST}}$ \tcp*[r]{Compute the path--edge incidence matrix}
    \For{$i\gets 1$ \KwTo $n$}{ 
    %Find geodesic $\beta$ from $V[1]$ to $V[i]$ in $\mathrm{ST}$ 
    \For{$j\gets 1$ \KwTo $(n-1)$}{
    Let $u_-,u_+$ be vertices of $E[j]$\;
    \If{$u_-,u_+\in V[i].\mathrm{ancestors}$}{
    \uIf{$\mathrm{depth}(u_-)<\mathrm{depth}(u_+)$}{
        $\matY_{\mathrm{ST}}[i,j]\gets 1$\;
        }
        \Else{
        $\matY_{\mathrm{ST}}[i,j]\gets -1$\;
        }}
    
    }
    }
    \Return{$\matV = \matC_{\mathrm{ST}}\matL_{\mathrm{ST}}\matY_{\mathrm{ST}}^\top$,\, $\matQ = \matC_{\mathrm{ST}}\matL_{\mathrm{ST}}\matC_{\mathrm{ST}}^\top+\matL_g$}
%\end{algorithmic}
\end{algorithm}

\paragraph{Interpolation.}

In the third step, we compute the tropical Abel--Jacobi transform for points in $\Gamma\backslash V(G)$. Since any $p\in\Gamma\backslash V(G)$ is in the interior of some edge $e\in E(G)$, a na\"{i}ve way to compute $\ajmap(p)$ is to subdivide the edge $e$ and add $p$ to the vertex set of the refined combinatorial model and apply the same computation as in the second step. However, this approach is inefficient and requires high space complexity to store refined combinatorial graphs if the sample size is large. Instead, we will see that we can easily obtain $\ajmap(p)$ by interpolating vectors from $\ajmap(V(G))$: A point $p\in \Gamma\backslash V(G)$ is called a \emph{$\theta$-percentile point of $e$} for some $0<\theta<1$ if it satisfies $d(e_-,p)=\theta d(e_-,e_+)=\theta\ell(e)$.   We split our discussion into two cases:\\

\noindent
\textbf{Case 1}: $e\in\mathrm{ST}$.  In this case, we have
$$
\ajmap(e_+) = \ajmap(e_-)\pm \int_e\quad ,
$$
where $\ajmap(e_+)$ and $\ajmap(e_-)$ are computed from the second step based on the paths inside the spanning tree $\mathrm{ST}$. Without loss of generality, we assume
\begin{equation}\label{eq:path-plus}
  \ajmap(e_+) = \ajmap(e_-)+ \int_e\quad .  
\end{equation}
Let $\gamma_{e_-}$ be the path from the base point $q$ to $e_-$ in the spanning tree $\mathrm{ST}$. Concatenating with the path $\gamma_\theta:[0,\theta\ell(e)]\to \Gamma$, we obtain a path from the base point $q$ to $p$. Then we have
\begin{equation}\label{eq:inter-1}
   \ajmap(p) = \int_{\gamma_{e_-}}+\int_{\gamma_\theta}=\ajmap(e_-)+\int_{\gamma_\theta} \quad .
\end{equation}
To express $\int_{\gamma_\theta}$ with respect to the fundamental basis $\{\omega_i\}_{1\le i\le g}$, we compute
\begin{equation}\label{eq:inter-2}
    \int_{\gamma_\theta}\omega_i = \int_0^{\theta\ell(e)}\omega_i(\gamma_\theta'(t))\,\mathrm{d}t=\theta\int_0^{\ell(e)}\omega_i(\gamma_\theta'(t))\,\mathrm{d}t = \theta\int_e\omega_i \, . 
\end{equation}
From \cref{eq:path-plus,eq:inter-1,eq:inter-2}, we have
\begin{equation}\label{eq:inter-p-span}
\begin{aligned}
    \ajmap(p) &= \ajmap(e_-) + \theta\int_e\\
    &= \ajmap(e_-) + \theta\bigg(\ajmap(e_+)-\ajmap(e_-)\bigg)\\
    &= (1-\theta)\ajmap(e_-)+\theta\ajmap(e_+). 
\end{aligned}
\end{equation}
% Note that
% \begin{equation}\label{eq:inter-3}
%     \ajmap(e_-) = \int_{\gamma_{e_-}} =  \int_{\gamma_{e_+}}-\int_e
% \end{equation}
In terms of vectors, suppose $j_-,j_+$ are indices of $e_-,e_+$ in $V(G)$, then the vector representing $\ajmap(p)$ is
$$
(1-\theta)\matV[:,j_-] + \theta\matV[:,j_+] .
$$

\noindent
\textbf{Case 2}: $e\notin \mathrm{ST}$.  In this case, $e$ defines a fundamental 1-cycle whose orientation is consistent with $e$. Assume $e$ corresponds to $\sigma_i$. We construct a path to $p$ by first following along $\gamma_{e_-}$ and then the path $\gamma_\theta:[0,\theta\ell(e)]\to \Gamma$. Since $e$ does not intersect with any other 1-cycle $\sigma_j$ for $j\neq i$, we have
$$
\int_{\gamma_{\theta}}\omega_j= \theta\ell(e)\delta_{ij}\, ,
$$
Therefore,
\begin{equation}\label{eq:inter-p-not-span}
    \ajmap(p) = \ajmap(e_-)+\theta\ell(e)\omega_i^* \, . 
\end{equation}
In terms of vectors, let $\bm{1}_i$ be the indicator vector where the $i$th element is 1 and all other elements are 0, then the vector representing $\ajmap(p)$ is
$$
\matV[:,j_-] + \theta\ell(e)\bm{1}_i\, .
$$

A common scenario to sample points from $\Gamma$ is to equidistantly subdivide  edges in $E(G)$. Let $\kappa$ be the sampling ratio so that each edge is equally subdivided into $(\kappa+1)$ segments. We summarize the interpolation step as pseudocode in \Cref{alg:interpolation}.

\begin{algorithm}[htbp]
\caption{Sampling points from the tropical transformation of a metric graph via interpolation}
\label{alg:interpolation}

\DontPrintSemicolon
\SetAlgoLined
\SetNoFillComment

\KwIn{$\matV$:  $g\times n$ matrix,\, $E$: List of $m$ edges,\, $L$: List of $m$ lengths,\, $\kappa$: Sampling ratio}
\KwOut{$\matV$: $g\times (n-m+\kappa m)$ matrix}
Sort $E$ such that the first $(n-1)$ edges form a spanning tree $\mathrm{ST}$\;
\For{$j\gets 1$ \KwTo $(n-1)$}{
 Let $j_-,j_+$ be the indices of vertices of $E[j]$ \tcp*[r]{Interpolate points in the spanning tree}
 %Find the indices $j_-,j_+$ of $v_-,v_+$ in $\matV$\;
\For{$i\gets 1$ \KwTo $\kappa$}{    $\matV\gets \big[\matV,\, \big(1-\frac{i}{\kappa+1}\big)\matV[:,j_-]+\frac{i}{\kappa+1}\matV[:,j_+]\big]$\;}
    
}

\For{$j\gets n$ \KwTo $m$}{
 Initialize a zero column vector $\mathbf{w}$ \tcp*[r]{Interpolate points outside the spanning tree}
$\mathbf{w}[j-n+1]\gets L[j]$\;
Let $j_-,j_+$ be the indices of vertices of $E[j]$\;
%Find the index $j_-$ of $v_-$ in $\matV$\;
\For{$i\gets 1$ \KwTo $\kappa$}{
 $\matV\gets \big[\matV,\, \matV[:,j_-]+\frac{i}{\kappa+1}\mathbf{w}\big]$\;}
}

\Return{ $\matV$ }
\end{algorithm}

\paragraph{Computational Complexity.}

We now analyze the computational complexity of \Cref{alg:cycle--edge,alg:interpolation}. Let $\Gamma$ be a metric graph and $G$ be its combinatorial model. Suppose $G$ has $n_G$ vertices and $m_G$ edges. Let $g = m_G-n_G+1$ be the genus of $\Gamma$.
\begin{itemize}
    \item The computation of a fundamental basis is equivalent to finding a spanning tree of $G$. This can be done via a simple depth-first search (DFS) or breadth-first search (BFS) to traverse the graph, whose time complexity is $O(m_G)$. If we require the spanning tree to have minimal weights, there are numerous algorithms to compute minimal spanning trees (MST) of a graph, such as Bor\r{u}vka's algorithm, Kruskal's algorithm, and  Prim's algorithm, all of which have time complexity $O(m_G\log n_G)$ \citep{bollobas2013modern}.
    \item To compute the reduced cycle--edge incidence matrix, for each 1-cycle $\sigma_i$ determined by the edge $e_{n-1+i}$ not in the spanning tree $\mathrm{ST}$, we need to find the path within $\mathrm{ST}$ connecting the endpoints of $e_{n-1+i}$. This is equivalent to finding the lowest common ancestor (LCA) of the two vertices in $\mathrm{ST}$. A common technique to answer an LCA query is \emph{binary lifting}: we lift both vertices up to the same depth and then simultaneously lift them up until their ancestors match. The time complexity for this operation is $O(\log n_G)$ if the spanning tree is preprocessed with a binary lifting table \citep{leiserson1994introduction}. The reduced path--edge incidence matrix can be computed alongside the spanning tree by recording the traversal paths during DFS/BFS. Alternatively,  given a spanning tree $\mathrm{ST}$,  for each vertex $p_i$ and each edge $e_j$ in the spanning tree $\mathrm{ST}$, we check whether $e_j$ lies on the path from $q$ to $p_i$. This operation is equivalent to checking whether the endpoints of $e_j$ are ancestors of $p_i$, which can also be done in $O(\log n_G)$ time. Consequently, the time complexity for computing the reduced cycle--edge incidence matrix is $O(gn_G\log n_G)$, and the time complexity for computing the reduced path--edge incidence matrix is $O(n_G^2\log n_G)$.
    \item Once we have the reduced cycle--edge incidence matrix and the reduced path--edge incidence matrix, the computations of $\matV$ and $\matQ$ are matrix multiplications, thus the time complexity for computing $\matV$ is $O(gn_G^2)$, and the time complexity for computing $\matQ$ is $O(g^2n_G)$.
    \item In the interpolation step, the time complexity for adding two $g$-dimensional vectors is $O(g)$. Since we have $m_G$ edges and each edge has $\kappa$ points to interpolate, the total time complexity is therefore $O(g\kappa m_G)$. 
\end{itemize}

% In summary, \Cref{alg:cycle--edge} requires $O(gn_G+g^2)$ space to store matrices $\matV$ and $\matQ$, and requires $O(gn_G^2+g^2n_G)$ time for computation. \Cref{alg:interpolation} requires $O(gn_G+g\kappa m_G)$ space to store the matrix $\matV$, and requires $O(g\kappa m_G)$ time for computation.

\paragraph{An Example.}
    
\begin{figure}
    \centering
    \begin{subfigure}{0.45\linewidth}
    \centering
        \includegraphics[width=0.9\linewidth]{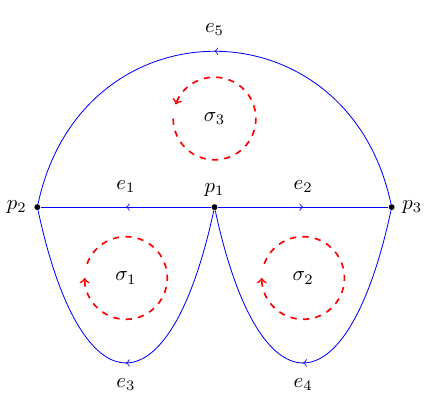}
        \caption{A metric graph and its combinatorial model}
        \label{subfig:metric-graph}
    \end{subfigure}
    \begin{subfigure}{0.45\linewidth}
    \centering
        \includegraphics[width=0.9\linewidth]{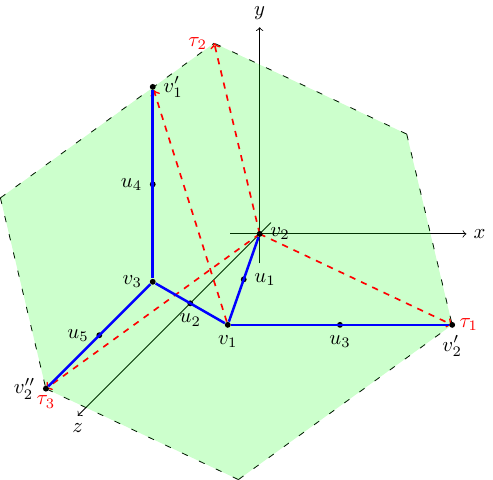}
        \caption{The tropical Abel--Jacobi transform}
        \label{subfig:trop-trans}
    \end{subfigure}
    \caption{A metric graph and its tropical Abel--Jacobi transform. On the left panel, a metric graph is represented by a combinatorial model with 3 vertices and 5 edges. The fundamental 1-cycles $\sigma_1,\sigma_2,\sigma_3$ are drawn in red dashed lines. On the right panel, the tropical Abel--Jacobi transform of the metric graph is the piecewise linear curve colored in blue. The $\tau_i$'s are lattice vectors corresponding to the $\sigma_i$'s. A fundamental domain of the tropical Jacobian is shaded in green. For the purpose of visualization, all $v_i$'s and $u_i$'s are translated by $[1,0,1]^\top$, so that they reside in the first quadrant and the fundamental domain.}
    \label{fig:embedding-3-circles}
    \end{figure}
    
    We demonstrate our algorithm on a toy example. Let $\Gamma$ be a metric graph of genus 3, represented by a combinatorial model $G$ shown in \Cref{subfig:metric-graph} with 3 vertices and 5 edges. All edges have unit length and their orientations are indicated by arrows. The first two edges $e_1,e_2$ define a spanning tree $\mathrm{ST}$, and the last three edges $e_3,e_4,e_5$ define three fundamental 1-cycles $\sigma_1,\sigma_2,\sigma_3$. Fix the base point at $p_1$. The cycle--edge incidence matrix and the reduced path--edge incidence matrix are given by

% $$
% \matC = 
% \begin{bNiceMatrix}[first-col,first-row]
%   & e_1  & e_2 & e_3 & e_4 & e_5 \\
% \sigma_1 & -1 & 0 & 1 & 0 & 0 \\
%   \sigma_2 & 0 & 1 & 0 & 1 & 0\\ 
% \sigma_3 & -1 & 1 & 0 & 0 & 1
% \end{bNiceMatrix}
% $$
    
\begin{equation*}
    \matC = 
    \begin{blockarray}{cccccc}
& e_1  & e_2 & e_3 & e_4 & e_5 \\
\begin{block}{c[ccccc]}
\bigstrut[t] \sigma_1 & -1 & 0 & 1 & 0 & 0 \\
  \sigma_2 & 0 & 1 & 0 & 1 & 0\\ 
\sigma_3 & -1 & 1 & 0 & 0 & 1 \bigstrut[b] \\
\end{block}
\end{blockarray}\, , \quad
\matY_{\mathrm{ST}} = 
    \begin{blockarray}{ccc}
&e_1  & e_2  \\
\begin{block}{c[cc]}
\bigstrut[t]\gamma_1 & 0 & 0 \\
  \gamma_2 & 1 & 0 \\ 
\gamma_3 & 0 & 1 \bigstrut[b]\\
\end{block}
\end{blockarray} . 
\end{equation*}
The tropical Abel--Jacobi transform for $p_1,p_2,p_3$ are computed by 
% $$
% \matY = 
% \begin{bNiceMatrix}[first-col,first-row]
%   &e_1  & e_2  \\
% \gamma_1 & 0 & 0 \\
%   \gamma_2 & 1 & 0 \\ 
% \gamma_3 & 0 & 1
% \end{bNiceMatrix}
% $$
\begin{equation*}
    \matV = \matC_{\mathrm{ST}}\matY_{\mathrm{ST}}^\top = \begin{bmatrix}
        -1 & 0\\
        0 & 1\\
        -1 & 1
    \end{bmatrix}\begin{bmatrix}
        0 &1 &0\\
        0 & 0 & 1
    \end{bmatrix} =
    \begin{blockarray}{ccc}
v_1  & v_2 & v_3 \\
\begin{block}{[ccc]}
 0 & -1 & 0 \\
  0 & 0 & 1 \\ 
0 & -1 & 1 \\
\end{block}
\end{blockarray} . 
\end{equation*}
The tropical period matrix is given by
\begin{equation*}
    \matQ = \matC_{\mathrm{ST}}\matC_{\mathrm{ST}}^\top + \mathbf{I}_3 = \begin{bmatrix}
        -1 & 0\\
        0 & 1\\
        -1 & 1
    \end{bmatrix}\begin{bmatrix}
        -1 & 0 & -1\\
        0 & 1 & 1\\
    \end{bmatrix} +\begin{bmatrix}
        1 & 0 & 0\\
        0 & 1 & 0\\
        0 & 0 & 1
    \end{bmatrix} = \begin{blockarray}{ccc}
\tau_1  & \tau_2 & \tau_3 \\
\begin{block}{[ccc]}
 2 & 0 & 1 \\
  0 & 2 & 1 \\ 
1 & 1 & 3 \\
\end{block}
\end{blockarray},
\end{equation*}
where the column vectors $\tau_1,\tau_2,\tau_3$ form a basis for the lattice $H_1(\Gamma;\mathbb{Z})$. The fundamental domain corresponding to $\tau_1,\tau_2,\tau_3$ is shown as the green region in \Cref{subfig:trop-trans}.

% Let $\omega_i^*$ be the dual basis of $\omega_i$. For each point $q_i$, the vectorization is given by the coefficients of $\int_p^{q_i}$ with respect to $\omega_i^*$'s. By algorithm (cite here), it suffices to consider the coefficients of $e_i$ in the cycles $\sigma_i$'s. For example, since the coefficent of $e_1$ in in $\sigma_1$ and $\sigma_3$ are $1$, the vectorization of $q_1$ is given by
% $$
% v_1 = (1,0,1)
% $$

To sample additional points from the tropical Abel--Jacobi transform of $\Gamma$, we apply \Cref{alg:interpolation} with a sampling ratio $\kappa=1$:  
\begin{itemize}
    \item For $e_1,e_2$, the vectors for midpoints are given by
    $$
    u_1 = \frac{1}{2}(v_1+v_2) = \bigg[-\frac{1}{2},0,-\frac{1}{2}\bigg]^\top,\quad u_2 = \frac{1}{2}(v_1+v_3) = \bigg[0, \frac{1}{2},\frac{1}{2}\bigg]^\top .
    $$
    \item For $e_3,e_4,e_5$, the vectors for midpoints are given by
    $$
    [u_3, u_4, u_5]= [v_1, v_3,v_3] +\frac{1}{2}\mathbf{I}_3 = \begin{bmatrix}
        \frac{1}{2} & 0 &0 \\
        0& \frac{3}{2} & 1\\
        0& 1 & \frac{3}{2}
    \end{bmatrix} .
    $$
\end{itemize}
% $$
% \begin{aligned}
%     v_2 = (2,0,2)\quad& v_3 = (2,1,3)\\
%     v_4 = (2,2,4)\quad& v_5 = (3,0,2)\\
%     v_6 = (2,3,4)\quad& v_7 = (2,2,5)
% \end{aligned}
% $$
% The lattice is spanned by the vectorization of $\int_{\sigma_i}$, denoted by $\tau_i$. By algorithm (cite here), it suffices to check the intersection of pairs $\sigma_i$ and $\sigma_j$. 
% $$
% \begin{aligned}
%     \tau_1 =& (4,0,2)\\
%     \tau_2 =& (0,4,2)\\
%     \tau_3 =& (2,2,6)
% \end{aligned}
% $$
% The embedded curve in the tropical Jacobian is piecewise-linear and has the same topology as the original metric graph.

The visualization for $u_i$'s can be found in \Cref{subfig:trop-trans}. 

\subsection{Changing Vector Representations}\label{subsec:change-vectors}

We now show how the matrices $\matV$ and $\matQ$ change with respect to the choice of combinatorial models $G$ for a given metric graph $\Gamma$. Specifically, we will consider the change of spanning trees, orientations, and edge refinements.

\paragraph{Changing Spanning Trees.}

The choice of spanning tree determines how the cycle--edge and path--edge incidence matrices are constructed. To analyze its effect, we first consider changing the root while keeping the same spanning tree, followed by fixing the root and altering the spanning tree. The resulting changes in the matrices 
$\matV$ and $\matQ$ are discussed below.

\begin{proposition}
    Let $\Gamma$ be a metric graph and fix a combinatorial model $G$ for $\Gamma$. Let $\mathrm{ST}$ be a spanning tree at root $q$, and suppose $\matV$ and $\matQ$ are matrices defined in \Cref{eq:aj-mat-1,eq:inner-product-matrix} corresponding to $\mathrm{ST}$.
    \begin{enumerate}[(i)]
        \item Change the root of $\mathrm{ST}$ to another vertex $q'\neq q$. Suppose $\widetilde{\matV}$ and $\widetilde{\matQ}$ are matrices corresponding to $q'$. Then $\widetilde{\matQ} = \matQ$ and there exists a column vector $\veca\in\mathbb{R}^g$ such that
        $$
        \widetilde{\matV} = \matV + \veca\bm{1}^\top ,
        $$
        where $\bm{1}$ is the column vector of all 1's. 
        \item Let $\mathrm{ST}'$ be another spanning tree at the same root as $\mathrm{ST}$. Suppose $\widetilde{\matV}$ and $\widetilde{\matQ}$ are matrices corresponding to $\mathrm{ST}'$. Then there exists an invertible matrix with integer entries $\mathbf{P}\in\mathrm{GL}(g;\mathbb{Z})$ such that
        \begin{equation}\label{eq:Q-change}
        \widetilde{\matQ} = \mathbf{P}\matQ\mathbf{P}^\top ,
        \end{equation}
        and there exists another matrix with integer entries $\matB\in\mathbb{Z}^{n_G\times g}$ such that
        \begin{equation}\label{eq:P-change}
        \mathbf{P}^{-1}\widetilde{\matV} = \matV+\matQ\matB^\top .
        \end{equation}
    \end{enumerate}
\end{proposition}

\begin{proof}
We now prove the statements.
    \begin{enumerate}[(i)]
        \item By construction, the cycle--edge incidence matrix and the edge length matrix are invariant under change of roots. Thus, $\widetilde{\matC}=\matC$, $\widetilde{\matL} = \matL$, and 
        $$
        \widetilde{\matQ} = \widetilde{\matC}\widetilde{\matL}\widetilde{\matC}^\top = \matQ . 
        $$
        Let $\gamma_{p',p}$ be the path from $p'$ to $p$ in $\mathrm{ST}'$ and consider the vector $\vecy\in\mathbb{R}^{n_G-1}$ given by $\vecy[j] = \gamma_{p',p}(e_j)$. Then we have
        $$
        \widetilde{\matY} = \matY + \bm{1}\vecy^\top ,
        $$
        and
        $$
        \widetilde{\matV} = \widetilde{\matC}\widetilde{\matL}\widetilde{\matY}^\top = \matC\matL(\matY^\top+\vecy\bm{1}^\top) = \matV + \matC\matL\vecy\bm{1}^\top = \matV +\veca\bm{1}^\top ,
        $$
        where we set $\veca = \matC\matL\vecy$.
        \item Let $\sigma_1,\ldots,\sigma_g$ be the fundamental 1-cycles determined by $\mathrm{ST}$ and $\widetilde{\sigma}_1,\ldots,\widetilde{\sigma}_g$ be the fundamental 1-cycles determined by $\mathrm{ST}'$. Then there exists an invertible matrix with integer entries $\mathbf{P}\in\mathrm{GL}(g;\mathbb{Z})$ given by the change of homology bases, 
        $$
        \widetilde{\sigma}_i = \sum_{j=1}^g\mathbf{P}[i,j]\sigma_j .
        $$
        For the cycle--edge incidence matrix, we have 
        \begin{equation}\label{eq:change-C}
        \widetilde{\matC} = \mathbf{PC} .
        \end{equation}
        As a result
        $$
        \widetilde{\matQ} = \widetilde{\matC}\widetilde{\matL}\widetilde{\matC}^\top= (\mathbf{PC})\matL(\matC^\top\mathbf{P}^\top)  =\mathbf{PQ}\mathbf{P}^\top .
        $$
        For any $p_j\in V(G)$, let $\gamma_j\in\mathrm{ST}$ and $\gamma_j'\in\mathrm{ST}'$ be the paths from $q$ to $p_j$. Then $\gamma_j'-\gamma_j$ is an integral linear combination of $\sigma_1,\ldots,\sigma_g$. In terms of matrices, there exists an matrix with integer entries $\matB\in\mathbb{Z}^{n_G\times g}$ such that
        \begin{equation}\label{eq:change-Y}
        \widetilde{\matY} - \matY= \matB\matC . 
        \end{equation}
        Therefore by \Cref{eq:change-C,eq:change-Y}
        $$
        \begin{aligned}
            \widetilde{\matV} =\widetilde{\matC}\widetilde{\matL}\widetilde{\matY}^\top 
             = (\mathbf{PC})\matL(\matY+ \matB\matC)^\top
             = \mathbf{PV} + \mathbf{PQ}\matB^\top ,
        \end{aligned}
        $$
        which proves the claim.
    \end{enumerate}
\end{proof}
\begin{remark}
% The first property is true in general that changing the base point of the tropical Abel--Jacobi map results in a translation in the tropical Jacobian. In fact, let $\mathcal{J}_q$ and $\mathcal{J}_{q'}$ be the tropical Abel--Jacobi maps based at $p$ and $p'$. Let $\gamma_{p',p}$ be any path from $p'$ to $p$. Then for any $q\in\Gamma$,
% $$
% \ajmap_{p'}(q) = \int_{\gamma_{p',q}} = \int_{\gamma_{p',p}}+\int_{\gamma_{p,q}} = \int_{\gamma_{p',p}}+\ajmap_p(q)
% $$
% In terms of vectors, the translation is represented by $\matA$ in the proof.

In the second statement (ii), note that a change of spanning tree results in a change of bases of both lattices $H_1(\Gamma;\mathbb{Z})$ and $\Omega^*_{\mathbb{Z}}(\Gamma)$ of the tropical Jacobian $\jac(\Gamma)$. By construction, if 
$$
\widetilde{\sigma}_i = \sum_{j=1}^g\mathbf{P}[i,j]\sigma_j \,,
$$
then the bases of $\Omega_{\mathbb{Z}}(\Gamma)$ and $\Omega^*_{\mathbb{Z}}(\Gamma)$ change as 
$$
\widetilde{\omega}_i = \sum_{j=1}^g\mathbf{P}[i,j]\omega_j\, , \quad \widetilde{\omega}_i^* = \sum_{j=1}^g\mathbf{P}^{-1}[i,j]\omega^*_j \, .
$$
By \Cref{eq:trop-polar-int-def}, the tropical period matrix is such that
$$
\widetilde{\matQ}[i,j] = \int_{\widetilde{\sigma}_j}\widetilde{\omega}_i = \sum_{r,s}\matP[j,r]\matQ[r,s]\matP[i,s] ,
$$
which is consistent with \Cref{eq:Q-change}. For vector representations of $\ajmap(V(G))$, if the path to each vertex is fixed, then $\widetilde{\matV}$ and $\matV$ simply differ by a coordinate change 
$$
\mathbf{P}^{-1}\widetilde{\matV}=\matV .
$$
However in our algorithm the path to each vertex is also dependent on the spanning tree. Thus there is an additional term in \Cref{eq:P-change} indicating the induced change of paths. 
\end{remark}

\paragraph{Refinement of Combinatorial Models.}

We show how the matrices $\matV$ and $\matQ$ change with respect to refinement of combinatorial models. Given two combinatorial models $G$ and $G'$, changing the orientation if necessary, it is always possible to find a ``larger'' combinatorial model which contains $G$ and $G'$ as subgraphs by iteratively subdividing  edges. Thus it suffices to consider change of combinatorial models by subdividing one edge at a time. 

First, we show that the matrices $\matV$ and $\matQ$ are independent of edge orientations.

\begin{proposition}
    Let $G$ be a combinatorial model for the metric graph $\Gamma$. For any $e\in E(G)$, let $\widetilde{\matV}$ and $\widetilde{\matQ}$ be the matrices obtained by changing the orientation of $e$. Then $\widetilde{\matV}=\matV$ and $\widetilde{\matQ}=\matQ$.
\end{proposition}
\begin{proof}
    Let $e_j\in E(G)$. Changing the orientation of $e_j$ yields
    $$
    \widetilde{\matC}[:,j] = -\matC[:,j], \quad \widetilde{\matY}[:,j] = -\matY[:,j] .
    $$
    When multiplying two matrices, the minus signs cancel, which implies
    $$
    \widetilde{\matV} = \widetilde{\matC}\matL\widetilde{\matY}^\top = \matV,\quad \widetilde{\matQ} = \widetilde{\matC}\matL\widetilde{\matC}^\top = \matQ .
    $$
\end{proof}

In terms of structure and data encoded by the graphs, subdividing an edge means adding a new vertex to $V(G)$ and replacing the old edge in $E(G)$ with two new edges. Computing $\matV$ and $\matQ$ by \Cref{eq:aj-mat-1,eq:inner-product-matrix} is equivalent to applying \Cref{alg:cycle--edge} to a new combinatorial model. As discussed in the description of \Cref{alg:interpolation}, interpolation offers a far more efficient method. Here we verify that both approaches indeed yield the same result, which, in a sense, proves the correctness of \Cref{alg:interpolation}. The main idea of the proof is to arrange the order of the edges in the new model in a proper way and check that computations from both approaches agree and give the same result. Due to its lengthy computation, the full proof is deferred to \Cref{app:proof-linear-alg}.  

\begin{theorem}\label{thm:subdivision}
    Let $\Gamma$ be a metric graph and $G$ be a combinatorial model for $\Gamma$. Fix a spanning tree $\mathrm{ST}$ for $G$, and order the edges in $E(G)$ such that the first $(n_G-1)$ edges correspond to $\mathrm{ST}$. Subdivide an edge $e_j\in E(G)$ into two edges and let $G'$ be the induced combinatorial model. Suppose the new vertex is a $\theta$-percentile point of $e_j$ and is indexed by $(n_G+1)$ in $V(G')$. Let $\widetilde{\matV}$ and $\widetilde{\matQ}$ be the output of \Cref{alg:cycle--edge} for $G'$. Then $\widetilde{\matQ} =  \matQ$, and
    \begin{enumerate}[(i)]
        \item If $e_j\in \mathrm{ST}$, let $j_-,j_+$ be the indices of endpoints of $e_j$. Then $\widetilde{\matV}[:,i]=\matV[:,i]$ for $i=1,\ldots,n_G$, and 
        $$
        \widetilde{\matV}[:,n_G+1] = (1-\theta)\matV[:,j_-]+\theta\matV[:,j_+] .
        $$
        \item If $e_j\notin\mathrm{ST}$, let $\mathbf{w}\in\mathbb{R}^g$ be the column vector whose only nonzero entry is  $\mathbf{w}[j-n_G+1] = \ell(e_j)$. Then $\widetilde{\matV}[:,i]=\matV[:,i]$ for $i=1,\ldots,n_G$, and either
$$
\widetilde{\matV}[:,n_G+1] = \matV[:,j_-] +\theta \mathbf{w} , 
$$
or
$$\widetilde{\matV}[:,n_G+1] = \matV[:,j_+]-(1-\theta)\mathbf{w} ,
$$
depending on which new edge is added to the spanning tree in $G'$.
    \end{enumerate}
\end{theorem}

%\begin{definition}
% Let $G$ and $G'$ be combinatorial models of $\Gamma$. Then $G'$ is a refinement of $G$ if $V(G)\subseteq V(G')$ and for any $e\in E(G)$, there exists $p_0=e_-,p_1,\ldots,p_k=e_+$ such that $[p_i,p_{i+1}]\in E(G')$ and $\sum_{i=0}^{k-1}\ell([p_i,p_{i+1}]) = \ell(e)$.  
% \end{definition}

\subsection{Properties of the Tropical Abel--Jacobi Transform}\label{subsec:property}

We now explore the tropical Abel--Jacobi transform from a computational perspective. We begin by proving several properties about the cycle--edge incidence matrix. Then we demonstrate how these properties can be leveraged to design algorithms to simplify combinatorial models of a metric graph.   Finally, we show that the tropical Abel--Jacobi map can be viewed as an orthogonal projection on the 1-chain space, highlighting its metric properties.

\paragraph{Properties of the Cycle--Edge Incidence Matrix.}

Though the definition of the cycle--edge incidence matrix is straightforward, we show that important information about the tropical Abel--Jacobi transform can be directly extracted from this matrix.  

\begin{definition}
    Let $G$ be a connected combinatorial graph. An edge $e\in E(G)$ is called a \emph{bridge} if $G-e$ is disconnected. A combinatorial graph $G$ is \emph{2-connected} if it does not contain any bridge. A metric graph $\Gamma$ is 2-connected if its combinatorial models are 2-connected.
\end{definition} 

We can read off bridges from the cycle--edge incidence matrix of a combinatorial model.

\begin{proposition}
    Let $G$ be a combinatorial model for a metric graph $\Gamma$. An edge $e_j\in E(G)$ is a bridge if and only if $\matC[:,j]=\bm{0}$.
\end{proposition}

\begin{proof}
    By construction $\matC[:,j]=\bm{0}$ if and only if $e_j$ is not part of any fundamental 1-cycle. This is equivalent to $G-e_j$ being disconnected, which in turn is equivalent to $e_j$ being a bridge. 
\end{proof}

\begin{corollary}{\citep[Theorem 4.1(3), Vector Version]{baker2011metric}}\label{coro:bridge}
    Let $e_j$ be a bridge of $G$ and let $j_-,j_+$ be the indices of the endpoints of $e_j$. Then $\matV[:,j_-]=\matV[:,j_+]$. 
\end{corollary}

\begin{proof}
    Let $\bm{1}_j$ be the row indicator vector whose only nonzero element is 1 at the $j$th position. Since $e_j$ is a bridge, we have 
    $$
    \matY[j_+,:] = \matY[j_-,:] \pm \bm{1}_j \,,
    $$
    which implies that
    $$
    \matV[:,j_+]  = \matC\matL\matY[j_+,:]^\top = \matV[:,j_-] + \matC\matL\bm{1}_j^\top = \matV[:,j_-] \pm \ell(e_j)\matC[:,j]=\matV[:,j_-] .
    $$
\end{proof}

As a result, we see that bridge edges can be omitted from the computation of the tropical Abel--Jacobi transform. This property enables us to simplify a combinatorial model by contracting bridge edges, which we will elaborate on later.

For a combinatorial graph $G$ with $n_G$ vertices and $m_G$ edges, let $\matA$ be the $m_G\times n_G$ matrix representing the boundary map $\partial:C_1(G;\mathbb{Z})\to C_0(G;\mathbb{Z})$.

\begin{theorem}
    {\citep[Theorem 4.1(5)(6), Vector Version]{baker2011metric}}\label{thm:pushforward-balance}
     Let $G$ be a combinatorial model for a metric graph $\Gamma$. Then
     \begin{enumerate}[(i)]
         \item For any $e_j\in E(G)$, $\matC[:,j]$ represents the pushforward of the unit tangent vector along $e_j$ under the tropical Abel--Jacobi map.
         \item The balancing condition at $\matV[:,j]$ is represented by the equality
         $$
         \matC\matA[:,j] = \bm{0} .
         $$
     \end{enumerate}
\end{theorem}

\begin{proof}
We now prove the statements
\begin{enumerate}[(i)]
    \item     Let $p$ be an endpoint of $e_j$ and let $\gamma:[0,\ell(e_j)]\to\Gamma$ be a unit speed geodesic starting from $p$. The pushforward of $\gamma'(0)$ is given by
    $$
    (\ajmap)_{*p}(\gamma'(0)) = \odv{\ajmap(\gamma(t))}{t}_{t=0}^{} = \odv{}{t}_{t=0}^{}\bigg(\ajmap(p)+\int_{\gamma(t)}\bigg) = \odv{}{t}_{t=0}^{}\bigg(\frac{t}{\ell(e_j)}\int_{e_j}\bigg) = \frac{1}{\ell(e_j)}\int_{e_j}.
    $$
    When represented under the fundamental basis $\{\omega_i^*\}_{1\le i\le g}$, the $i$th coefficient is given by
    $$
    \frac{1}{\ell(e_j)}\int_{e_j}\omega_i =  \frac{1}{\ell(e_j)}\int_{e_j}\sum_{k=1}^{m_G}\sigma_i(e_k)\,\mathrm{d}t_{e_k} = \sigma_i(e_j) = \matC[i,j] ,
    $$
    which proves the claim.
    \item Since each $\sigma_i$ is a homology cycle, by definition,
    $$
    \partial\sigma_i = \sum_{k=1}^{m_G}\sigma_i(e_k)\partial e_k = 0 .
    $$
    In matrix form, it is equivalent to $\matC\matA = \bm{0}$. Since each column of $\matC$ represents the pushforward of a unit tangent vector, at $\matV[:,j]$ we have
    $$
    \matC\matA[:,j] = \matC[:,1]\matA[1,j]+\cdots+\matC[:,m(G)]\matA[m(G),j] = \bm{0} ,
    $$
    which means that these pushforward tangent vectors are balanced.
\end{enumerate}
\end{proof}

It is worth noting that seemingly trivial facts about the cycle--edge incidence matrix reveal important tropical properties. In particular, \Cref{thm:pushforward-balance} implies that the tropical Abel--Jacobi map is, in fact, a \emph{tropical map}, which is one of the central results proved by \cite{baker2011metric}.  

\paragraph{Simplifying Combinatorial Models.}

\begin{figure}[htbp]
\centering
    % \begin{subfigure}{0.45\linewidth}
    %    \centering 
    %    \includegraphics[width=0.9\linewidth]{figs/before-sim.pdf}
    %    \caption{Original model}
    % \end{subfigure}
    % \begin{subfigure}{0.45\linewidth}
    %    \centering 
    %    \includegraphics[width=0.77\linewidth]{figs/after-sim.pdf}
    %    \caption{After preprocessing}
    % \end{subfigure}
    \includegraphics[width=0.9\linewidth]{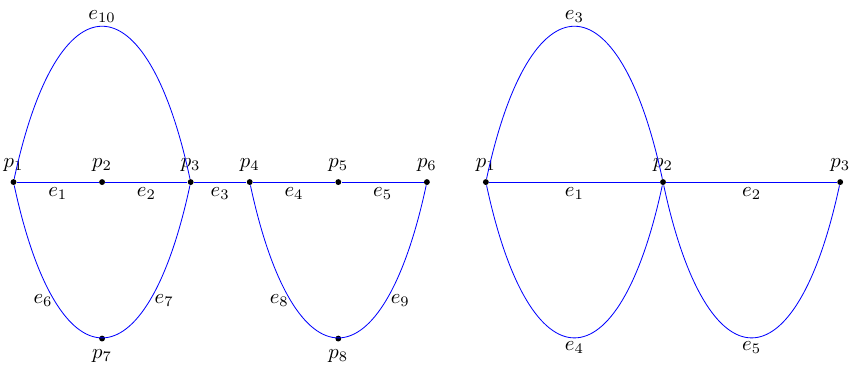}
    \caption{Illustration of Simplification. On the left panel, a combinatorial model with 8 vertices and 10 edges is shown. On the right panel, after contracting bridge edges and deleting vertices of valence 2, the simplified combinatorial model has only 3 vertices and 5 edges.}
    \label{fig:preprocessing}
\end{figure}

Based on \Cref{thm:subdivision,coro:bridge}, we can simplify a combinatorial model by contracting bridge edges and removing vertices of valence 2, thereby improving the computational efficiency of the tropical Abel--Jacobi transform. Specifically, we preprocess a combinatorial model $G$ through the following steps:

\begin{enumerate}
    \item \emph{Contracting bridges}: If $e$ is a bridge of $G$, then it collapses to a single point under the tropical Abel--Jacobi map and thus does not contribute to the computation. As a result, we can traverse the combinatorial model $G$ and find all bridges of $G$, which can be done in $O(m_G)$ time using Tarjan's bridge-finding algorithm \citep{tarjan1974note}. Let $e_-,e_+$ be the endpoints of $e$. By contracting $e$, we remove $e_+$ from the vertex set $V(G)$, and replace all edges incident to $e_+$ with connections to $e_-$. Contracting all bridges in $G$ results in a 2-connected combinatorial model.  
    \item \emph{Removing vertices of valence 2}: If $p$ is a vertex of valence 2, then its tropical Abel--Jacobi transform can be computed via interpolation. To remove $p$ from the combinatorial model $G$, let $e_1,e_2$ be two edges incident to $p$. We first remove $p$ from $V(G)$ and $e_2$ from $E(G)$. Let $p'$ be the other endpoint of $e_2$. Then replace $p$ with $p'$ in $e_1$ and update the length of $e_1$ by $\ell(e_1)+\ell(e_2)$. For a metric graph of genus $g\ge 2$, all vertices of valence 2 can be removed. In the special case of $g=1$, however, one vertex should be retained to avoid leaving $V(G)$ empty.
\end{enumerate}

We summarize the procedure of simplification in \Cref{alg:preprocess}. We also present an example in \Cref{fig:preprocessing} to illustrate the algorithm. 

\begin{remark}
    The minimal combinatorial model of a metric graph corresponds to a maximal cell of the moduli space of tropical curves. See \cite{chan2012tropical} for an explicit computation of the moduli space of tropical curves up to genus $g=5$.
\end{remark}

\begin{algorithm}[htbp]
\caption{Finding the minimal combinatorial model for a metric graph}
\label{alg:preprocess}

\DontPrintSemicolon
\SetAlgoLined
\SetNoFillComment

\KwData{$V$: List of $n$ vertices,\, $E$: List of $m$ edges,\, $L$: List of $m$ lengths}

Find and store bridges in list $\mathrm{Br}$\;
\For{$e\in \mathrm{Br}$}{
$E\gets E-\{e\}$\tcp*[r]{Contracting bridges}
$L\gets L-\{\ell(e)\}$\;
Let $e_-,e_+$ be the endpoints of $e$\; 
$V\gets V-\{e_+\}$\;
Let $\mathrm{N}(e)$ be the list of edges incident to $e_+$\;
\For{$e'\in \mathrm{N}(e)$}{    
\uIf{$e'_-=e_+$}{
$e'_-\gets e_-$\;
}\Else{
$e'_+\gets e_-$\;
} }
}

Find and store all vertices of valence 2 in list $\mathrm{Vt}$\;
\For{$v\in \mathrm{Vt}$}{
$V\gets V-\{v\}$ \tcp*[r]{Deleting vertices of valence 2}
Let $e_1,e_2$ be the edges incident to $v$\;
$\ell(e_1)\gets \ell(e_1)+\ell(e_2)$\;
$e_{1+}\gets e_{2+}$\;
$E\gets E-\{e_2\}$\;
$L\gets L-\{\ell(e_2)\}$
}

\Return{$V,E,L$}
\end{algorithm}

\paragraph{The Tropical Abel--Jacobi Map as an Orthogonal Projection.}

For a combinatorial graph $G$, consider the inner product $Q_G$ on the 1-chain space $C_1(G;\mathbb{R})$ as defined in \cref{eq:inner-product}. Under the standard basis $\{e:e\in E(G)\}$, the matrix representing $Q_G$ is precisely the edge length matrix $\matL$. Fix a basis $\{\sigma_i\}_{1\le i\le g}$ of fundamental 1-cycles for $H_1(G;\mathbb{Z})$. Since $H_1(G;\mathbb{R})$ is a closed subspace of $C_1(G;\mathbb{R})$, the inner product $Q_G$ restricted to $H_1(G;\mathbb{R})$ is well-defined, and the matrix representing $Q_G$ on $H_1(G;\mathbb{R})$ is given by
$$
Q_G(\sigma_i,\sigma_j) = \matQ[i,j] ,
$$
which is precisely the tropical period matrix. Let $\pi:C_1(G;\mathbb{R})\to H_1(G;\mathbb{R})$ be the orthogonal projection map with respect to $Q_G$. The \emph{Albanese torus} of $G$ is defined as the real torus $\alb(G) = H_1(G;\mathbb{R})/H_1(G;\mathbb{Z})$ together with the flat Riemannian metric induced by $Q_G$ \citep{kotani2000jacobian}. We can define an Abel--Jacobi-like map via the orthogonal projection map: Fix a base point $q\in G$. For any point $p\in G$, let $\gamma_{q,p}$ be a path from $q$ to $p$. Then $\pi(\gamma_{q,p})$ is an element in $H_1(G;\mathbb{R})$. Modulo the lattice $H_1(G;\mathbb{Z})$, we obtain a well-defined map from $G$ to its Albanese torus $\alb(G)$ as the following
\begin{equation}\label{eq:ortho-map}
   \begin{aligned}
    \widehat{\ajmap}: G &\to \alb(G)\\
    p&\mapsto \pi(\gamma_{q,p}) \quad\big(\mathrm{mod}\, H_1(G;\mathbb{Z})\big).
\end{aligned} 
\end{equation}
The definition of $\widehat{\ajmap}$ is independent of the choice of combinatorial models and extends well to a map from a metric graph $\Gamma$ to its Albanese torus $\alb(\Gamma)$. This raises the question of whether there is a connection between the two Abel--Jacobi maps $\ajmap$ and $\widehat{\ajmap}$. In particular, from a computational viewpoint, we ask what is the relationship between the vector representations of  $\ajmap(V(G))$ and $\widehat{\ajmap}(V(G))$? Assume that 
$$
\pi(e_j) = \pi_{1j}\sigma_1+\cdots+\pi_{gj}\sigma_g .
$$
Let $\bm{\Pi}$ be the $g\times m$ matrix whose entries are give by $\pi_{ij}$. The following proposition shows that the projection matrix $\bm{\Pi}$ can be expressed by matrices $\matC, \matL$, and $\matQ$.

\begin{proposition}
Let $G$ be a combinatorial graph. Let $\bm{\Pi}$ be the matrix representing the orthogonal projection map $\pi:C_1(G;\mathbb{R})\to H_1(G;\mathbb{R})$ with respect to the bases $\{e_i\}_{1\le i\le m_G}$ and $\{\sigma_i\}_{1\le i\le g}$. Then
\begin{equation}\label{eq:ortho-matrix}
   \bm{\Pi} = \matQ^{-1}\mathbf{CL} .
\end{equation}
\end{proposition}

\begin{proof}
For any path $\gamma$ in $G$, let $\vecy\in\mathbb{R}^{m_G}$ denote the column vector such that $\vecy[j] = \gamma(e_j)$. Then the orthogonal projection of $\vecy$ is  the unique minimizer of the following objective function
$$
\begin{aligned}
f(\mathbf{x}) &= (\vecy-\matC^\top\mathbf{x})^\top\matL(\vecy-\matC^\top\mathbf{x})\\
&= \vecy^T\matL\vecy-2\mathbf{x}^\top\matC\matL\vecy+\mathbf{x}^\top\matC\matL\matC^\top\mathbf{x} .
\end{aligned}
$$
Setting $\nabla f(\mathbf{x}) = 0$, we have 
$$
\bm{\Pi}\vecy = \mathbf{x}^* = \big(\mathbf{CLC}^\top\big)^{-1}\mathbf{CL}\vecy = \matQ^{-1}\mathbf{CL}\vecy .
$$
Therefore, the orthogonal projection matrix is  $\bm{\Pi} = \matQ^{-1}\mathbf{CL}$.
\end{proof}

\begin{corollary}\label{coro:projection}
    Let $G$ be a combinatorial model for a metric graph $\Gamma$.  Then $\matV=\matQ\bm{\Pi}\matY^\top$, i.e.,
    the vector representations of $\ajmap(V(G))$ and $\widehat{\ajmap}(V(G))$ differ by a linear transformation $\matQ$.
\end{corollary}

Based on \Cref{eq:ortho-map,eq:ortho-matrix}, a natural question to ask is whether we can generalize the tropical Abel--Jacobi map by placing an arbitrary inner product on $C_1(G;\mathbb{R})$ other than the one given by $Q_G$.  In fact, a more ``reasonable'' choice is to define an inner product $Q'_G$ such that $Q'_G(e_i,e_i) = \ell(e_i)^2$ so that each 1-chain $e_i$ has norm $\ell(e)$ rather than $\sqrt{\ell(e_j)}$ \citep{ji2012complete}. However, not all inner products allow us to apply the interpolation algorithm to metric graphs, because with different inner products, the image of $\widehat{\ajmap}$ can fail to be piecewise linear. This highlights the critical role of tropical polarization in defining the tropical Jacobian and the tropical Abel--Jacobi map.

%%%%%%%%%%%%%%%%%%%%%%%%%%%%%%%%%%%%%%%%%%%%%%%%%%%
\section{Computing Distances on the Tropical Jacobian}\label{sec:metrics}

The tropical Jacobian of a metric graph is naturally associated with tropical distance functions. Given a list of vectors from the tropical Abel--Jacobi transform of a metric graph, we aim to compute the pairwise distances out of the point cloud data. In this section, we begin by defining two distance functions on the tropical Jacobian: the tropical polarization distance and the Foster--Zhang distance, which are associated distance functions of the corresponding Riemannian metric and Finsler metric defined in \cite{baker2011metric}. Next, we show that computing distances on the tropical Jacobian is equivalent to solving classical lattice problems in computational complexity theory and cryptology, which are known to be NP-hard. We illustrate the hardness of lattice problems and prove that a closed-form computation for the tropical polarization distance only exists for a specific class of metric graphs. Then we utilize existing results from lattice basis reduction to compute truncated tropical polarization distances as computationally tractable alternatives. Finally, we present numerical experiments to illustrate our methods of computation and approximation in practice.

\subsection{Distance Functions on the Tropical Jacobian}\label{subsec:dist-func}

We begin by defining and presenting two distance functions on the tropical Jacobian: the tropical polarization distance and the Foster--Zhang distance.

\paragraph{The Tropical Polarization Distance.}

Let $\Gamma$ be a metric graph of genus $g$. By fixing a fundamental basis $\{\omega^*_i\}_{1\le i\le g}$ of $\Omega^*_{\mathbb{Z}}(\Gamma)$, we can identify the tropical Jacobian $\jac(\Gamma)$ with  $\mathbb{R}^g/\mathbb{L}$, where the lattice $\mathbb{L}$ is generated by the column vectors of the tropical period matrix $\matQ$. The \emph{tropical polarization} is an inner product on $\Omega^*(\Gamma)$ whose matrix representation is given by $\matQ^{-1}$. Thus for any vectors $\vecx,\vecy\in\mathbb{R}^g$, the distance function on $\jac(\Gamma)$ induced by the tropical polarization is 
\begin{equation}\label{eq:trop-eucli-1}
        d_{\Trop}([\vecx],[\vecy]) = \min_{\vecn\in\mathbb{Z}^g}\left((\vecx-\vecy-\matQ\vecn)^\top\matQ^{-1}(\vecx-\vecy-\matQ\vecn)\right)^{\frac{1}{2}} .
\end{equation}
The linear transformation defined by $\matQ^{-1}$ carries $\mathbb{L}$ to the standard lattice $\mathbb{Z}^g$. Let $\vecx'=\matQ^{-1}\vecx$ and $\vecy'=\matQ^{-1}\vecy$. The distance function \cref{eq:trop-eucli-1} then becomes
\begin{equation}\label{eq:trop-eucli-3}
    d_{\alb}([\vecx'],[\vecy']) = \min_{\vecn\in\mathbb{Z}^g}\left((\vecx'-\vecy'-\vecn)^\top\matQ(\vecx'-\vecy'-\vecn)\right)^{\frac{1}{2}} . 
\end{equation}
As shown in \Cref{subsec:property}, the tropical Abel--Jacobi map can be alternatively defined via the orthogonal projection map from $C_1(\Gamma;\mathbb{R})$ to $H_1(\Gamma;\mathbb{R})$, where the induced distance function on the Albanese torus $\alb(\Gamma)$ is exactly \cref{eq:trop-eucli-3}. 

The following theorem shows that under the tropical polarization distance, the tropical Abel--Jacobi map is a H\"{o}lder continuous map from $\Gamma$ to $\jac(\Gamma)$.

\begin{theorem}
    Let $\Gamma$ be a metric graph and $d_\Gamma$ be the distance function on $\Gamma$ induced by its length structure. Equip the tropical Jacobian $\jac(\Gamma)$ with the tropical polarization distance function $d_\Trop$. Then the tropical Abel--Jacobi map satisfies
    \begin{equation}
        d_\Trop\big(\ajmap(p),\ajmap(p')\big)\le \sqrt{d_\Gamma(p,p')} .
    \end{equation}
\end{theorem}

\begin{proof}
    Assume $\gamma_{q,p}$ and $\gamma_{q,p'}$ are paths from $q$ to $p$ and $p'$ respectively. Let $\pi:C_1(\Gamma;\mathbb{R})\to H_1(\Gamma;\mathbb{R})$ be the orthogonal projection map. Then by \cref{eq:trop-eucli-1,eq:trop-eucli-3},
    \begin{equation}\label{eq:holder}
        d_\Trop\big(\ajmap(p),\ajmap(p')\big) \le \|\pi\big(\gamma_{q,p}\big)-\pi\big(\gamma_{q,p'}\big)\|_{C_1(\Gamma;\mathbb{R})}\le\|\gamma_{q,p}-\gamma_{q,p'}\|_{C_1(\Gamma;\mathbb{R})},
    \end{equation}
    where $\|\cdot\|_{C_1(\Gamma;\mathbb{R})}$ is the norm on $C_1(\Gamma;\mathbb{R})$ induced by the inner product \cref{eq:inner-product}. Since \cref{eq:holder} holds for all paths, taking the shortest path $\gamma_{p,p'}$ from $p$ to $p'$, we have
    $$
    d_\Trop\big(\ajmap(p),\ajmap(p')\big)\le \|\gamma_{p,p'}\|_{C_1(\Gamma;\mathbb{R})}=\sqrt{\ell(\gamma_{p,p'})}=\sqrt{d_\Gamma(p,p')} ,
    $$
    which proves the claim.
\end{proof}

\paragraph{The Foster--Zhang Distance.}

\cite{baker2011metric} introduced three candidate norms on the tropical Jacobian: the tropical norm, the Foster--Zhang norm, and the Euclidean norm. Among these, the tropical norm is not a rigorous norm, as it fails to satisfy the triangle inequality. The Euclidean norm is essentially the norm induced by tropical polarization in our context. We now introduce the distance function induced by the Foster--Zhang norm, which is originally motivated from Arakelov geometry \citep{zhang1993admissible} and circuit theory \citep{baker2006metrized}. 
 
 We define the Foster--Zhang distance function on the Albanese torus $\alb(\Gamma)$ as follows: Pick an arbitrary combinatorial model $G$ for the metric graph $\Gamma$. Fix a basis $\{\sigma_i\}_{1\le i\le g}$ for $H_1(G;\mathbb{Z})$. Let $\matC$ be the cycle--edge incidence matrix for $G$. For any vectors $\vecx,\vecy\in\mathbb{R}^g$, the \emph{Foster--Zhang distance} on $\alb(G)$ is defined as
\begin{equation}\label{eq:trop-foster}
    d_{\mathrm{FZ}}([\vecx],[\vecy]) = \min_{\vecn\in\mathbb{Z}^g}\big\|(\vecx-\vecy-\vecn)^\top\matC\big\|_\infty ,   
    \end{equation}
where $\|\cdot\|_\infty$ is the $\ell^\infty$ norm on $\mathbb{R}^{m_G}$. The definition is independent of the choice of combinatorial models and thus $d_{\mathrm{FZ}}$ is a well-defined distance function on the Albanese torus $\alb(\Gamma)$. 

If we view the orthogonal projection $\pi$ as a map from $C_1(G;\mathbb{R})$ to itself, under the basis $\{e_i\}_{1\le i\le m_G}$, the matrix representing $\pi$ is given by
$$
\widetilde{\bm{\Pi}} = \bm{\Pi}^\top\matC .
$$
The \emph{Foster coefficient} associated to edge $e_j$ is defined as
$$
\mathrm{Fs}(e_j) = \big\|\widetilde{\bm{\Pi}}[j,:]\big\|_\infty .
$$
The Foster coefficient has the following properties.
\begin{proposition}{\citep[Section 6]{baker2011metric}}
    Let $G$ be a combinatorial graph. 
    \begin{enumerate}[(i)]
        \item For any $e\in E(G)$, 
        $
        0\le \mathrm{Fs}(e) \le 1 ,
        $
        with $\mathrm{Fs}(e)=0$ if and only if $e$ is a bridge, and $\mathrm{Fs}(e)=1$ if and only if $e$ is a self-loop.
        \item Suppose $G$ has genus $g$, then
        $$
        \sum_{e\in E(G)}\mathrm{Fs}(e) = g .
        $$
        \item The Foster coefficient is compatible with edge refinement in the sense that if the edge $e_j$ of $G$ is subdivided into $e_{j_1},\ldots,e_{j_k}$ in $G'$, then
        $$
        \mathrm{Fs}(e_j) = \mathrm{Fs}(e_{j_1})+\cdots+\mathrm{Fs}(e_{j_k}) . 
        $$
    \end{enumerate}
\end{proposition}
The last property indicates that the Foster coefficient can be viewed as a function on the set of all geodesics on a metric graph $\Gamma$. For any $p,q\in \Gamma$, consider
\begin{equation}\label{eq:distance-fs}
    d_{\mathrm{Fs}}(p,q) = \inf \{\mathrm{Fs}(\gamma):\gamma\text{ is a geodesic connecting }p \text{ and }q\}.
\end{equation}
Then $d_{\mathrm{Fs}}$ defines a pseudo-metric on $\Gamma$, and it defines a metric if and only if $\Gamma$ is 2-connected.

\begin{theorem}
    Let $\Gamma$ be a metric graph and $d_{\mathrm{Fs}}$ be the (pseudo-)distance function as in \cref{eq:distance-fs}. Equip the Albanese torus $\alb(\Gamma)$ with the Foster--Zhang distance function $d_{\mathrm{FZ}}$. Then
    \begin{equation}\label{eq:fz-ineq}
       d_{\mathrm{FZ}}\big(\widehat{\ajmap}(p),\, \widehat{\ajmap}(p')\big)\le d_{\mathrm{Fs}}(p,p') . 
    \end{equation}
\end{theorem}
\begin{proof}
    Assume $\gamma_{q,p}$ and $\gamma_{q,p'}$ are geodesics from $q$ to $p$ and $p'$ respectively. Let $\gamma_{p,p'}$ be the geodesic joining $\gamma_{p,q}$ and $\gamma_{q,p'}$. By construction we have
    \begin{equation}\label{eq:fz-1}
           d_{\mathrm{FZ}}\big(\widehat{\ajmap}(p),\, \widehat{\ajmap}(p')\big)\le \|\pi(\gamma_{q,p})-\pi(\gamma_{q,p'})\|_\infty \le \mathrm{Fs}(\gamma_{p,p'}). 
    \end{equation}
    Since \cref{eq:fz-1} holds for arbitrary $\gamma_{p,p'}$, minimizing over all such geodesics we obtain \cref{eq:fz-ineq}.
\end{proof}

\subsection{Hardness of Computation and Approximation}\label{subsec:np-hardness}

Let $\mathbb{L}\subseteq \mathbb{R}^g$ be a full rank lattice so that $\mathbb{T}^g = \mathbb{R}^g/\mathbb{L}$ is a $g$-dimensional real torus. Let $\|\cdot\|$ be any norm on the vector space $\mathbb{R}^g$. Then $\|\cdot\|$ induces a norm on the tangent spaces of $\mathbb{T}^g$, which further makes $\mathbb{T}^g$ a flat Finsler manifold. Let $\btau_1,\ldots,\btau_g$ be a basis of $\mathbb{L}$. For any $\vecx,\vecy\in\mathbb{R}^g$, the distance function on the flat torus is given by
$$
d_{\mathbb{T}^g}([\vecx],[\vecy]) = \min_{k_i\in\mathbb{Z}}\bigg\|\vecx-\vecy-\sum_{i=1}^gk_i\btau_i\bigg\| .
$$
It turns out that for a generic lattice $\mathbb{L}$, the computation of $d_{\mathbb{T}^g}$ is highly nontrivial, and is equivalent to well-known NP-hard problems in lattice-based cryptography and computational complexity theory.

\paragraph{NP-Hardness of Lattice Problems.}

For a lattice $\mathbb{L}$, let $\lambda(\mathbb{L})$ be the length of the shortest nonzero vector in $\mathbb{L}$, that is,
$$
\lambda(\mathbb{L}) = \min_{\vecx\in\mathbb{L}\backslash\{\mathbf{0}\}}\|\vecx\| .
$$
The shortest vector problem (SVP) seeks to find such shortest vector in $\mathbb{L}$.

\begin{definition}
The \emph{shortest vector problem} $\svp$ is the following: Given a lattice $\mathbb{L}$, find $\btau\in\mathbb{L}$ such that $\|\btau\|=\lambda(\mathbb{L})$. 
An \emph{approximate SVP}, denoted by $\svpz$, seeks to find a vector in $\mathbb{L}^g$ which has length bounded by a factor of $\lambda(\mathbb{L})$ and is stated as follows: Given a lattice $\mathbb{L}$ and an approximation factor $\zeta\ge 1$, find $\btau\in\mathbb{L}$ such that $0<\|\btau\|\le \zeta\lambda(\mathbb{L})$. 

\end{definition}
    
A closely related problem of $\svp$ is the closest vector problem (CVP), where for a given target vector $\vect\notin\mathbb{L}$, the goal is to find the closest vector in $\mathbb{L}$ with respect to $\vect$.

\begin{definition}
The \emph{closest vector problem }$\cvp$ is the following: Given a lattice $\mathbb{L}$ and a target vector $\vect\notin\mathbb{L}$, find $\btau\in\mathbb{L}$ such that $\|\vect-\btau\|=d_{\mathbb{T}^g}([\vect],[\mathbf{0}])$. An \emph{approximate CVP}, denoted by $\cvpz$, is the following: Given a lattice $\mathbb{L}$, a target vector $\vect\notin \mathbb{L}$, and an approximation factor $\zeta\ge 1$, find $\btau\in\mathbb{L}$ such that $\|\vect-\btau\|\le \zeta d_{\mathbb{T}^g}([\vect],[\mathbf{0}])$.
\end{definition}

Note that for any $[\vecx]\neq[\vecy]\in\mathbb{T}^g$, computing the distance $d_{\mathbb{T}^g}([\vecx],[\vecy])$ is exactly solving the $\cvp$ for the target vector $\vect=\vecx-\vecy$. 

The SVP and CVP belong to a general class of optimization problems known as \emph{lattice problems}. Such problems are typically hard to solve, providing foundations to  the construction of secure lattice-based cryptosystems \citep{regev2009lattices}. 
%In the following we present a short literature review about the NP-hardness of SVP, CVP and their approximate versions. We refer to \cite{aggarwal2023lattice,micciancio2009lattice,pradhan2019lattice} for a comprehensive introduction to lattice problems. 
%
% In computational complexity theory, the class NP (Nondeterministic Polynomial Time) consists of problems for which a proposed solution can be verified as correct within polynomial time. A problem is NP-hard if all problems in class NP can be reduced to it in polynomial time. That is,  a polynomial time solution to any NP-hard problem would imply a polynomial time solution for every problem in class NP. A problem is NP-complete if it is both NP and NP-hard. The NP-complete problems represent the hardest problems in NP.
It was first proven by \cite{van1981another} that the CVP is NP-complete for all $\ell^p$ ($1\le p\le \infty$) norms and that the SVP is NP-complete for the $\ell^\infty$ norm. It is known that the SVP is ``easier'' to solve than the CVP in the sense that the SVP reduces to polynomially many instances of the CVP of the same rank and dimension \citep{manohar2016hardness}, which essentially means that showing the NP-completeness of the SVP is much harder. \cite{ajtai1998shortest} first showed that the SVP is NP-hard for the $\ell^2$ norm under randomized reductions. It is still an open problem to prove that the SVP is NP-hard for $\ell^p$ ($1\le p<\infty$) norms under deterministic reductions. 

The approximate versions of the SVP and CVP also appear to be NP-hard to solve for different levels of approximation factors. For constant factors (i.e., $\zeta$ independent of dimension $g$), both the $\cvpz$ and $\svpz$ are known to be NP-hard \citep{arora1997hardness,khot2005hardness}. For sub-polynomial factors, \cite{dinur1998approximating} proved that the $\cvpz$ is NP-hard for any factor up to $\zeta=O\big(2^{(\log g)^{1-\epsilon(g)}}\big)$ where $\epsilon(g)$ is a slowly decreasing function of $g$, while \cite{haviv2007tensor} proved that the $\svpz$ is NP-hard for any factor up to $\zeta=O\big(2^{(\log g)^{1-\epsilon}}\big)$ where $\epsilon>0$ is any constant. For certain polynomial factors, it has been shown that the $\svpz$ and $\cvpz$ are not NP-hard under standard complexity assumptions \citep{aharonov2005lattice,goldreich1998limits}. The question of whether there exists $\epsilon>0$ that the $\cvpz$ and $\svpz$ are NP-hard for $\zeta=O(n^\epsilon)$ remains open. Currently, any practical algorithm to solve the $\cvpz$ and $\svpz$ for polynomial approximation factors still takes exponential time with respect to dimension \citep{chen2011bkz}.

\paragraph{Special Cases with Explicit Solutions.}

Consider the standard torus $\mathbb{R}^g/\mathbb{Z}^g$. For any $\ell^p$ norm on $\mathbb{R}^g/\mathbb{Z}^g$, computing the distance can be solved explicitly. In fact, for any $\vecx,\vecy\in\mathbb{R}^g$, assuming $p<\infty$,
$$
\begin{aligned}
    d_{\ell^p}([\vecx],[\vecy]) &= \min_{k_i\in\mathbb{Z}}\bigg\|\vecx-\vecy-\sum_{i=1}^gk_i\mathbf{e}_i\bigg\|_p
     = \min_{k_i\in\mathbb{Z}}\left(\sum_{i=1}^g |x_i-y_i-k_i|^p\right)^{1/p} .
\end{aligned}
$$
Since the optimizing variables can be separated along with the summands, we have
$$
   d_{\ell^p}^p([\vecx],[\vecy]) = \sum_{i=1}^g\left( \min_{k_i\in\mathbb{Z}}\big|x_i-y_i-k_i\big|^p\right) .
$$
Thus, the optimization problem can be solved coordinate-wise and the optimal translations are given by the nearest integers $k_i=\lfloor x_i-y_i\rceil$.

We observe that although computing distances on flat tori with skew lattices is hard, explicit solutions exist if a lattice allows orthogonal bases. In the following we describe metric graphs whose tropical Jacobian allows explicit computation of the tropical polarization distance. 

\begin{definition}
    Let $\mathbb{T}^g=\mathbb{R}^g/\mathbb{L}$ be a flat torus with metric induced from the $\ell^2$ norm on $\mathbb{R}^g$. The lattice $\mathbb{L}$ is \emph{rectangular} if there exists a pairwise orthogonal lattice basis.
\end{definition}

Let $\matW$ be a matrix whose column vectors form a pairwise orthogonal basis of $\mathbb{L}$. By normalizing these column vectors, we can write $\matW=\widetilde{\matW}\matD$ where $\widetilde{\matW}$ is an orthogonal matrix and $\matD$ is a diagonal matrix with diagonal entries equal to norms of the columns of $\matW$. Applying an orthogonal transformation $\widetilde{\matW}^{-1}$, the lattice $\mathbb{L}$ is generated by columns of $\matD$, which can be viewed as a scaled version of $\mathbb{Z}^g$.  

\begin{lemma}\label{lemma:lattice-isometry}
Let $\Gamma$ be a metric graph. There exists an isometry $(\jac(\Gamma),d_{\Trop})\to (\mathbb{R}^g/\mathbb{L},d_{\ell^2})$ where $\mathbb{L}$ is rectangular if and only if there exists $\matP\in\mathrm{GL}(g;\mathbb{Z})$ such that $\matP^\top\matQ\matP$ is diagonal.
\end{lemma}

\begin{proof}
    Let $\matQ^{\frac{1}{2}}$ be the square root of $\matQ$. The tropical polarization distance \cref{eq:trop-eucli-1} can be written as 
\begin{equation}\label{eq:trop-eucli-2}
    d_\Trop([\vecx], [\vecy]) = \min_{\vecn\in\mathbb{Z}^g}\big\|\matQ^{-\frac{1}{2}}(\vecx-\vecy)-\matQ^{\frac{1}{2}}\vecn\big\|_2
\end{equation}
Thus the linear transformation by $\matQ^{-\frac{1}{2}}$ defines an isometry from $(\jac(\Gamma),d_\Trop)$ to $(\mathbb{R}^g/\mathbb{L},\|\cdot\|_2)$ where $\mathbb{L}$ is generated by the column vectors of $\matQ^{\frac{1}{2}}$. If $\mathbb{L}$ is rectangular, then there exists $\matP\in\mathrm{GL}(g;\mathbb{Z})$ such that the column vectors of $\matQ^{\frac{1}{2}}\matP$ are pairwise orthogonal. Thus there exists a diagonal matrix $\matD$ such that
\begin{equation}\label{eq:diag-mat}
    \big(\matQ^{\frac{1}{2}}\matP\big)^\top \matQ^{\frac{1}{2}}\matP  = \matP^\top\matQ\matP = \matD.
\end{equation}
Conversely, if \cref{eq:diag-mat} holds, then the column vectors of $\matQ^{\frac{1}{2}}\matP$ are pairwise orthogonal, which implies that $\mathbb{L}$ is rectangular.
\end{proof}

We will use the definition of cycle decomposition of an undirected combinatorial graph from graph theory \citep{arumugam2013decomposition}.

\begin{definition}
    Let $G$ be an undirected combinatorial graph. A decomposition of $G$ is a collection of edge-disjoint subgraphs $G_1,\ldots,G_r$ such that each edge of $G$ belongs to exactly one $G_i$. If each $G_i$ is a cycle in $G$, then  $G_1,\ldots,G_r$ is a \emph{cycle decomposition} of $G$.
\end{definition}

\begin{lemma}
    Let $\Gamma$ be a metric graph and assume $G$ is a combinatorial model of $\Gamma$. Let $\widetilde{G}$ be the graph constructed by contracting all bridge edges from $G$. Then $\widetilde{G}$ admits a cycle decomposition if and only if $\matQ$ is diagonal.
\end{lemma}

\begin{proof}
Let $\sigma_1,\ldots,\sigma_g$ be a basis of fundamental 1-cycles. For each $\sigma_i$, let $\overline{\sigma}_i$ be the support set of edges in $G$.  For any $i\neq j$, $\matQ[i,j] = 0$ if and only if $\overline{\sigma}_i$ is disjoint from $\overline{\sigma}_j$. That is, $\matQ$ is diagonal if and only if each non-bridge edge belongs to exactly one cycle, which is equivalently to $\widetilde{G}$ admitting a cycle decomposition.      
\end{proof}

Let $\mathrm{SPD}(g)$ be the set of all $g\times g$ symmetric positive definite matrices, and let $\mathrm{TSPD}(g)\subseteq \mathrm{SPD}(g)$ be the subset of all tropical polarization matrices arising from metric graphs of genus $g$. For a non-diagonal matrix $\matQ\in\mathrm{SPD}(g)$, it is possible to diagonalize $\matQ$ with some $\matP\in \mathrm{GL}(g;\mathbb{Z})$. However, if $\matQ\in\mathrm{TSPD}(g)$ is non-diagonal, then such a $\matP$ does not exist, which is a consequence of the injectivity of the \emph{tropical Torelli map} \citep{brannetti2011tropical}. We apply this result to prove the following theorem that the only metric graphs for which we can explicitly compute the tropical polarization distance are those that allow cycle decompositions. Since the technical details are beyond the main scope of our paper, we defer the proof details to \Cref{app:proof-rect-trop}. 

\begin{theorem}\label{thm:rect-trop-jac}
    Let $\Gamma$ be a metric graph and $G$ be its combinatorial model. Let $\widetilde{G}$ be the graph constructed by contracting all bridge edges from $G$. Then there exists an isometry $(\jac(\Gamma),d_{\Trop})\to (\mathbb{R}^g/\mathbb{L},d_{\ell^2})$ where $\mathbb{L}$ is rectangular if and only if $\widetilde{G}$ admits a cycle decomposition.
\end{theorem}

% \begin{example}
%     Consider any $2\times 2$ positive definite matrix in the form that
%     $$
%     \begin{bmatrix}
%         a & -a\\ -a& b
%     \end{bmatrix},\quad b>a>0
%     $$
%     It is straightforward to check that
%     $$
%     \begin{bmatrix}
%         1& 0\\ 1& 1
%     \end{bmatrix}\begin{bmatrix}
%         a & -a\\ -a& b
%     \end{bmatrix} \begin{bmatrix}
%         1& 1\\ 0&1
%     \end{bmatrix}  =\begin{bmatrix}
%         a& 0 \\ 0 &b-a
%     \end{bmatrix}
%     $$
% \end{example}

\paragraph{Available Solvers for Low Dimensional Computations.}

As we have seen from previous discussions, the computational  complexity of the CVP/SVP is exponential with respect to dimension for generic lattices. However, it is still possible to carry out practical computations in low dimensions, for example, when a metric graph has low genus. There are publicly available CVP/SVP solvers developed for cryptography purposes, among which the most extensively used ones are \texttt{fplll}\footnote{\url{https://github.com/fplll/fplll}} and \texttt{G6K}\footnote{\url{https://github.com/fplll/g6k}}. With \texttt{fplll}, exact solutions of CVP/SVP are given by a variant of the Kannan--Fincke--Pohst algorithm \citep{kannan1983improved,fincke1985improved}, which is based on a process of recursive enumeration and pruning. \texttt{G6K} is based on lattice sieving algorithms and contains a comprehensive list of lattice sieves such as the Nguyen--Vidick Sieve \citep{nguyen2008sieve}, the Gauss Sieve \citep{micciancio2010faster}, and the Becker--Gama--Joux Sieve \citep{becker2015speeding}, among others. 

%However, most CVP/SVP-solvers are designed for $\ell^2$ norms.  Not all algorithms are implemented.

%There are mathematical libraries containing implementations of lattice reductions, for example, \texttt{NLT} \citep{shoup2001ntl} and \texttt{FLINT} \citep{hart2010fast}, but for even smaller computational scales.

Another approach for practical computation is to use mixed integer programming (MIP) solvers. Recall that the square of the tropical polarization distance \cref{eq:trop-eucli-1} is an optimization problem over integer variables, which belongs to the class of integer quadratic programming (IQP) problems. The Foster--Zhang metric then reduces to the following optimization problem: Let $\vecb = (\vecx-\vecy)^\top\matC$. Expanding \cref{eq:trop-foster} coordinate-wise, we have
\begin{equation}\label{eq:opt-fz-dist}
    \min_{\vecn\in\mathbb{Z}^g}\bigg\{\max_{1\le j\le m_G} \bigg\{\pm\bigg(\vecn[1]\matC[1,j]+\vecn[2]\matC[2,j]+\cdots+\vecn[g]\matC[g,j]-\vecb[j]\bigg)\bigg\}\bigg\} ,
\end{equation}
where we express the absolute values as positive terms and negative terms so that the maximum is taken over $2m_G$ linear functions. By introducing a new variable $t\in\mathbb{R}$, \cref{eq:opt-fz-dist} is equivalent to the following mixed integer linear programming (MILP) problem:
\begin{equation}\label{eq:MILP}
    \begin{aligned}
        \min\quad & t\\
        \mathrm{s.t.}\quad &\begin{cases}
             -t \le \vecn[1]\matC[1,j]+\vecn[2]\matC[2,j]+\cdots+\vecn[g]\matC[g,j]-\vecb[j]\le t ,\\
         t\ge 0 ,\\
         \vecn \in\mathbb{Z}^g .
        \end{cases} 
    \end{aligned}
\end{equation}
The above formulations can be used as input of any MIP solver, for which there exist noncommercial MIP software tools \citep{linderoth2005noncommercial}. We will demonstrate the use of some publicly available MIP solvers further on in \Cref{subsec:simulation}.

%directly put into many (mixed) integer programming solvers, for example, \texttt{Gurobi} \citep{gurobi}, \texttt{CPLEX} \citep{cplex2009v12}, \texttt{SCIP} \citep{Achterberg2009} etc.. 

\subsection{Computing Truncated Tropical Polarization Distances}\label{subsec:truncated}

For high-dimensional tropical Jacobians, it is hard to obtain exact values or even close approximations of pairwise distances for point cloud data. Instead of computing a full set of pairwise distances, we focus on computing distances locally at each point. In this section, we illustrate how to compute truncated tropical polarization distances on the tropical Jacobian.

\paragraph{Babai's Algorithms and Lattice Reduction.}

Let $\mathbb{T}^g=\mathbb{R}^g/\mathbb{L}$ be a $g$-dimensional torus with metric induced from the $\ell^2$ norm. Fix a matrix $\matW$ whose column vectors form a basis of the lattice $\mathbb{L}$. We review two classical algorithms by LBabai to approximate solutions of the CVP \citep{babai1986lovasz}. 

The first algorithm is known as \emph{Babai's rounding algorithm}:  Let $\vect\in\mathbb{R}^g$ be a target vector. The (continuous) minimizer of the optimization problem
$$
\min_{\vecx\in\mathbb{R}^g} \|\vect-\matW\vecx\|_2^2
$$
is given by $\vecx^* = \matW^{-1}\vect$. Thus to obtain an integral solution, the coordinates of $\vecx^*$ are simply rounded to the nearest integers, i.e., the computed lattice vector is given by $\matW\lfloor\vecx^* \rceil$.

The second algorithm is called \emph{Babai's nearest plane algorithm}. The idea is to project the target vector onto a series of hyperplanes and adjust the vector iteratively by finding the closest lattice point on each hyperplane. We formulate the algorithm in terms of matrix decompositions. Specifically, let $\matW = \matU\widetilde{\matW}$ be the QR decomposition of $\matW$, where $\matU$ is an orthogonal matrix and $\widetilde{\matW}$ is an upper-triangular matrix. Since orthogonal transformation preserves the $\ell^2$ norm, we have 
\begin{equation}\label{eq:qr-decomp}
\min_{\vecn\in\mathbb{Z}^g} \|\vect-\matW\vecn\|_2^2 = \min_{\vecn\in\mathbb{Z}^g} \|\vect'-\widetilde{\matW}\vecn\|_2^2 ,   
\end{equation}
where $\vect' = \matU^{-1}\vect$. Writing in coordinates, \cref{eq:qr-decomp} is equivalent to minimizing 
\begin{equation}\label{eq:sum-loss}
    \begin{aligned}
        & \big(\widetilde{\matW}[1,1]\vecn[1]+\widetilde{\matW}[1,2]\vecn[2]+\cdots+\widetilde{\matW}[1,g]\vecn[g]-\vect'[1]\big)^2\\
        +&\big(\widetilde{\matW}[2,2]\vecn[2]+\cdots+\widetilde{\matW}[2,g]\vecn[g]-\vect'[2]\big)^2\\
        &\vdots\\
        +&\big(\widetilde{\matW}[g,g]\vecn[g]-\vect'[g]\big)^2 .
    \end{aligned}
\end{equation}
Then we minimize \cref{eq:sum-loss} successively from the bottom to the top. For the last summand, the optimal integer is given by $\vecn[g]^*=\lfloor \vect'[g]/\widetilde{\matW}[g,g]\rceil$. Suppose we have computed $\vecn[j+1]^*,\ldots,\vecn[g]^*$, which are then substituted in to the $j$th summand.  Then the optimal integer $\vecn[j]^*$ is given by
$$
\vecn[j]^* = \left\lfloor \frac{\vect'[j]-\widetilde{\matW}[j,g]\vecn[g]^*-\cdots-\widetilde{\matW}[j,j+1]\vecn[j+1]^*}{\widetilde{\matW}[j,j]} \right\rceil .
$$

The performance of both Babai's algorithms heavily depends on the basis, i.e., the matrix $\matW$. The best possible basis is the one whose vectors are pairwise orthogonal, however, it only exists for cycle decomposable graphs by \Cref{thm:rect-trop-jac}. In practice, we aim to find a basis that is as close to being orthogonal as possible. The process of finding a well-conditioned or nearly orthogonal basis is referred to as lattice reduction.

The Lenstra--Lenstra--Lov\'asz (LLL) reduction algorithm is a foundational algorithm in lattice reduction \citep{lenstra1982factoring}. Here, we will not describe in detail the steps of the LLL algorithm; we note that practical implementations have further improved of the original LLL algorithm and its variants. Instead, we will focus on the properties of the bases produced by the LLL algorithm and their implications.

%such as the LLL-BKZ algorithm which produces better bases than LLL, albeit being more complicated and adding more processing time \citep{chen2011bkz}.

\begin{definition}
    Let $\btau_1,\ldots,\btau_g$ be an ordered basis of the lattice $\mathbb{L}$, and let $\widetilde{\btau}_1,\ldots,\widetilde{\btau}_g$ be the corresponding Gram--Schmidt orthogonalization. The ordered basis $\btau_1,\ldots,\btau_g$ is called an \emph{LLL-reduced basis} if it satisfies the following two conditions.
    \begin{enumerate}[(i)]
        \item \emph{Size condition}: For all $1\le j<i\le g$,
        $$
        \mu_{i,j} = \frac{|\btau_i^\top\widetilde{\btau}_j|}{\|\widetilde{\btau}_j\|_2^2}\le \frac{1}{2} \,.
        $$
        \item \emph{Lov\'asz condition}: For all $1\le j\le g-1$,
        $$
        \|\widetilde{\btau}_{j+1}\|_2^2\ge \bigg(\frac{3}{4}-\mu_{j+1,j}\bigg)\|\widetilde{\btau}_j\|_2^2 \,.
        $$
    \end{enumerate}
\end{definition}
In the Gram--Schmidt orthogonalization, the projection of $\btau_i$ to $\widetilde{\btau}_j$ is given by $\mu_{i,j}\widetilde{\btau}_j$. If $\mu_{i,j}>\frac{1}{2}$, then by replacing $\btau_i$ with $\btau_i-\lfloor\mu_{i,j}\rceil\btau_j$, the coefficient can be made smaller, which decreases the length of $\btau_i$. The size condition essentially means that each $\btau_i$ is small enough relative to the space spanned by its previous basis vectors $\btau_1,\ldots,\btau_{i-1}$. However, the size condition alone cannot force the near orthogonality of the basis. The Lov\'asz condition imposes further restriction to the basis that the vectors are ordered nicely in the sense that the length of $\widetilde{\btau}_{i+1}$ is proportionally larger than $\widetilde{\btau}_i$. The two conditions together imply that an LLL-reduced basis is nearly orthogonal.

For any lattice $\mathbb{L}$, the LLL algorithm finds an LLL-reduced basis in polynomial time \citep{lenstra1982factoring}. Given an LLL-reduced basis, we have theoretical guarantees on approximations of the SVP and CVP.

\begin{theorem}{\citep[Theorem 18.1.6, Theorem 18.2.1]{galbraith2012mathematics}}\label{thm:lll-reduced}
    Let $\btau_1,\ldots,\btau_g$ be an LLL-reduced basis for $\mathbb{L}$. Then
    \begin{enumerate}[(i)]
        \item $\|\btau_1\|_2\le 2^{\frac{g-1}{2}}\lambda(\mathbb{L})$, thus an LLL-reduced basis solves $\svpz$ for $\zeta=2^{\frac{g-1}{2}}$.
        \item Let $\matW$ be the matrix whose columns are given by $\btau_1,\ldots,\btau_g$. Babai's rounding algorithm solves the $\cvpz$ for $\zeta = 1+2g\big(9/2\big)^{\frac{g}{2}}$, and Babai's nearest plane algorithm solves the $\cvpz$ for $\zeta=2^{\frac{g}{2}}$.
    \end{enumerate}
\end{theorem}

% as we can see from the following toy example.  

% \begin{example}
% Let $k>0$ be a fixed even number. Consider the standard two-dimensional torus $\mathbb{R}^2/\mathbb{Z}^2$, but choose $[1,0]^\top$ and $[k,1]^\top$ as the basis. Set the  We want to compute the norm of the point $\bm{b}=\frac{1}{2}(n+1,1)$. Thus we want to minize the following function
%     $$
%     f(\bm{x}) = \left\|\begin{pmatrix}
%         1 & n\\ 0 &1
%     \end{pmatrix}\bm{x}-\bm{b}\right\|_2^2
%     $$
%     The continuous minimizer is give by 
%     $$
%     \bm{x}^* = \frac{1}{2}\begin{pmatrix}
%         1 & -n\\ 0 & 1
%     \end{pmatrix}\begin{pmatrix}
%         n+1\\ 1
%     \end{pmatrix} = \frac{1}{2}\begin{pmatrix}
%         1 \\ 1
%     \end{pmatrix}
%     $$
%     We may expect the nearest point appear in 
%     $$
%     (0,0),(1,0),(0,1),(1,1)
%     $$
%     and the computed norm is $\frac{1}{2}\sqrt{(n-1)^2+1}$. 

%     However, spanning the cost function is coordinates we have
%     $$
%     f(x_1,x_2) = (x_1+(2x_2-1)k-\frac{1}{2})^2+(x_2-\frac{1}{2})^2\ge \frac{1}{2}
%     $$
%     and the minimum is achievable by 
%     $$
%     (k,0), (-k,1)
%     $$
%     Thus the integral minimizers are in the box of length $k$ instead of the unit box. By adjusting $k$ we can make the integral minimizer arbitrarily far away from the continuous minimizer, and the error from rounding as arbitrarily large as we want. 

%     \textcolor{red}{illustration figure here}
% \end{example}

\paragraph{Thresholding.} As implied by \Cref{thm:lll-reduced}, an LLL-reduced basis allows us to solve the $\svpz$ and $\cvpz$ with exponential approximation factors. This is consistent with our discussion in \Cref{subsec:np-hardness}, since both the LLL algorithm and Babai's algorithms run in polynomial time, we cannot expect better approximation factors. However, the approximations are still insufficient for computing pairwise distances in the tropical Jacobian. As a result, we discard large values, and only preserve those under a certain threshold. This is based on the following theorem.

\begin{theorem}\label{thm:inj-rad}
    Let $\mathbb{T}^g=\mathbb{R}^g/\mathbb{L}$ be a $g$-dimensional torus with metric induced from the $\ell^2$ norm. For any $\vecx,\vecy\in\mathbb{R}^g$, if $\|\vecx-\vecy\|_2\le\frac{1}{2}\lambda(\mathbb{L})$, then $d_{\ell^2}([\vecx],[\vecy])=\|\vecx-\vecy\|_2$.
\end{theorem}

\begin{proof}
    Suppose $\btau^*\in\mathbb{L}\backslash\{\bm{0}\}$ is such that 
    $$
    d_{\ell^2}([\vecx],[\vecy])=\min_{\btau\in\mathbb{L}}\|\vecx-\vecy-\btau\|_2 = \|\vecx-\vecy-\btau^*\|_2 .
    $$
    By the triangle inequality, we have
    \begin{equation}\label{eq:inj-rad}
        \begin{aligned}
            \frac{1}{2}\lambda(\mathbb{L})\ge\|\vecx-\vecy\|_2&\ge d_{\ell^2}([\vecx,\vecy])
            = \|\vecx-\vecy-\btau^*\|_2\\
            &\ge \|\btau^*\|_2-\|\vecx-\vecy\|_2\\
            &\ge \lambda(\mathbb{L})-\frac{1}{2}\lambda(\mathbb{L}) = \frac{1}{2}\lambda(\mathbb{L}) .
        \end{aligned}
    \end{equation}
    If $\|\vecx-\vecy\|_2<\frac{1}{2}\lambda(\mathbb{L})$, then \cref{eq:inj-rad} yields a contradiction, thus $\btau^*=\bm{0}$. If $\|\vecx-\vecy\|_2=\frac{1}{2}\lambda(\mathbb{L})$, then \cref{eq:inj-rad} forces all inequalities to be equalities. In both cases, $d_{\ell^2}([\vecx],[\vecy])=\|\vecx-\vecy\|_2$.
\end{proof}

\begin{remark}
    For a Riemannian manifold, the \emph{injective radius} is defined as the largest radius such that all geodesics within this radius are unique and length-minimizing. \Cref{thm:inj-rad} implies that the injective radius of $\mathbb{T}^g$ is $\frac{1}{2}\lambda(\mathbb{L})$.
\end{remark}

Via thresholding, we essentially localize the computation of pairwise distances around each point. Given a data set $\{\vecx_1,\ldots,\vecx_N\}\subseteq\mathbb{R}^g$, we first compute the full distance matrix
$$
\matM[i,j] = \|\vecx_i-\vecy_j-\btau_{ij}\|_2 ,
$$
where $\btau_{ij}\in\mathbb{L}, 1\le i\le j\le N$, are obtained from Babai's algorithm. Then we fix a threshold $0<\vartheta<\frac{1}{2}\lambda(\mathbb{L})$, and obtain the truncated distance matrix $\matM_\vartheta$ as
$$
\matM_\vartheta[i,j] = \begin{cases}
    \infty, \text{ if } \|\vecx_i-\vecx_j-\btau_{ij}\|_2>\vartheta\\
    \|\vecx_i-\vecx_j-\btau_{ij}\|_2, \text{ if } \|\vecx_i-\vecx_j-\btau_{ij}\|_2\le \vartheta .
\end{cases}
$$
By \Cref{thm:inj-rad}, all finite entries of $\matM_\vartheta$ are true distances.

We summarize the above computation in the setting of the tropical Abel--Jacobi transform of a metric graph.  Let $\Gamma$ be a metric graph of genus $g$. Applying \Cref{alg:cycle--edge,alg:interpolation}, we obtain a point cloud data in the tropical Jacobian $\jac(\Gamma)$ where the lattice basis is given by the column vectors of $\matQ$. After applying a linear transformation $\matQ^{-\frac{1}{2}}$, we map the point cloud data isometrically into the torus $(\mathbb{R}^g/\mathbb{L},d_{\ell^2})$ where the lattice basis is given by the column vectors of $\matQ^{\frac{1}{2}}$. We first apply the LLL algorithm to $\matQ^{\frac{1}{2}}$ to find a LLL-reduced basis for $\mathbb{L}$. Then we use Babai's algorithms to compute the pairwise distances and filter the distances below the threshold $\vartheta = \|\btau_1\|/2^{\frac{g+1}{2}}$ given by \Cref{thm:lll-reduced}. We formulate the computation of truncated tropical polarization distance in \Cref{alg:truncated-dist}.

\begin{algorithm}[H]
\caption{Computation of the truncated tropical polarization distance matrix}
\label{alg:truncated-dist}

\DontPrintSemicolon
\SetAlgoLined
\SetNoFillComment

\KwIn{$\matV$: $g\times N$ matrix,\, $\matQ$: $g\times g$ matrix}
\KwOut{$\matM$: $N\times N$ matrix}

Compute the square root $\matQ^{\frac{1}{2}}$\;

$\widetilde{\matV}\gets \matQ^{-\frac{1}{2}}\matV$\tcp*[r]{Apply linear transformation}

$\matW \gets \texttt{LLL}(\matQ^{\frac{1}{2}})$\tcp*[r]{Compute LLL-reduced basis}

$\vartheta\gets \|\matW[:,1]\|/2^{\frac{g+1}{2}}$\tcp*[r]{Set threshold}

Initialize an $N\times N$ zero matrix $\matM$\;

\For{$i\gets 1$ \KwTo $N$}{
\For{$j\gets i+1$ \KwTo $N$}{
$\btau\gets \texttt{Babai}(\widetilde{\matV}[:,i]-\widetilde{\matV}[:,j], \matW)$\tcp*[r]{Use Babai's algorithm to find the closest lattice vector}

$\mathrm{dist}\gets \|\widetilde{\matV}[:,i]-\widetilde{\matV}[:,j]-\btau\|$\;
\uIf{$\mathrm{dist}\le \vartheta$}{
$\matM[i,j]\gets \mathrm{dist}$\;
}\Else{
$\matM[i,j]\gets \infty$\;
}
}
}

$\matM\gets \matM+\matM^\top$\;

\Return{$\matM$}
\end{algorithm}

% \paragraph{Bounds of approximation}

% We prove that the solution of greedy method yields a uniform bound on the norm of points in the torus.

% From the choice of $n_j^*$ we have
% $$
% \frac{b_j^*-s_{jg}n_g^*-\cdots-s_{j(j+1)}n_{j+1}^*}{s_{jj}}-\frac{1}{2}\le n_j^*\le \frac{b_j^*-s_{jg}n_g^*-\cdots-s_{j(j+1)}n_{j+1}^*}{s_{jj}} + \frac{1}{2}
% $$
% Thus
% $$
% (s_{jj}n_j^*-(b_j^*-s_{jg}n_g^*-\cdots-s_{j(j+1)}n_{j+1}^*))^2\le \frac{1}{4}s_{jj}^2
% $$
% Summing up we have
% $$
% OP(\mathbf{H},\bm{b})\le \frac{1}{4}\sum_{j}s_{jj}^2
% $$
% The bound is tight since for orthogonal lattices the center point achieves this bound. For skew lattice, for example in the counterexample of rounding, the bound can also be tight.

% For tropical Jacobians the bound is
% $$
% length \le \frac{1}{2}\sqrt{\sum_{i=1}^g \ell(\sigma_i)} 
% $$
% where $\ell(\sigma_i)$ is the length of the $i$th cycle. This can be seen from Cholesky decomposition. 

\subsection{Simulations}\label{subsec:simulation}

To evaluate the performance of our algorithms for computing tropical distance matrices, we conduct a simulation study on synthetic metric graphs. All simulations were performed within a Linux environment using Windows Subsystem for Linux (WSL) on a PC with an AMD 4750U CPU, 16 GB of RAM.   

\paragraph{Computing Tropical Polarization Distance Matrices.}

We evaluate the computation time for tropical polarization distance matrices using exact CVP solvers, as discussed in \Cref{subsec:np-hardness}. Specifically, we employ the enumeration algorithm from \texttt{fplll} and the sieving algorithm based on the Gauss sieve and Nguyen--Vidick sieve from \texttt{G6K}. The computation time is analyzed with respect to both the number of graph nodes and the graph genus.

To test the computation time with respect to the number of nodes, we generate graphs with $n=20\sim 120$, while keeping the genus fixed at $g=15$. For each graph, we compute the tropical Abel--Jacobi transform of the node set and then compute the tropical polarization distance matrices using the aforementioned CVP algorithms. \Cref{fig:time-nodes-cvp} presents the log-log plot of computation time versus the number of nodes, which is consistent with the theoretical complexity that the algorithm is quadratic to the number of nodes.

To analyze the computation time with respect to genus, we generate graphs with genus $g=5\sim 45$, while fixing the number of nodes at $n=50$. We follow the same procedure, computing the tropical Abel--Jacobi transform and then the tropical polarization distance matrices. As shown in \Cref{fig:time-genus-cvp}, the log-log plot of computation time versus genus aligns with the theoretical complexity that the algorithm costs exponential time with respect to genus, i.e., dimension of tropical Jacobians.

\begin{figure}[htbp]
    \centering
    \begin{subfigure}{0.45\linewidth}
        \includegraphics[width=\textwidth]{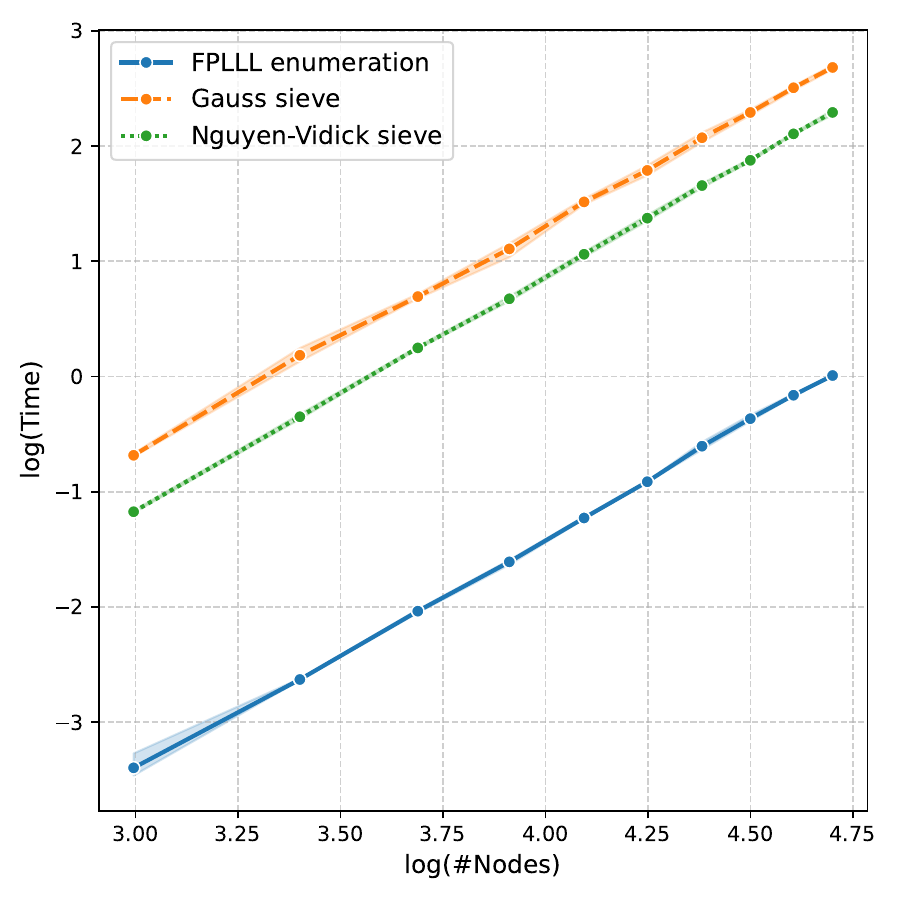}
        \caption{$\log(\text{Time})$--$\log(\#\text{Nodes})$ plot}
        \label{fig:time-nodes-cvp}
    \end{subfigure}
    \begin{subfigure}{0.45\linewidth}
        \includegraphics[width=\textwidth]{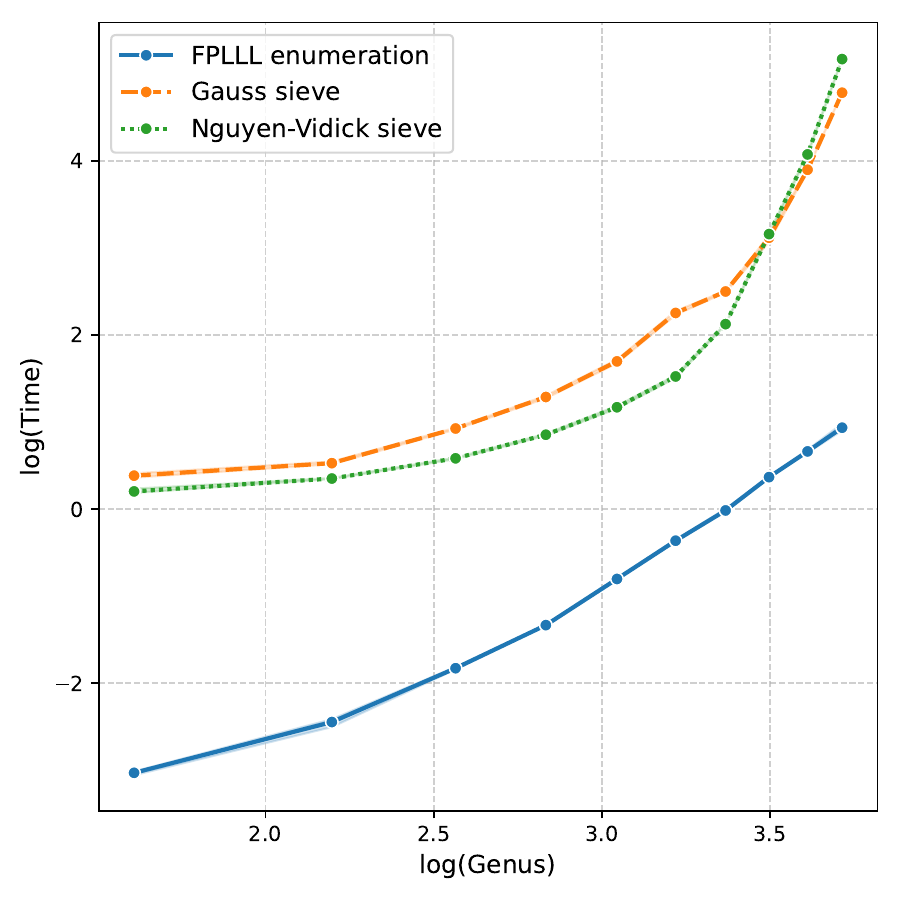}
        \caption{$\log(\text{Time})$--$\log(\#\text{Genus})$ plot}
        \label{fig:time-genus-cvp}
    \end{subfigure}
    
    \caption{Computation time for tropical polarization distance matrices. The left/right panel shows a log-log plot of computation time versus the number of nodes/graph genus. Different colors represent algorithms for solving the exact CVP.}
    \label{fig:l2-dist}
\end{figure}

\paragraph{Computing Foster--Zhang Distance Matrices.}

We test the computation time for Foster--Zhang distance matrices using the MIP formulation \Cref{eq:MILP} and compare four publicly available MIP solvers: COIN Branch-and-Cut (CBC)\footnote{\url{https://www.coin-or.org/Cbc/}}; Interior Point Optimizer (IPOPT)\footnote{\url{https://coin-or.github.io/Ipopt/}}; GNU Linear Programming Kit (GLPK)\footnote{\url{https://www.gnu.org/software/glpk/glpk.html}}; and Solving Constraint Integer Programs (SCIP)\footnote{\url{https://www.scipopt.org/}}. 

To test the computation time with respect to the number of nodes, we generate graphs with $n=20\sim 100$, while keeping the genus fixed at $g=10$. For each graph, we compute the tropical Abel--Jacobi transform of the node set and then compute the Foster--Zhang distance matrices using the aforementioned MIP solvers. \Cref{fig:time-nodes-fs} presents the log-log plot of computation time versus the number of nodes, from which we see that  the algorithm is approximately \emph{cubic} to the number of nodes. This is due to the fact that the number of constraints in a single MIP is $O(n)$ by \Cref{eq:MILP}, and there are $O(n^2)$ MIPs to solve for a distance matrix.

\begin{figure}[htbp]
    \centering
    \begin{subfigure}{0.45\linewidth}
        \includegraphics[width=\textwidth]{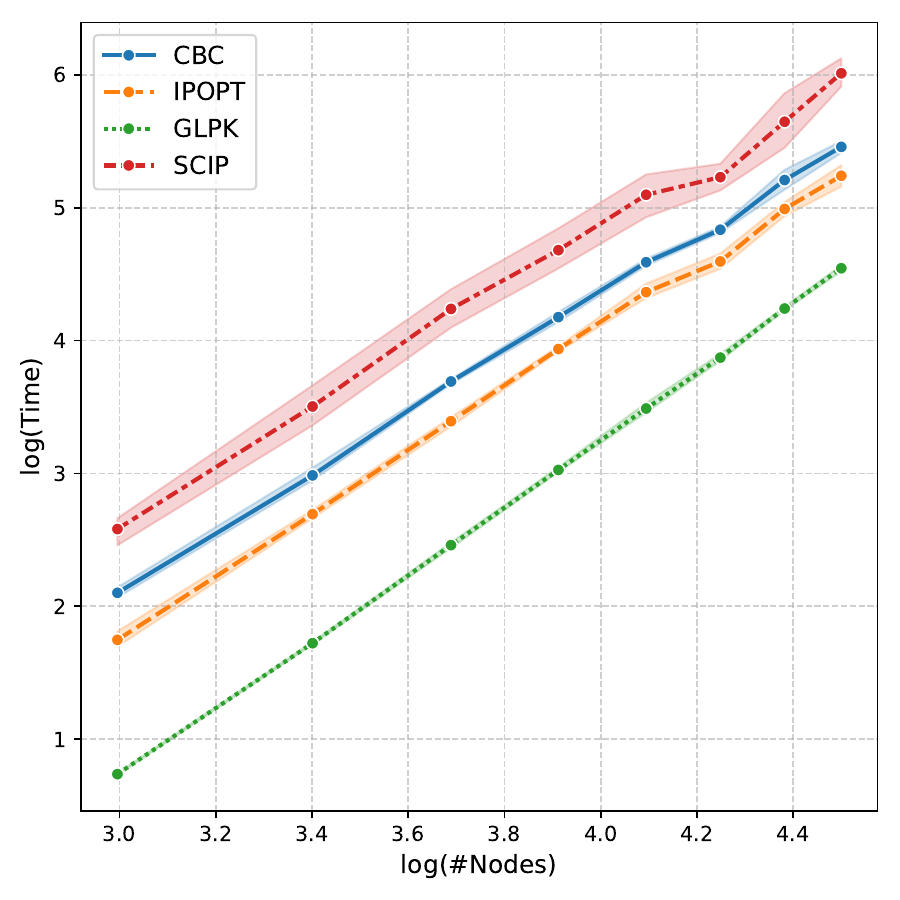}
        \caption{$\log(\text{Time})$--$\log(\#\text{Nodes})$ plot}
        \label{fig:time-nodes-fs}
    \end{subfigure}
    \begin{subfigure}{0.45\linewidth}
        \includegraphics[width=\textwidth]{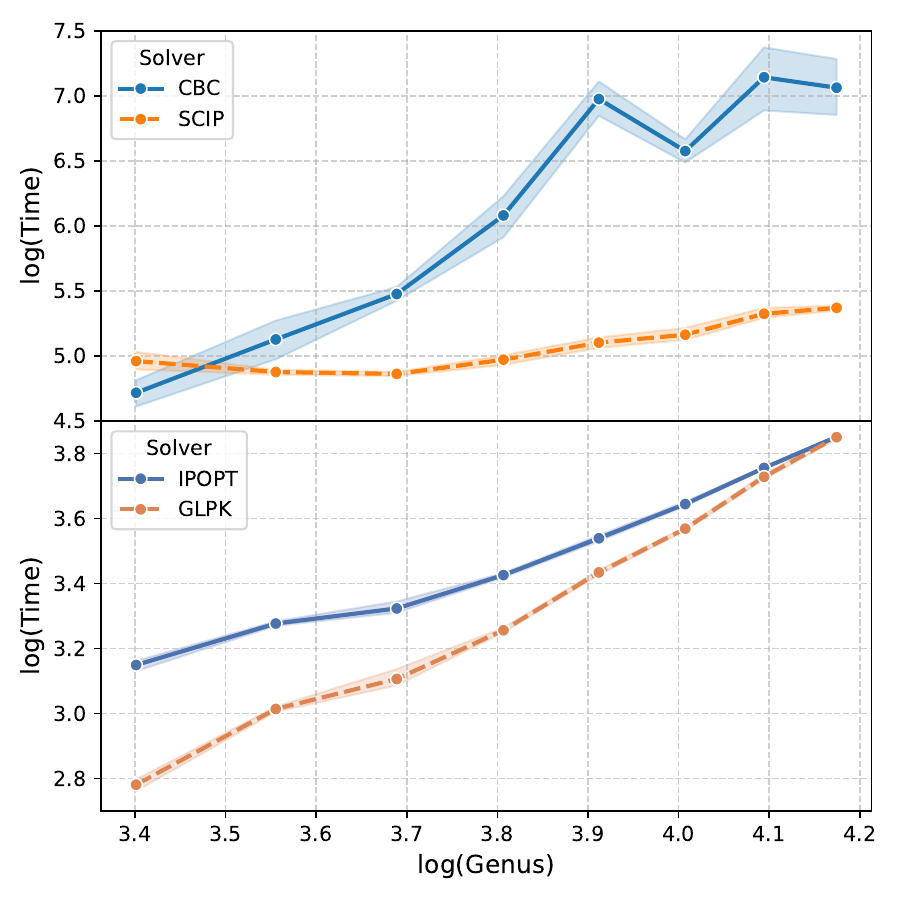}
        \caption{$\log(\text{Time})$--$\log(\#\text{Genus})$ plot}
        \label{fig:time-genus-fs}
    \end{subfigure}
    
    \caption{Computation time for Foster--Zhang distance matrices. The left/right panel shows a log-log plot of computation time versus the number of nodes/graph genus. The colors represent different MIP solvers.}
    \label{fig:fs-dist}
\end{figure}

To test the computation time with respect to genus, we generate graphs with genus $g=30\sim 70$, while fixing the number of nodes at $n=30$. We then compute the tropical Abel--Jacobi transform and the Foster--Zhang distance matrices. \Cref{fig:time-genus-fs} shows the log-log plot of computation time versus genus.  The theoretical complexity of solving an MIP is exponential with respect to the number of variables \emph{in the worst case}. In practice, the computation time  varies depending on several factors, such as initialization, heuristic strategies, and stopping criteria. As a result, the plot does not exhibit a consistent pattern across the four solvers.

\paragraph{Approximation by Babai's Algorithms.}

We analyze the computation time and approximation error of Babai's algorithms for approximating the tropical polarization distance matrices.

To test the computation time with respect to the number of nodes, we generate graphs with $n=20\sim 120$, while keeping the genus fixed at $g=15$. To test the computation time with respect to genus, we generate graphs with genus $g=100\sim 600$, while fixing the number of nodes at $n=40$. For each graph, we compute the tropical Abel--Jacobi transform and the tropical polarization distance matrices using Babai's algorithms, as described in \Cref{subsec:truncated}. The log-log plots are shown in \Cref{fig:babai-time}, and are consistent with the fact that Babai's algorithms run in polynomial time. Since Babai's algorithms only depend on graph genus, the total time complexity of computing a distance matrix is $O(n^2g^2+g^3)$.

\begin{figure}[htbp]
    \centering
    \begin{subfigure}{0.45\linewidth}
        \includegraphics[width=\textwidth]{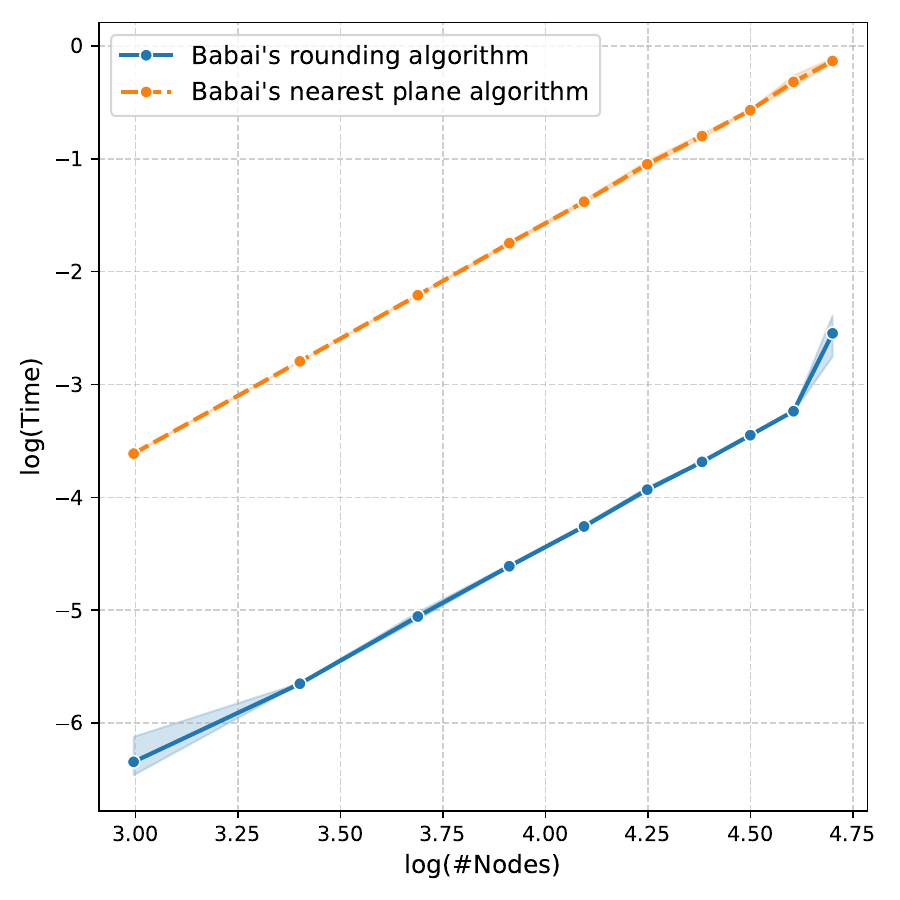}
        \caption{$\log(\text{Time})$--$\log(\#\text{Nodes})$ plot}
        \label{fig:time-nodes-babai}
    \end{subfigure}
    \begin{subfigure}{0.45\linewidth}
        \includegraphics[width=\textwidth]{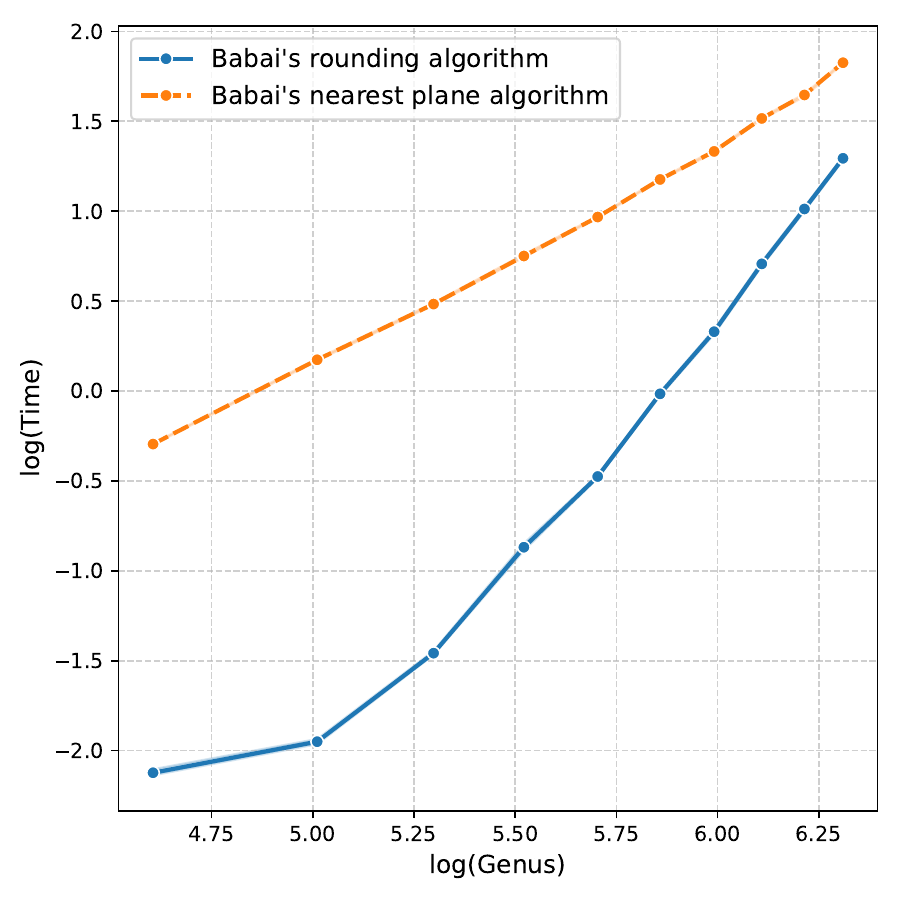}
        \caption{$\log(\text{Time})$--$\log(\#\text{Genus})$ plot}
        \label{fig:time-genus-babai}
    \end{subfigure}
    
    \caption{Computation time of Babai's algorithms. The left/right panel shows a log-log plot of computation time versus the number of nodes/graph genus.}
    \label{fig:babai-time}
\end{figure}

We also analyze the approximation error of Babai's algorithms with respect to the number of nodes and graph genus. To test the approximation error with respect to the number of nodes, we use the same hyperparameters as in the experiment of computation time. To test the approximation error with respect to graph genus, we generate graphs of genus $g=5\sim 45$ for a fixed number of nodes $n=50$. For each graph, we use Babai's algorithm to compute the approximated tropical polarization distance matrix $\widehat{M}$, and use the enumeration algorithm to compute the true distance matrix $\matM$. For each set of hyperparameters, we repeat the computation for $T=10$ times and compute the mean squared error (MSE),
$$
\mathrm{MSE} = \frac{1}{T}\sum_{i=1}^{T} \|\widehat{M}_i-\matM\|_F^2 ,
$$
where $\|\cdot\|_F$ is the Frobenius norm. \Cref{fig:babai-error} presents a comparison of accuracy of Babai's rounding algorithm and Babai's nearest plane algorithm. Together with \Cref{fig:babai-time}, the experimental results aligns with theory that Babai's nearest plane algorithm is more accurate, while slower, than Babai's rounding algorithm.

\begin{figure}[htbp]
    \centering
    \begin{subfigure}{0.45\linewidth}
        \includegraphics[width=\textwidth]{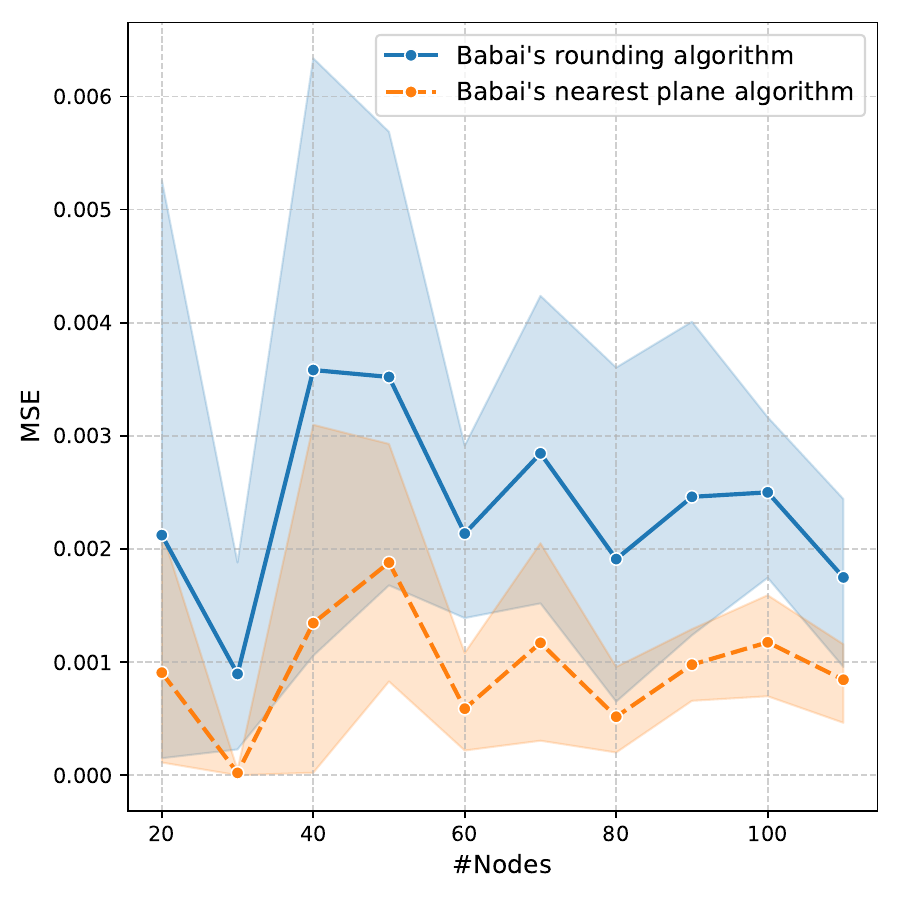}
    \end{subfigure}
    \begin{subfigure}{0.45\linewidth}
        \includegraphics[width=\textwidth]{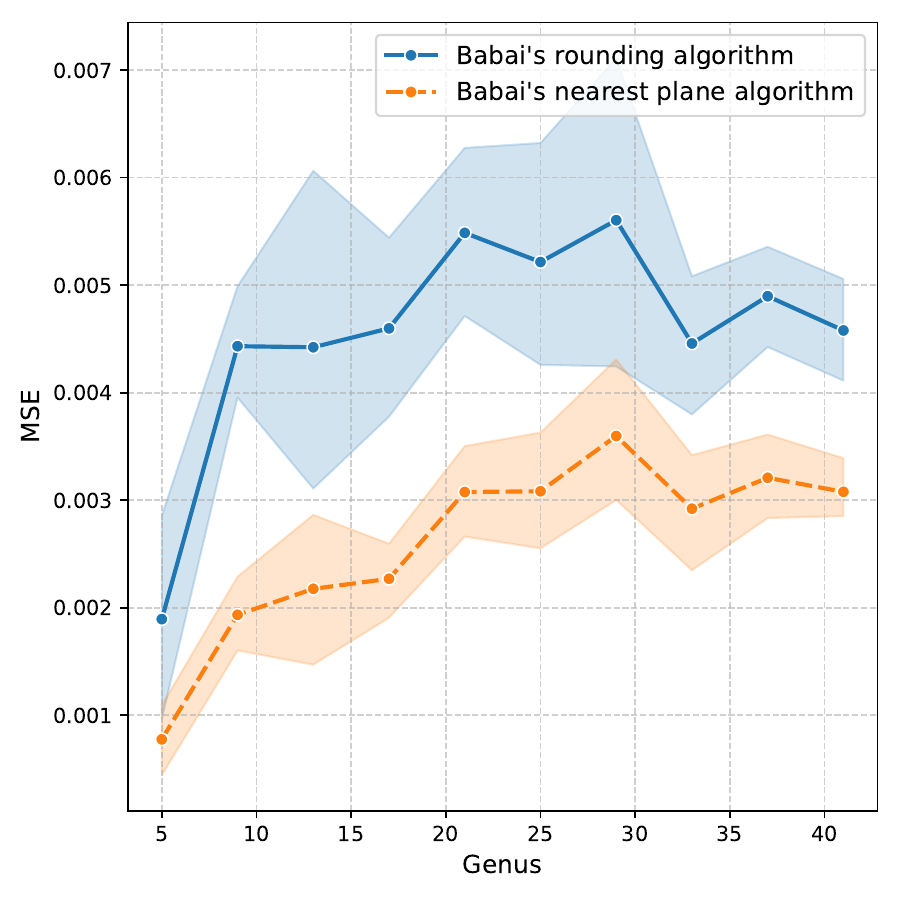}
    \end{subfigure}
    
    \caption{Approximation error of tropical polarization distance matrices using Babai's algorithms. The left panel shows MSE with respect to the number of nodes, while the right panel shows the MSE with respect to graph genus. The colors represent different algorithms.}
    \label{fig:babai-error}
\end{figure}

\paragraph{Code Availability.}

The code for all numerical experiments is available at \url{https://github.com/YueqiCao/Tropical-Abel-Jacobi}

\section{Discussion and Future Work}
\label{sec:discussion}

In this paper, we developed computational methods for the tropical Abel--Jacobi transform of metric graphs and explored associated distance functions on the tropical Jacobian, laying a foundation for applications to real-world problems in machine learning and data science for metric graphs as data structures. Our work bridges tropical geometry, computational mathematics, and computational complexity theory, and inspires new future directions in metric geometry, topological data analysis, and mathematical statistics.

\paragraph{Generalized Abel--Jacobi Map.}

In the smooth setting, the Abel--Jacobi map can be generalized to any Riemannian manifold \citep[Section 4.21]{gromov1999metric}: Let $\mathcal{X}$ be a compact Riemannian manifold such that $H_1(\mathcal{X};\mathbb{Z})$ has no torsion. Let $\Omega(\mathcal{X})$ be the vector space of closed 1-forms on $\mathcal{X}$. Fix a base point $q\in\mathcal{X}$. For any $p\in\mathcal{X}$, let $\gamma$ be a path joining $q$ and $p$,  the \emph{generalized Abel--Jacobi map} for $\mathcal{X}$ is defined as 
$$
\begin{aligned}
    \ajmap: \mathcal{X}&\to \Omega^*(\mathcal{X})/H_1(\mathcal{X};\mathbb{Z})\\
    p&\mapsto \int_\gamma \quad \big(\mathrm{mod}\, H_1(\mathcal{X};\mathbb{Z})\big).
\end{aligned}
$$
There is no canonical choice of ``polarization'' on the real Jacobian $\jac(\mathcal{X})=\Omega^*/H_1(\mathcal{X};\mathbb{Z})$. However, \cite{gromov1999metric} showed that there is always a well-defined left-invariant Finsler metric on $\jac(\mathcal{X})$ induced from the Riemannian metric on $\mathcal{X}$. 

In the discrete setting, the generalized Abel--Jacobi map offers a natural extension from the Abel--Jacobi transform of metric graphs to higher dimensional polyhedral spaces \citep{burago2001course}.  This extension would introduce a range of open problems in computational geometry, particularly in the study of periodic geometric structures and their algorithmic properties.

\paragraph{Topological Data Analysis.}

In machine learning tasks, metric graphs remain challenging to study computationally, for example, when comparing two metric graphs or understanding clustering behavior on a single metric graph.  Recent work has bypassed this difficulty by adapting methods from \emph{topological data analysis} (TDA) \citep{oudot2021barcode,dey2015comparing,gasparovic2018complete}. A key tool in TDA is \emph{persistent homology}, which measures the significance, or persistence, of these features. However, persistent homology itself is a computationally expensive technique. Metric graphs thesmelves are also computationally complex data objects, due to different scales of their combinatorial models.

The tropical Abel--Jacobi map induces a map $\mathcal{F}_*$ from the persistent homology of $\Gamma$ to the persistent homology of $\mathcal{F}(\Gamma)$, whose properties remain unknown. In computations, existing algorithms to compute the persistent homology of $\Gamma$ lack efficiency, while the computation of the persistent homology of $\mathcal{F}(\Gamma)$ is more tractable due to its reduction to point clouds. This motivates a deeper investigation into the properties of $\mathcal{F}_*$ and the development of efficient algorithms for computing the persistent homology of the tropical Abel--Jacobi transform.

\paragraph{Statistical Inference on Metric Graphs with Tropical Probability Measures.}

The tropical Abel--Jacobi map on $\Gamma$ naturally extends to its $g$-fold symmetric product $\Gamma^{(g)}$, which is defined as the quotient of the cartesian product $\Gamma^g$ by the action of the symmetric group of degree $g$. For any  probability measure $\mu$ on the tropical Jacobian $\jac(\Gamma)$, the pullback $\ajmap^*(\mu)$ is a well-defined probability measure on $\Gamma^{(g)}$, which we refer to as a \emph{tropical probability measure}. 

This tropical probability measure provides a novel statistical framework for analyzing unstructured data supported on metric graphs. Since $\jac(\Gamma)$ has a well-understood structure as a flat torus, statistical problems that are difficult to address directly on $\Gamma$ can be reformulated as problems on the $\jac(\Gamma)$, where established statistical techniques for toroidal data such as regression, hypothesis testing, and Bayesian inference can be applied more effectively \citep{garcia2019langevin,xu2023density}.  The tropical probability measure allows for the definition and analysis of random processes on metric graphs, such as Gaussian processes \citep{bolin2024gaussian}, facilitating tasks such as network-based statistical learning and uncertainty quantification.

\section*{Acknowledgments}

We would like to thank Matthew Baker, Yue Ren, and Bernd Sturmfels for helpful discussions. We would like to thank Alessandro Micheli for his careful reading and suggestions for this paper. Y.C. would like to thank Shiqiang Zhang for his help with MIP solvers.

Y.C. is funded by Digital Futures Postdoctoral Fellowship. A.M. is supported by the the UKRI EPSRC grant [EP/Y028872/1], Mathematical Foundations of Intelligence: An “Erlangen Programme” for AI.

\bibliographystyle{chicago}
\bibliography{references}
\newpage

%%%%%%%%%%%%%%%%%%%%%%%%%%%%%%%%%%%%%%%%%%%%%%%%%%
% \section{Applications}\label{sec:app}
% \subsection{Metric Graph Embedding and Dimension Reduction}
% \subsection{Topological Data Analysis}

\begin{appendices}

\section{Technical Proofs}\label{app:proof}

\subsection{Proof of \Cref{thm:subdivision}}\label{app:proof-linear-alg}

   Note that in \Cref{alg:cycle--edge}, we always sort edges in the spanning tree to the front of the edge list. After edge subdivision, we need to reindex edges in the new model $G'$. Suppose an edge $e\in E(G)$ is subdivided into $e'$ and $e''$. Without loss of generality, we index the edges in $E(G')$ in the following way: 
\begin{enumerate}
    \item If $e\in\mathrm{ST}$, index the edge whose terminal point is $p_{n_G+1}$ by $j$ and the other edge by $n_G$. Preserve indices of other edges in $\mathrm{ST}$, and increase indices of edges not in $\mathrm{ST}$ by one;
    \item If $e\notin \mathrm{ST}$, then one of $e'$ and $e''$ should be added to $\mathrm{ST}$ to form the new spanning tree $\mathrm{ST}'$ in $G'$. Index the edge in $\mathrm{ST}'$ by $n_G$ and the other edge by $j+1$. Preserve indices of edges in $\mathrm{ST}$, and increase indices of other edges not in $\mathrm{ST}$ by one.
\end{enumerate}

We show an visual illustration of the reindexing rule in \Cref{fig:re-index}.

\begin{figure}[H]
    \centering
    \begin{subfigure}{0.3\linewidth}
        \includegraphics[width=0.9\linewidth]{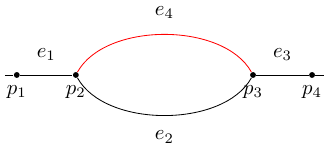}
        \caption{Original model $G$}
    \end{subfigure}
    \begin{subfigure}{0.3\linewidth}
        \includegraphics[width=0.9\linewidth]{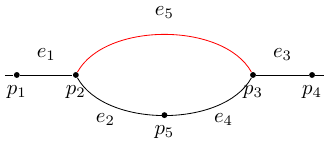}
        \caption{Subdivide $e_2\in \mathrm{ST}$}
    \end{subfigure}
    \begin{subfigure}{0.3\linewidth}
        \includegraphics[width=0.9\linewidth]{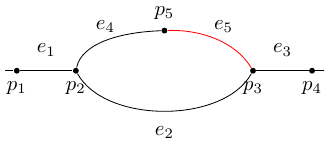}
        \caption{Subdivide $e_4\notin \mathrm{ST}$}
    \end{subfigure}
    \caption{An illustration of edge subdivision and reindexing. (a) The original combinatorial model $G$. The spanning tree $\mathrm{ST}$ is in black. The edge $e_4$ not in $\mathrm{ST}$ defines a cycle and is colored in red; (b) A subdivision of $e_2\in\mathrm{ST}$. Two new edges are indexed by $e_2$ and $e_4$. The edge not in $\mathrm{ST}$ is reindexed as $e_5$; (c) A subdivision of $e_4\notin \mathrm{ST}$. One of the new edges is indexed as $e_4$ and  added to the spanning tree, while the other edge defining the same cycle is indexed as $e_5$.}
    \label{fig:re-index}
\end{figure}

\begin{enumerate}
    \item If $e_j\in\mathrm{ST}$, for cycle--edge incidence matrices we can check that 
    \begin{equation}\label{eq:cycleedge-subdv}
      \begin{cases}
        \widetilde{\matC}[:,k] = \matC[:,k],& k=1,\ldots,n_G-1\\
        \widetilde{\matC}[:,l] = \matC[:,l-1],& l=n_G+1,\ldots,m_G+1\\
        \widetilde{\matC}[:,n_G] = \matC[:,j],&
    \end{cases}  
    \end{equation}
    for path--edge incidence matrices we have
    \begin{equation}\label{eq:pathedge-subdv}
        \begin{cases}
        \widetilde{\matY}[i,k] = \matY[i,k], & i=1,\ldots,n_G,\, k=1,\ldots,n_G-1\\

        \widetilde{\matY}[:,l]=\matY[:,l-1],& l=n_G+1,\ldots,m_G+1\\
        \widetilde{\matY}[i,n_G] = \matY[i,j], & i=1,\ldots,n_G\\
        \widetilde{\matY}[n_G+1,k] = \matY[j_+,k], & k=1,\ldots,n_G-1\\
        \widetilde{\matY}[n_G+1,n_G] = 0,
    \end{cases}   
    \end{equation}
    and for edge length matrices we have
    \begin{equation}\label{eq:edgelength-subdv}
      \begin{cases}
        \widetilde{\matL}[k,k] = \matL[k,k], & k\neq j, k\le n_G-1\\
        \widetilde{\matL}[l.l] = \matL[l-1,l-1], & l=n_G+1,\ldots,m_G+1\\
        \widetilde{\matL}[j,j]+\widetilde{\matL}[n_G,n_G] = \matL[j,j] .
    \end{cases}  
    \end{equation}
    For $i=1,\ldots,n_G$, using \cref{eq:cycleedge-subdv,eq:pathedge-subdv,eq:edgelength-subdv}, we can derive that
    $$
    \begin{aligned}
        \widetilde{\matV}[:,i] &= \widetilde{\matC}_{\mathrm{ST}}\widetilde{\matL}_{\mathrm{ST}}\widetilde{\matY}_{\mathrm{ST}}^\top[:,i]\\
        & = \sum_{k=1}^{n_G}\widetilde{\matC}[:,k]\widetilde{\matL}[k,k]\widetilde{\matY}[i,k]\\
        & = \sum_{\substack{k\neq j\\ k\le n_G-1 }}\matC[:,k]\matL[k,k]\matY[i,k] + \matC[:,j]\widetilde{\matL}[j,j]\matY[i,j] + \widetilde{\matC}[:,n_G]\widetilde{\matL}[n_G,n_G]\widetilde{\matY}[i,n_G]\\
        & = \sum_{\substack{k\neq j\\ k\le n_G-1}}\matC[:,k]\matL[k,k]\matY[i,k] + \matC[:,j]\bigg(\widetilde{\matL}[j,j]+\widetilde{\matL}[n_G,n_G]\bigg)\matY[i,j]\\
        & = \sum_{k=1}^{n_G-1}\matC[:,k]\matL[k,k]\matY[i,k] = \matV[:,i] .
    \end{aligned}
    $$
For the new point we can compute
$$
\begin{aligned}
    \widetilde{\matV}[:,n_G+1] &= \widetilde{\matC}_{\mathrm{ST}}\widetilde{\matL}_{\mathrm{ST}}\widetilde{\matY}_{\mathrm{ST}}^\top[:,n_G+1]\\
     & = \sum_{k=1}^{n_G}\widetilde{\matC}[:,k]\widetilde{\matL}[k,k]\widetilde{\matY}[n_G+1,k]\\
     & = \sum_{k=1}^{n_G-1}\matC[:,k]\widetilde{\matL}[k,k]\matY[j_+,k]\\
     & = \sum_{\substack{k\neq j\\ k\le n_G-1}}\matC[:,k]\matL[k,k]\matY[j_+,k] + \matC[:,j]\widetilde{\matL}[j,j]\matY[j_+,j] .
\end{aligned} 
$$
Since $\theta = \widetilde{\matL}[j,j]/\matL[j,j]$, and note that $\matY[j_-,k]=\matY[j_+,k]$ for $k\neq j$ and $k\le n_G-1$,
$$
\begin{aligned}
    \widetilde{\matV}[:,n_G+1] =&{} (1-\theta)\sum_{\substack{k\neq j\\ k\le n_G-1}}\matC[:,k]\matL[k,k]\matY[j_-,k] + \theta \sum_{\substack{k\neq j\\ k\le n_G-1}}\matC[:,k]\matL[k,k]\matY[j_-,k]\\
    & + \theta\matC[:,j]\matL[j,j]\matY[j_+,j]\\
    =&{} (1-\theta)\sum_{k=1}^{n_G-1} \matC[:,k]\matL[k,k]\matY[j_-,k] +\theta \sum_{k=1}^{n_G-1}\matC[:,k]\matL[k,k]\matY[j_+,k]\\
    =&{} (1-\theta)\matV[:,j_-] + \theta \matV[:,j_+]  .
\end{aligned}
$$
For any $1\le i,l\le g$, we have
$$
\begin{aligned}
    \widetilde{\matQ}[i,l] &= \sum_{k=1}^{m_G}\widetilde{\matC}[i,k]\widetilde{\matL}[k,k]\widetilde{\matC}^\top[k,l]\\
    &= \sum_{k=1}^{n_G-1}\matC[i,k]\widetilde{\matL}[k,k]\matC[l,k]+\widetilde{\matC}[i,n_G]\widetilde{\matL}[n_G,n_G]\widetilde{\matC}[l,n_G] + \sum_{k=n_G+1}^{m_G+1}\widetilde{\matC}[i,k]\widetilde{\matL}[k,k]\widetilde{\matC}[l,k]\\
    & = \sum_{k=1}^{n_G-1}\matC[i,k]\widetilde{\matL}[k,k]\matC[l,k]+\matC[i,j]\widetilde{\matL}[n_G,n_G]\matC[l,j] + \sum_{k=n_G}^{m_G}\matC[i,k]\matL[k,k]\matC[l,k]\\
    & = \sum_{k=1}^{m_G}\matC[i,k]\matL[k,k]\matC[l,k] = \matQ[i,l] .
\end{aligned}
$$
\item If $e_j\notin \mathrm{ST}$, for cycle--edge incidence matrices \cref{eq:cycleedge-subdv} still holds. For path--edge incidence matrices, if the edge whose terminal point is $p_{n_G+1}$ is added to the spanning tree. Then
\begin{equation}\label{eq:case1}
 \begin{cases}
    \widetilde{\matY}[i,k] = \matY[i,k], & i=1,\cdots,n_G,\, k=1,\ldots, n_G-1\\
    \widetilde{\matY}[i,n_G] = 0, & i=1,\ldots,n_G\\
    \widetilde{\matY}[n_G+1,k] = \matY[j_-,k], & k=1,\ldots,n_G-1\\
    \widetilde{\matY}[n_G+1,n_G] = 1 ,
\end{cases}   
\end{equation}
otherwise it is given by
\begin{equation}\label{eq:case2}
    \begin{cases}
    \widetilde{\matY}[i,k] = \matY[i,k], & i=1,\ldots,n_G,\, k=1,\ldots, n_G-1\\
    \widetilde{\matY}[i,n_G] = 0, & i=1,\ldots,n_G\\
    \widetilde{\matY}[n_G+1,k] = \matY[j_+,k], & k=1,\ldots,n_G-1\\
    \widetilde{\matY}[n_G+1,n_G] = -1 ,  
    \end{cases}
\end{equation}
and for edge length matrices we have
$$
\begin{cases}
    \widetilde{\matL}[k,k] = \matL[k,k], & k\le n_G-1\\
    \widetilde{\matL}[l,l] =  \matL[l-1,l-1], & l\neq j,\, l=n_G+1,\ldots,m_G+1\\
    \widetilde{\matL}[j+1,j+1]+\widetilde{\matL}[n_G,n_G] = \matL[j,j] .
\end{cases}
$$
For $i=1,\ldots,n_G$ we can compute 
$$
\begin{aligned}
    \widetilde{\matV}[:,i] &= \widetilde{\matC}_{\mathrm{ST}}\widetilde{\matL}_{\mathrm{ST}}\widetilde{\matY}_{\mathrm{ST}}^\top[:,i]= \sum_{k=1}^{n_G}\widetilde{\matC}[:,k]\widetilde{\matL}[k,k]\widetilde{\matY}[i,k]\\
    &= \sum_{k=1}^{n_G-1}\matC[:,k]\matL[k,k]\matY[i,k] + \widetilde{\matC}[:,n_G]\widetilde{\matL}[n_G,n_G]\widetilde{\matY}[i,n_G] = \matV[:,i] .
\end{aligned}
$$
In case \cref{eq:case1}, we have $\widetilde{\matL}[n_G,n_G]/\matL[j,j] =\theta$, and
$$
\begin{aligned}
    \widetilde{\matV}[:,n_G+1] &=  \widetilde{\matC}_{\mathrm{ST}}\widetilde{\matL}_{\mathrm{ST}}\widetilde{\matY}_{\mathrm{ST}}^\top[:,n_G+1]\\
     & = \sum_{k=1}^{n_G}\widetilde{\matC}[:,k]\widetilde{\matL}[k,k]\widetilde{\matY}[n_G+1,k]\\
     & = \sum_{k=1}^{n_G-1}\matC[:,k]\matL[k,k]\widetilde{\matY}[n_G+1,k] + \widetilde{\matC}[:,n_G]\widetilde{\matL}[n_G,n_G]\widetilde{\matY}[n_G+1,n_G]\\
     & = \sum_{k=1}^{n_G-1}\matC[:,k]\matL[k,k]\matY[j_-,k] + \theta\matC[:,j]\matL[j,j] = \matV[:,j_-]+\theta\mathbf{w} .
\end{aligned}
$$
In case \cref{eq:case2}, we have $\widetilde{\matL}[n_G,n_G]/\matL[j,j] =1-\theta$, and
$$
\begin{aligned}
    \widetilde{\matV}[:,n_G+1] 
     & = \sum_{k=1}^{n_G-1}\matC[:,k]\matL[k,k]\widetilde{\matY}[n_G+1,k] + \widetilde{\matC}[:,n_G]\widetilde{\matL}[n_G,n_G]\widetilde{\matY}[n_G+1,n_G]\\
     & = \sum_{k=1}^{n_G-1}\matC[:,k]\matL[k,k]\matY[j_+,k] - (1-\theta)\matC[:,j]\matL[j,j] = \matV[:,j_+]-(1-\theta)\mathbf{w} .    
\end{aligned}
$$
For any $1\le i,l\le g$, we have
$$
\begin{aligned}
    \widetilde{\matQ}[i,l] =&{} \sum_{k=1}^{m_G}\widetilde{\matC}[i,k]\widetilde{\matL}[k,k]\widetilde{\matC}^\top[k,l]\\
    =&{} \sum_{k=1}^{n_G-1}\matC[i,k]\widetilde{\matL}[k,k]\matC[l,k]+\widetilde{\matC}[i,n_G]\widetilde{\matL}[n_G,n_G]\widetilde{\matC}[l,n_G]\\
    +& \sum_{\substack{k\neq j+1\\ k\le m_G+1}}\widetilde{\matC}[i,k]\widetilde{\matL}[k,k]\widetilde{\matC}[l,k] + \widetilde{\matC}[i,j+1]\widetilde{\matL}[j+1,j+1]\widetilde{\matC}[l,j+1]\\
    =&{} \sum_{k=1}^{n_G-1}\matC[i,k]\matL[k,k]\matC[l,k]+\matC[i,j]\widetilde{\matL}[n_G,n_G]\matC[l,j]\\
    +& \sum_{\substack{k\neq j\\ k\le m_G}}\matC[i,k]\matL[k,k]\matC[l,k] + \matC[i,j]\widetilde{\matL}[j+1,j+1]\matC[l,j]\\
    =&{} \sum_{k=1}^{m_G}\matC[i,k]\matL[k,k]\matC[l,k] = \matQ[i,l] ,
\end{aligned}
$$
which completes the proof.
\end{enumerate}

\subsection{Proof of \Cref{thm:rect-trop-jac}}\label{app:proof-rect-trop}

The proof of \Cref{thm:rect-trop-jac} relies on the injectivity of the tropical Torelli map, which is defined as a map from the moduli space of tropical curves to the moduli space of principally polarized tropical Abelian varieties. We will not need the full theory of tropical Torelli maps, instead, the following theorem suffices for our use.

\begin{theorem}{\citep[Theorem 5.3.3]{brannetti2011tropical}, \citep[Theorem 4.1.9]{caporaso2010torelli}}\label{thm:trop-torelli}
    Let $\Gamma$ and $\Gamma'$ be two metric graphs. Then $\jac(\Gamma)$ and $\jac(\Gamma')$ are isometric under the tropical polarization distance if and only if the 3-edge-connectivizations of any combinatorial models of $\Gamma$ and $\Gamma'$ are 2-isomorphic.
\end{theorem}

We now introduce the relevant definitions. By convention, all graphs are assumed to be connected.

\begin{definition}
    Two combinatorial graphs $G$ and $G'$ are said to be \emph{2-isomorphic} if there exists a bijection between $E(G)$ and $E(G')$ which further induces a bijection between the sets of cycles of $G$ and $G'$.
\end{definition}

\begin{definition}
    Let $G$ be a combinatorial graph and $\widetilde{G}$ be the graph by contracting all bridge edges from $G$. A set of edges $S\subseteq E(\widetilde{G})$ is a \emph{C1-set} of $G$ if $\widetilde{G}\backslash S$ has no bridges and contracting all edges not in $S$ yields a cycle. The set of all C1-sets of $G$ is denoted by $\mathrm{Set}^1(G)$.
\end{definition}

\begin{lemma}{\citep[Lemma 2.3.2]{caporaso2010torelli}}\label{lemma:C1-set}
     Let $G$ be a combinatorial graph. Each non-bridge edge belongs to a unique C1-set of $G$. Suppose $e,e'\in E(G)$ are not bridge edges. Then the following are equivalent:
     \begin{enumerate}[(i)]
         \item $e$ and $e'$ belong to the same C1-set of $G$;
         \item $e$ and $e'$ belong to the same cycles of $G$;
         \item $G\backslash\{e,e'\}$ is disconnected. 
     \end{enumerate}
 \end{lemma}

 \Cref{lemma:C1-set} implies that the C1-sets form a partition of the set of non-bridge edges. Note that though contracting all edges outside a C1-set yields a cycle, a C1-set itself may not form a cycle. See \Cref{fig:c1-set}.

 \begin{definition}
     Let $G$ be a combinatorial graph. A \emph{3-edge-connectivization} of $G$ is a graph, denoted by $G^3$, obtained from $G$ by contracting all bridges in $G$ and all but one among the edges of each C1-set of $G$. If $G$ is weighted, then the weight function $\ell^3$ on $G^3$ is given by $\ell^3(e_S) = \sum_{e\in S}\ell(e)$ for each C1-set $S$ of $G$.
 \end{definition}

The 3-edge-connectivization and C1-sets of a graph are related by the following lemma.

\begin{lemma}{\citep[Lemma 3.2.8]{caporaso2010torelli}}\label{lemma:3-edge}
    Let $G$ be a combinatorial graph. 
    \begin{enumerate}[(i)]
        \item $\dim H_1(G^3;\mathbb{R}) = \dim H_1(G;\mathbb{R})$;
        \item There are canonical bijections
        $
        \mathrm{Set}^1(G^3)\leftrightarrow E(G^3)\leftrightarrow \mathrm{Set}^1(G)
        $;
        \item Any two 3-edge-connectivizations of $G$ are 2-isomorphic.
    \end{enumerate}
\end{lemma}

%An example illustrating 3-edge-connectivization and C1-sets is given in \Cref{fig:c1-set}.

\begin{figure}
    \centering
    \includegraphics[width=0.7\linewidth]{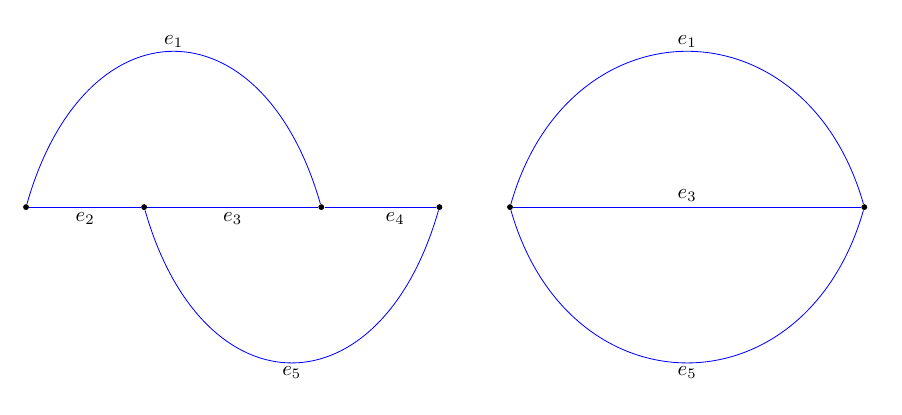}
    \caption{Example of C1-sets and 3-edge-connectivization. On the left panel, the C1-sets of $G$ are given by $\{e_1,e_2\},\{e_3\},\{e_4,e_5\}$. By contracting $e_2$ and $e_4$, the 3-edge-connectivization $G^3$ is shown on the right panel. The C1-sets of $G^3$ are singletons.}
    \label{fig:c1-set}
\end{figure}

We are now ready to prove \Cref{thm:rect-trop-jac}.

\begin{proof}
    If $\widetilde{G}$ admits a cycle decomposition, by construction $\matQ$ is a diagonal matrix and the lattice generated by $\matQ^{\frac{1}{2}}$ is rectangular. Thus it suffices to prove the converse. By \Cref{lemma:lattice-isometry}, there exists $\matP\in\mathrm{GL}(g;\mathbb{Z})$ and a diagonal matrix $\matD$ such that $\matP^\top\matQ\matP=\matD$. Construct a metric graph $\Gamma_0$ by attaching $g$ self-loops to a single vertex such that self-loops have lengths given by the diagonal entries of $\matD$. It follows that $\jac(\Gamma)$ and $\jac(\Gamma_0)$ are isometric under tropical polarization distance. By \Cref{thm:trop-torelli}, the 3-edge-connectivizations of combinatorial models of $\Gamma$ and $\Gamma_0$ are 2-isomorphic. Note that the construction of $\Gamma_0$ already gives a 3-edge-connected combinatorial graph $G_0$. Fix a combinatorial model $G$ for $\Gamma$. Then $G^3$ is 2-isomorphic to $G_0$.

    We will prove that the C1-sets of $G$ form a cycle decomposition of $\widetilde{G}$. Fix a spanning tree $\mathrm{ST}$ for $G$. By \Cref{lemma:C1-set}, any edge $e\notin\mathrm{ST}$ defines a C1-set $S_e$, and $S_e\neq S_{e'}$ if $e\neq e'$. Assume $S_{e_0}$ is not a cycle for some $e_{0}\notin \mathrm{ST}$. Then there exists $e_0^*$ in the fundamental 1-cycle determined by $e_0$ such that $e_0^*\notin S_{e_0}$. Let $S_{e_0^*}$ be the C1-set containing $e_0^*$. Again by \Cref{lemma:C1-set}, $S_{e_0^*}\neq S_e$ for all $e\notin \mathrm{ST}$, which implies $\#\mathrm{Set}^1(G)>g$. However, since $G$ is 2-isomorphic to $G_0$, by \Cref{lemma:3-edge}, $\#\mathrm{Set}^1(G)=\# E(G_0) = g$, which is a contradiction. Therefore the C1-sets are the support sets of fundamental 1-cycles of $G$. Since C1-sets are disjoint and form a partition of $E(\widetilde{G})$, they form a cycle decomposition of $\widetilde{G}$, which completes the proof.
\end{proof}

\section{Related Theories from Complex and Tropical Geometry}\label{app:theory}
\subsection{Abel--Jacobi Theory for Various Data}\label{app:aj-theory}

An Abel–-Jacobi type theory connects divisors (formal sums of points) on a geometric object with another group object, called the \emph{Jacobian group} or \emph{Jacobian variety}. It revolves around the \emph{Abel–-Jacobi map}, which sends a divisor of degree zero to a point in the Jacobian by integrating ``1-forms'' along paths. Depending on the type of data, these notions are defined differently, yet in the end the theorems are alike. In this section, we summarize the Abel--Jacobi theory for Riemann surfaces (complex algebraic curves), combinatorial graphs, and metric graphs. The main purpose is to present the differences as well as similarities of the theory among different settings. More details in this section can be found in \cite{griffiths1989introduction,baker2011metric,mikhalkin2008tropical}.

%for Riemann surfaces see \cite{griffiths1989introduction}, for combinatorial graphs see \cite{baker2011metric}, for metric graphs see \cite{mikhalkin2008tropical}.

\paragraph{The Abel--Jacobi Theory for Riemann Surfaces.}

Let $\mathscr{C}$ be a compact Riemann surface of genus $g$. A \emph{divisor} $D$ on $\mathscr{C}$ is a formal finite sum of points with integer coefficients
$$
D=\sum n_ip_i \, ,
$$
where $n_i\in\mathbb{Z}$ and $p_i\in\mathscr{C}$. The \emph{degree} of $D$ is defined as the sum of its coefficient $\deg(D) = \sum_i n_i$. The set of all divisors on $\mathscr{C}$ forms an abelian group $\Div(\mathscr{C})$ under formal addition, called the \emph{divisor group}. For any nonzero meromorphic function $f$ on $\mathscr{C}$, let $(f)$ be the divisor defined by
$$
(f) = \sum_{p\in\mathscr{C}}\mathrm{ord}_p(f)p\, ,
$$
where $\mathrm{ord}_p(f)$ is the order of vanishing of $f$ at $p$. Any divisor in the form $(f)$ is called a \emph{principal divisor}, and the set of all principal divisors form a subgroup $\pdiv(\mathscr{C})$ of $\Div(\mathscr{C})$. For any two divisors $D,E\in \Div(\mathscr{C})$, $D$ is \emph{linear equivalent} to $E$ if they differ by a principal divisor, i.e., there exists a meromorphic function $f$ such that $D = E+(f)$. Thus a linear equivalence class of divisors is an element in the quotient group $\Div(\mathscr{C})/\pdiv(\mathscr{C})$.

Let $\{(U_i,z_i)\}$ be an atlas of holomorphic coordinate charts on $\mathscr{C}$. A \emph{holomorphic differential 1-form} $\omega$ on $\mathscr{C}$ is a collection of holomorphic functions $f_i:U_i\to\mathbb{C}$ such that over non-empty intersections $U_i\cap U_j$, 
$$
f_i = f_j \frac{\diff{z_j}}{\diff{z_i}} \, .
$$
Let $\Omega(\mathscr{C})$ be the set of all holomorphic differential 1-forms on $\mathscr{C}$. If $\mathscr{C}$ has genus $g$, then $\Omega(\mathscr{C})$ is a complex vector space of dimension $g$. Let $H_1(\mathscr{C};\mathbb{Z})$ be the integral 1-homology group of $\mathscr{C}$. Path integration of holomorphic differential 1-forms yields an injective homomorphism
$$
\begin{aligned}
    H_1(\mathscr{C};\mathbb{Z})&\to \Omega^*(\mathscr{C})\\
    \gamma&\mapsto \int_\gamma
\end{aligned}
$$
Thus $H_1(\mathscr{C};\mathbb{Z})$ can be identified as a full rank lattice of $\Omega^*(\mathscr{C})$. The quotient $\jac(\mathscr{C}) = \Omega^*(\mathscr{C})/H_1(\mathscr{C};\mathbb{Z})$ is called the \emph{Jacobian variety} of $\mathscr{C}$.

% A divisor is called effective if $n_i\ge 0$ and denoted by $D\ge 0$. Fix a divisor $D$. Let $\mathcal{L}(D) = \{f\in\mathcal{K}(C)\mid (f)+D\ge 0 \}$, and $\dim\mathcal{L}(D) = \ell(D)$. Let $\mathcal{K}^1(D) = \{\omega\in\mathcal{K}^1(C)\mid (\omega)\ge D\}$, and $\dim\mathcal{K}^1(D) = i(D)$.

%Let $\{(U_i,z_i)\}$ be holomorphic coordinate charts on $C$. Let $\mathcal{O}(U_i)$ (resp. $\mathcal{K}(U_i)$) be the space of holomorphic (resp. meromorphic) functions on $U_i$. A holomorphic (resp. meromorphic) differential form $\omega$ can be locally represented by $\omega_i = f_i(z_i)\mathrm{d}z_i$ for some $f_i\in \mathcal{O}(U_i)$ (resp. $\mathcal{K}(U_i)$). Let $\Omega^1(C)$ (resp. $\mathcal{K}^1(C)$) be the space of holomorphic (meromorphic) differentials.

% \begin{theorem}[Riemann-Roch]
%     Let $C$ be a Riemann surface with genus $g$ and $D$ be a divisor with degree $d$. Then
%     $\ell(D) = d-g+i(D)+1$.
% \end{theorem}

% \begin{theorem}[Hodge decomposition]
%     For a compact Riemann surface $C$, let $H_{de}^1(C)$ be its de Rham cohomology group. Then
%     $$
%     \Omega^1(C)\oplus \overline{\Omega^1(C)}\cong H^1_{de}(C)
%     $$
% \end{theorem}

% Note that on a compact Kahler manifold (a Riemann surface is a simplest example) holomophic forms are equivalent to harmonic forms on Riemannian manifolds.

Let $\Div^0(\mathscr{C})\subseteq \Div(\mathscr{C})$ be the subgroup of divisors of degree zero. Any divisor $D\in\Div^0(\mathscr{C})$ can be written as a finite sum $D=\sum (p_i-q_i)$ for some $p_i,q_i\in\mathscr{C}$. The \emph{complex Abel--Jacobi map} is defined as
$$
\begin{aligned}
   \mathcal{J}: \Div^0(\mathscr{C})&\to \jac(\mathscr{C})\\
   D&\mapsto \sum_i\int_{q_i}^{p_i} \quad\left(\text{mod } H_1(\mathscr{C};\mathbb{Z})\right) ,
\end{aligned}
$$
where for each summand the integration is taken over any continuous path from $q_i$ to $p_i$.

Historically, Abel first proved that the kernel of $\mathcal{J}$ is the subgroup of all principal divisors in $\Div^0(\mathscr{C})$. Then Jacobi proved that $\mathcal{J}$ is a surjective map which is now known as \emph{Jacobi's Inversion Theorem}. Altogether, we thus have the Abel--Jacobi theorem.

\begin{theorem}[Abel--Jacobi Theorem]
    Let $\mathscr{C}$ be a compact Riemann surface. The Abel--Jacobi map induces a canonical isomorphism 
    $$
    \Div^0(\mathscr{C})/\pdiv(\mathscr{C})\xrightarrow{\sim} \jac(\mathscr{C}) .
    $$
\end{theorem}

In particular, fix a base point $q\in\mathscr{C}$, any other point $p\in\mathscr{C}$ gives rise to a divisor $p-q\in \Div^0(\mathscr{C})$. The Abel--Jacobi map reduces to
$$
\ajmap:\mathscr{C}\to\jac(\mathscr{C}) ,
$$
which is in the form of our interest. It can be shown that for any $g\ge 1$, the Abel--Jacobi map $\ajmap$ is an embedding of complex algebraic varieties.

% and a basis $\gamma_1,\ldots,\gamma_{2g}$ in $H_1(C,\mathbb
% Z)$. The vectors defined by
% $$
% \pi_i = \begin{pmatrix}\int_{\gamma_i}\omega_1\\ \int_{\gamma_i}\omega_2\\ \vdots \\ \int_{\gamma_i}\omega_g\end{pmatrix}
% $$
% are called period vectors of $C$. The matrix $\Pi = [\pi_1,\ldots,\pi_{2g}]$ is called the period matrix of $C$.

% \begin{theorem}
%     The period vectors are $\mathbb{R}$-linearly independent.
% \end{theorem}

% The $2g$ period vectors generate a lattice $\mathbb{L}=\{\sum_i m_i\pi_i\mid m_i\in\mathbb{Z}\}$ in $\mathbb{C}^g$. \emph{The Jacobian variety} of $C$ is defined as $J(C) = \mathbb{C}^g/\mathbb{L}$.

% \begin{definition}[The Abel-Jacobi map]
%     Fix a point $q\in C$, the Abel-Jacobi map
%     $$
%     \mathbf{AJ}:\text{Div}(C)\to J(C)
%     $$
%     is defined by
%     $$
%     \mathbf{AJ}(D) = \begin{pmatrix}
%         \sum_i n_i\int_q^{p_i}\omega_1\\
%         \sum_i n_i\int_q^{p_i}\omega_2\\
%         \vdots\\
%         \sum_i n_i\int_q^{p_i}\omega_g
%     \end{pmatrix}
%     $$
% \end{definition}

% \begin{theorem}
%     Let $\mathcal{K}^*(C)$ be the multiplicative group of $\mathcal{K}(C)$. The following sequence is exact
%     $$
%     \mathcal{K}^*(C)\xrightarrow{(\dot)} \text{Div}^0(C)\xrightarrow{\mathbf{AJ}} J(C)\to 0
%     $$
% \end{theorem}

% Let $\text{Pic}(C)$ be the quotient group $\text{Div}(C)/\text{im}()$, called the Picard variety. We have 
% $$
% \text{Pic}(C)\cong J(C)
% $$

\paragraph{The Abel--Jacobi Theory for Combinatorial Graphs.}

Let $G$ be a combinatorial graph. A \emph{divisor} $D$ on $G$ is a linear combination of vertices with integer coefficients. Thus the divisor group $\Div(G)$ is equal to the 0-chain group $C_0(G;\mathbb{Z})$. Unlike the divisor group of a Riemann surface, the divisor group of a combinatorial graph is a free abelian group of finite rank.

Let $\partial:C_1(G;\mathbb{R})\to  C_0(G;\mathbb{R})$ be the boundary operator. Define inner products on $C_1(G;\mathbb{R})$ and $C_0(G;\mathbb{R})$ by
$$
\begin{aligned}
\langle \alpha_1,\alpha_2 \rangle &= \sum_{e\in E(G)}\alpha_1(e)\alpha_2(e)\ell(e),\, \forall\ \alpha_1,\alpha_2\in C_1(G;\mathbb{R}) ,\\
\langle \varphi_1,\varphi_2 \rangle &= \sum_{v\in V(G)}\varphi_1(v)\varphi_2(v),\, \forall\ \varphi_1,\varphi_2\in C_0(G;\mathbb{R}) ,
\end{aligned}
$$
and let $\partial^*:C_0(G;\mathbb{R})\to C_1(G;\mathbb{R})$ be the adjoint operator (essentially the dual version of \cref{eq:cobound-graph} and \cref{eq:adjoint-cobound-graph}). We have the following Hodge decomposition of the 1-chain group
$$
    C_1(G;\mathbb{R}) = \mathrm{ker}(\partial)\oplus \mathrm{im}(\partial^*) = H_1(G;\mathbb{R})\oplus \mathrm{im}(\partial^*) .
$$
The group of principal divisors is defined as
$$
\pdiv(G) = \partial(\mathrm{im}(\partial^*)_{\mathbb{Z}}) .
$$

A \emph{discrete 1-form} on $G$ is an element in the real vector space spanned by the formal basis $\{\diff{e}:e\in E(G)\}$. A discrete 1-form $\omega=\sum\omega_e\diff{e}$ is \emph{harmonic} if
$$
\sum_{\substack{e\in E(G)\\
e_+=v}}\omega_e = \sum_{\substack{e\in E(G)\\
e_-=v}}\omega_e \, ,
$$
for all $v\in V(G)$. Let $\Omega(G)$ be the space of discrete harmonic 1-forms. Define the integration of the basic 1-form $\diff{e}$ by
$$
\int_{e'}\diff{e} = \begin{cases}
    \ell(e) &\text{ if }e=e' ,\\
    0 &\text{ if }e\neq e' .
\end{cases}
$$
The integration of discrete harmonic 1-forms yields a surjective homomorphism
$$
\begin{aligned}
    C_1(G;\mathbb{R})&\to \Omega^*(G)\\
    \alpha&\mapsto \int_{\alpha} 
\end{aligned}
$$
which becomes an isomorphism when restricted to $H_1(G;\mathbb{R})$. Instead of taking $\Omega^*(G)$ as the candidate covering space of Jacobian, consider the following subgroup
$$
\Omega^{\#}(G) = \left\{\int_{\alpha}\in \Omega^*(G): \alpha\in C_1(G;\mathbb{Z})\right\} .
$$
The \emph{Jacobian group} of $G$ is defined as the quotient group $\jac(G) = \Omega^{\#}(G)/H_1(G;\mathbb{Z})$. By construction, the Jacobian group of $G$ is a finitely generated abelian group. 

Note that $\Div^0(G) = \mathrm{im}(\partial)_{\mathbb{Z}}$, thus for any $D\in \Div^0(G)$, there exists $\alpha\in C_1(G;\mathbb{Z})$ such that $\partial\alpha=D$. The \emph{discrete Abel--Jacobi map} is defined by
$$
\begin{aligned}
    \mathcal{J}: \Div^0(G)&\to \jac(G)\\
    D=\partial\alpha&\mapsto \int_\alpha \qquad\left(\text{mod } H_1(G;\mathbb{Z})\right) .
\end{aligned}
$$
Under the above settings, we have the discrete version of Abel--Jacobi theorem.
\begin{theorem}{\citep[Theorem 2.8]{baker2011metric}}
    Let $G$ be a combinatorial graph. The discrete Abel--Jacobi map induces a canonical isomorphism
    $$
    \Div^0(G)/\pdiv(G)\xrightarrow{\sim} \jac(G) .
    $$
\end{theorem}

\paragraph{The Abel--Jacobi Theory for Metric Graphs.}

Let $\Gamma$ be a metric graph. Similar to previous constructions, we define a \emph{divisor} $D$ on $\Gamma$ to be an integral linear combination of points on $\Gamma$ and denote the divisor group by $\Div(\Gamma)$. For any nonzero piecewise linear function $f$ on $\Gamma$, define the following divisor
$$
(f) = \sum_{p\in\Gamma}\Delta f(p)p ,
$$
where $\Delta f$ is the Laplacian of $f$ given by \cref{eq:lapla-metric-graph}. Let $\mathscr{A}(\Gamma)$ be the set of all piecewise linear functions on $\Gamma$. The subgroup of \emph{principal divisors} is defined as
$$
\pdiv(\Gamma) = \{(f):f\in\mathscr{A}(\Gamma)\} .
$$
In \Cref{sec:trop-harmonic-form}, we have defined tropical harmonic 1-forms on $\Gamma$. We extend the tropical Abel--Jacobi map in \Cref{sec:trop-aj} to the group of divisors of degree zero $\Div^0(\Gamma)$ by
$$
\begin{aligned}
   \mathcal{J}: \Div^0(\Gamma)&\to \jac(\Gamma)\\
   D&\mapsto \sum_i\int_{q_i}^{p_i} \quad\left(\text{mod } H_1(\Gamma;\mathbb{Z})\right) ,
\end{aligned}
$$
where $D=\sum (p_i-q_i)$ for $p_i,q_i\in\Gamma$. We have the following tropical Abel--Jacobi theorem.
\begin{theorem}{\cite[Theorem 6.2]{mikhalkin2008tropical}}
    Let $\Gamma$ be a metric graph. The tropical Abel--Jacobi map induces a canonical isomorphism 
    $$
    \Div^0(\Gamma)/\pdiv(\Gamma)\xrightarrow{\sim}\jac(\Gamma) .
    $$
\end{theorem}

Fix a combinatorial model $G_0$ for $\Gamma$. Let $R(G_0)$ be the set of all combinatorial models that admit a common refinement with $G_0$. Then $R(G_0)$ is a directed set with respect to refinements. It can be shown that 
$$
\varinjlim_{G\in R(G_0)}\jac(G)\cong \jac(\Gamma),\quad \varinjlim_{G\in R(G_0)}\frac{\Div^0(G)}{\pdiv(G)}\cong\frac{\Div^0(\Gamma)}{\pdiv(\Gamma)} ,
$$
and the discrete Abel--Jacobi map commutes with respect to refinements \citep{baker2011metric}. In other words, the tropical Abel--Jacobi map is the directed limit of the discrete Abel--Jacobi map.

\subsection{Complex and Tropical Abelian Varieties}\label{app:trop-abel-var}

The purpose of this section is to introduce complex and tropical abelian varieties to non-experts, and more importantly, to elaborate on the motivations of many constructions in tropical abelian varieties based on its formal similarity to complex abelian varieties. More details in this section can be found in  \cite{lange2023abelian,mikhalkin2008tropical,kontsevich2006affine,gross2023tautological}.

\paragraph{Complex Abelian Varieties.}

Let $V$ be a complex vector space of dimension $g$ and $\mathbb{L}\subseteq V$ be a lattice of rank $2g$. The quotient $X = V/\mathbb{L}$ is a complex manifold of dimension $g$, known as a complex torus.  Let $\mathbb{Z}$, $\mathcal{O}_X$ and $\mathcal{O}_X^*$ denote the constant sheaf, sheaf of holomorphic functions on $X$, and sheaf of non-vanishing holomorphic functions on $X$, respectively. The \emph{exponential sheaf sequence} is 
$$
\begin{tikzcd}
	0 & \mathbb{Z} & \mathcal{O}_X & \mathcal{O}_X^* & 0
	\arrow[from=1-1, to=1-2]
	\arrow[from=1-2, to=1-3]
	\arrow["{e^{2\pi i\cdot}}", from=1-3, to=1-4]
	\arrow[from=1-4, to=1-5]
\end{tikzcd}.
$$
Its long exact sequence of sheaf cohomology is
$$
\begin{tikzcd}
	\cdots &  H^1(X;\mathbb{Z}) & H^1(X;\mathcal{O}_X) & H^1(X;\mathcal{O}_X^*) & H^2(X;\mathbb{Z}) & \cdots
	\arrow[from=1-1, to=1-2]
	\arrow[from=1-2, to=1-3]
        \arrow[from=1-3, to=1-4]
	\arrow["{c_1}", from=1-4, to=1-5]
	\arrow[from=1-5, to=1-6]
\end{tikzcd}
$$
The homomorphism $c_1:H^1(X;\mathcal{O}_X^*)\to H^2(X;\mathbb{Z})$ is called \emph{the first Chern class map}, which sends a \emph{holomorphic line bundle class} to a real-valued \emph{alternating form} on $\mathbb{L}$. Let $L$ be a line bundle on $X$ and let $E=c_1([L])$. Consider the \emph{hermitian form} defined by 
\begin{equation}\label{eq:hermitian}
  H(x,y) = E(ix,y)+iE(x,y),\, \forall\ x,y\in V  .
\end{equation}
The line bundle $L$ is called \emph{positive} if the hermitian form $H$ is positive definite. In this case the line bundle, or the hermitian form, is called a \emph{polarization} on $X$.

If a complex torus $X$ admits a polarization, then there exists a holomorphic embedding of $X$ to some projective space, which makes $X$ an algebraic variety. An \emph{abelian variety} is by definition a complex torus $X$ admitting a polarization. In this case, the pair $(X,L)$ (or $(X,H)$) is called a \emph{polarized abelian variety}. Given a real-valued alternating form $E$ on $\mathbb{L}$, one way to verify whether $E$ induces a polarization on $X$ is via the \emph{Riemann bilinear relations}: choose a basis $\tau_1,\ldots,\tau_g,\mu_1,\ldots,\mu_g$ of $\mathbb{L}$ such that $E$ is represented by the block matrix 
$$
\begin{bmatrix}
    \bm{0} & \mathbf{D}\\
    -\mathbf{D} & \bm{0}
\end{bmatrix} ,
$$
where $\mathbf{D} = \diag\{d_1,\ldots,d_g\}$ is a diagonal matrix and $d_i\in\mathbb{N}$, $d_1\mid\ldots\mid d_g$. Choose a (complex) basis $e_1,\ldots,e_g$ of $V$. Let $\mathbf{Z}_1$ and $\mathbf{Z}_2$ be complex matrices such that
$$
    \tau_i = \sum_{j=1}^g \mathbf{Z}_1[j,i]e_j,\quad
    \mu_i = \sum_{j=1}^g \mathbf{Z}_2[j,i]e_j \, .
$$
The matrix $\mathbf{Z} = [\mathbf{Z}_1,\mathbf{Z}_2]$ is called the \emph{period matrix} of $X$ with respect to the bases $\{\tau_i\}$, $\{\mu_i\}$ and $\{e_i\}$. The Riemann bilinear relations are 
\begin{itemize}
    \item $\mathbf{Z}_2\matD^{-1}\mathbf{Z}_1^\top = \mathbf{Z}_1\matD^{-1}\mathbf{Z}_2^\top$;
    \item $i(\mathbf{Z}_2\matD^{-1}\overline{\mathbf{Z}}_1^\top-\mathbf{Z}_1\matD^{-1}\overline{\mathbf{Z}}_2^\top)$ is positive definite.
\end{itemize}
Thus the alternating form $E$ induces a polarization $H$ on $X$ if and only if the Riemann bilinear relations are satisfied. The tuple $(d_1,\ldots,d_g)$ is called the type of the polarization. If $d_1=\ldots=d_g=1$, then $H$ is called a \emph{principal polarization}, and the pair $(X,H)$ is called a \emph{principally polarized abelian variety}.

Let $\mathscr{C}$ be a compact Riemann surface of genus $g$. The Jacobian variety of $\mathscr{C}$ is defined as the quotient $\jac(\mathscr{C})=\Omega^*(\mathscr{C})/H_1(\mathscr{C};\mathbb{Z})$ where $\Omega(\mathscr{C})$ is the space of holomorphic differential 1-forms on $\mathscr{C}$. The Jacobian of $\mathscr{C}$ is a $g$-dimensional complex torus. Moreover, there is a canonical polarization on $\jac(\mathscr{C})$ which turns it into a principally polarized abelian variety: Fix a homology basis $\tau_1,\ldots,\tau_{2g}$ of $H_1(\mathscr{C};\mathbb{Z})$ with intersection matrix 
$$
\mathbf{S} = \begin{bmatrix}
    \bm{0} & -\mathbf{I}_g\\
    \mathbf{I}_g & \bm{0}
\end{bmatrix} .
$$
Let $E$ be the real-valued alternating form on $\Omega^*(\mathscr{C})$ represented by the matrix $\mathbf{S}^{-1}$ with respect to the basis $\int_{\tau_1},\ldots,\int_{\tau_{2g}}$. Choose a basis of holomorphic differential 1-forms $\omega_1,\ldots,\omega_g$ of $\Omega(\mathscr{C})$. Then the period matrix under $\{\int_{\tau_i}\}$ and $\{\omega_i^*\}$ satisfies the Riemann bilinear relations. Therefore the hermitian form $H:\Omega^*(\mathscr{C})\times \Omega^*(\mathscr{C})\to\mathbb{C}$ constructed by \cref{eq:hermitian} defines the principal polarization on $\jac(\mathscr{C})$.

\paragraph{Tropical Abelian Varieties.}

% \paragraph{Tropical Tori}

% In this subsection we introduce the concept of tropical torus. Since this is the only object we are concerned with, we will follow the route of \cite{kontsevich2006affine} to define it as an affine manifold. We show that the construction of a tropical torus is determined by two lattices in a vector space. 

Recall that a function $f:\mathbb{R}^g\to\mathbb{R}$ is integral affine if it is in the form $f(x) = a_1x_1+\cdots+a_gx_g+b$, where $a_1,\ldots,a_g\in\mathbb{Z}$ and $b\in\mathbb{R}$. Let $\mathcal{AF}_{\mathbb{R}^g}$ be the sheaf of integral affine functions on $\mathbb{R}^g$. 

Let $X$ be a topological manifold of dimension $g$. A integral affine structure on $X$ is a subsheaf $\mathcal{AF}_X$ of the sheaf of continuous functions on $X$ such that $(X,\mathcal{AF}_X)$ is locally isomorphic to $(\mathbb{R}^g,\mathcal{AF}_{\mathbb{R}^g})$. Equivalently, the integral affine structure can be defined as an atlas of coordinate charts $\{(U_i,\phi_i)\}$ such that the transition maps $\phi_i\circ\phi_j^{-1}$ are integral affine functions. A topological manifold together with an integral affine structure is called an \emph{affine manifold}. 

Let $\mathbb{L},\mathbb{M}\subseteq V$ be two full rank lattices of a $g$-dimensional real vector space $V$. Consider $V$ as the vector space $\mathbb{M}_{\mathbb{R}}= \mathbb{M}\otimes_\mathbb{Z}\mathbb{R}$. The real torus $X=V/\mathbb{L}$ has a natural integral affine structure induced from $\mathbb{M}$: let $\pi:V\to X$ be the quotient map. For any open set $U\subseteq X$, a function $f:U\to\mathbb{R}$ is in $\mathcal{AF}_X(U)$ if and only if $f\circ \pi\in\mathcal{AF}_{\mathbb{M}_\mathbb{R}}(\pi^{-1}(U))$. By definition $X$ with the integral affine structure is an affine manifold. The pair $(X,\mathbb{M})$ is called a \emph{tropical torus}, and $\mathbb{M}$ is conventionally called the \emph{tropical structure} of $X$.

Given a tropical torus $(X,\mathbb{M})$, let $\mathcal{T}_{X}$ be the sheaf of integral differential 1-forms. We have the following short exact sequence of sheaves
$$
\begin{tikzcd}
	0 & \mathbb{R} & \mathcal{AF}_X & \mathcal{T}_X & 0
	\arrow[from=1-1, to=1-2]
	\arrow[from=1-2, to=1-3]
	\arrow["{\mathrm{d}}", from=1-3, to=1-4]
	\arrow[from=1-4, to=1-5]
\end{tikzcd}
$$
where $\mathbb{R}$ denotes the constant sheaf on $X$ and $\mathrm{d}$ is the exterior differential operator. The corresponding long exact sequence is 
\begin{equation}\label{eq:long-exact}
\begin{tikzcd}
  0\rar& H^0(X,\mathbb{R}) \rar & H^0(X,\mathcal{AF}_X) \rar
             \ar[draw=none]{d}[name=X, anchor=center]{}
    & H^0(X,\mathcal{T}_X) \ar[rounded corners,
            to path={ -- ([xshift=2ex]\tikztostart.east)
                      |- (X.center) \tikztonodes
                      -| ([xshift=-2ex]\tikztotarget.west)
                      -- (\tikztotarget)}]{dll}[at end]{\delta^0} \\      
  &H^1(X,\mathbb{R}) \rar & H^1(X,\mathcal{AF}_X) \rar{c_1}\ar[draw=none]{d}[name=Y, anchor=center]{} & H^1(X,\mathcal{T}_X) \ar[rounded corners,
            to path={ -- ([xshift=2ex]\tikztostart.east)
                      |- (Y.center) \tikztonodes
                      -| ([xshift=-2ex]\tikztotarget.west)
                      -- (\tikztotarget)}]{dll}[at end]{\delta^1} \\
  &H^2(X,\mathbb{R}) \rar & \cdots \phantom{H1(Aff)} 
\end{tikzcd}
\end{equation}
We unwrap the definitions and decipher the long exact sequence as follows: since there are no non-constant affine functions on the torus, the map $H^0(X,\mathbb{R}) \to H^0(X,\mathcal{AF}_X)$ is an isomorphism. Thus the connection homomophism $\delta^0:H^0(X,\mathcal{T}_X)\to H^1(X,\mathbb{R})$ is a monomorphism. For $H^1(X,\mathbb{R})$, we have isomorphisms 
$$
H^1(X,\mathbb{R})\cong \homz(H_1(X,\mathbb{Z}),\mathbb{R}) \cong \homz(\mathbb{L},\mathbb{R}) .
$$
The sheaf $\mathcal{T}_X$ is isomorphic to the constant sheaf $\mathbb{M}$ \citep[Example 2.10]{gross2023tautological}. Thus we have $H^0(X,\mathcal{T}_X)\cong \mathbb{M}$, and we can identify the connection homomophism $\delta^0$ as an embedding from $\mathbb{M}$ to $\homz(\mathbb{L},\mathbb{R})$.

%note that the sheaf $\mathcal{T}_X$ is isomorphic to the constant sheaf $\mathbb{\mathbb{L}}^*_X$. This is because the fundamental group $\pi_1(X,x)$ acts on the fiber by translation, hence the monodromy $\pi_1(X,x)\to GL(T^*_x(X))$ is trivial

%where we use the fact that the first homology group of $X=V/\mathbb{L}'$ is isomorphic to $\mathbb{L}'$. Hence we can identify the connection homomophism $\delta^0$ as an embedding from $\mathbb{L}^*=\text{Hom}(\mathbb{L},\mathbb{Z})\to \text{Hom}(\mathbb{L}',\mathbb{R})$.

A \emph{tropical line bundle} $L$ on $X$ is a (real) vector bundle of rank one on $X$ such that any two trivializations are related via the translation
by an integral affine function. Elements in the group $H^1(X,\mathcal{AF}_X)$ are identified as tropical line bundle classes. For group $H^1(X,\mathcal{T}_X)$, we have the following isomorphisms
\begin{equation}\label{eq:line-bundle}
  H^1(X,\mathcal{T}_X)\cong H^1(X,\mathbb{M})\cong \homz(H_1(X,\mathbb{Z}),\mathbb{M})\cong \homz(\mathbb{L},\mathbb{M})  .
\end{equation}
The homomorphism $c_1:H^1(X,\mathcal{AF}_X)\to H^1(X,\mathcal{T}_X)$
is called the \emph{tropical first Chern class map}, which sends a tropical line bundle class to a lattice homomorphism from $\mathbb{L}$ to $\mathbb{M}$. Composing with the embedding $\delta^0:\mathbb{M}\to\homz(\mathbb{L},\mathbb{R})$, we denote image of the tropical first Chern class map by the bilinear form $c_1([L]):\mathbb{L}\times \mathbb{L}\to\mathbb{R}$. 

% connected to line bundles on affine mainfolds.
% \begin{definition}
%     Let $X$ be an affine manifold. A line bundle $L$ on $X$ is a vector bundle of rank 1 such that any two trivializations differ by addition of a $\mathbb{Z}$-affine function. Two line bundles $L_1$ and $L_2$ on $X$ are isomorphic if there is a continuous map $L_1\to L_2$ which respects the projections and local trivializations up to $\mathbb{Z}$-affine functions.  
% \end{definition}
% For tropical torus $X$ the sheaf cohomology group $H^1(X,\mathcal{AF}_X)$ is isomorphic to the C\v{e}ch cohomology group, where each element corresponds to an isomorphism class of line bundles. 

%The bilinear form on the lattice extends to be a well-defined bilinear form on the vector space $V\cong \mathbb{L}'\otimes\mathbb{R}$.  

Elements in $H^2(X,\mathbb{R})$ are alternating forms on $\mathbb{L}$. Thus the connection homomorphism $\delta^1:H^1(X,\mathcal{T}_X)\to H^2(X,\mathbb{R})$ is the restriction of a bilinear form to its skew-symmetric component. By exactness, the image of $c_1$ is the group of all symmetric bilinear forms on $\mathbb{L}$. A tropical line bundle $L$ is \emph{positive} if $c_1([L])$ is positive definite. A tropical torus $X$ with a positive tropical line bundle $L$ is called a \emph{polarized tropical abelian variety}. The tropical line bundle $L$, or the  symmetric bilinear form $Q=c_1([L])$, is called a \emph{tropical polarization}. By the identification \cref{eq:line-bundle}, A polarization corresponds to a lattice isomorphism $\mathbb{L}\to\mathbb{M}$. The index of the image of $\mathbb{L}$ in $\mathbb{M}$ is called the index of the polarization. If the index is one, the polarization $L$ (or $Q$) is called a \emph{principal polarization}, and the pair $(X,L)$ (or $(X,Q)$) is called a \emph{principally polarized tropical abelian variety}. 

% For group $H^2(X,\mathbb{R}_X)$, the ring structure on cohomology implies that 
% $$
% H^2(X,\mathbb{R}_X) \cong \bigwedge^2 H^1(X,\mathbb{R})\cong \bigwedge^2 \text{Hom}(\mathbb{L}',\mathbb{R})
% $$
% Thus for each element in $H^2(X,\mathbb{R}_X)$, it can be identified with a skew-symmetric bilinear form $\mathbb{L}'\times \mathbb{L}'\to\mathbb{R}$. The connection homomorphism $\delta^1:H^1(X,\mathbb{T}^*_\mathbb{Z})\to H^2(X,\mathbb{R}_X)$ is the restrction of a bilinear form to its skew-symmetric part.

Let $\Gamma$ be a metric graph. Its tropical Jacobian is a tropical torus $\jac(\Gamma) = (\Omega^*(\Gamma)/H_1(\Gamma;\mathbb{Z}), \Omega_{\mathbb{Z}}^*(\Gamma))$ (\Cref{def:trop-jac}), where $\Omega_{\mathbb{Z}}^*(\Gamma)$ represents its tropical structure. Let $Q_\Gamma$ be the bilinear form given by \cref{eq:trop-polar}. The following theorem shows that $Q_\Gamma$ is indeed a principal polarization which turns the tropical Jacobian into a principally polarized tropical abelian variety, thus explaining its name \emph{tropical polarization} (\Cref{def:trop-polar}). 

\begin{theorem}\label{thm:principal-polar}
    Let $\Gamma$ be a metric graph. The bilinear form $Q_\Gamma$ defined in \cref{eq:trop-polar} is a principal polarization on the tropical Jacobian $\jac(\Gamma)$.
\end{theorem}
\begin{proof}
    By definition $Q_\Gamma$ is symmetric positive definite on $\Omega^*(\Gamma)$. It suffices to show that $Q_\Gamma$ restricted to $H_1(\Gamma;\mathbb{Z})\times\Omega^*_{\mathbb{Z}}(\Gamma)$ is integral, and there exists bases of $H_1(\Gamma;\mathbb{Z})$ and $\Omega^*_{\mathbb{Z}}(\Gamma)$ such that $Q_\Gamma$ is represented by the identity matrix.

    Fix a combinatorial model $G$ for $\Gamma$, and a basis $\sigma_1,\ldots,\sigma_g$ of $H_1(G;\mathbb{Z})$. Write $\sigma_i=\sum_j a_{ij}e_j$ and let $\omega_i=\sum_j a_{ij}\diff{t_{e_j}}$. By \Cref{thm:oneform-char}, $\omega_1,\ldots,\omega_g$ is a basis of $\Omega_\mathbb{Z}(\Gamma)$, and $\omega^*_1,\ldots,\omega^*_g$ is the dual basis of $\Omega^*_\mathbb{Z}(\Gamma)$. For each $\omega_i^*$ there exists $\eta_i\in H_1(G;\mathbb{R})$ such that $\omega_i^*=\int_{\eta_i}$. Then we can check that
    $$
    Q_\Gamma\left(\int_{\sigma_i\,},\,\omega^*_j\right) = \langle \sigma_i,\eta_j \rangle = \int_{\eta_i}\omega_i = \delta_{ij} \, ,
    $$
    which proves the assertion.
\end{proof}

\end{appendices}

\end{document}